\documentclass{article}
\usepackage{graphicx, amsthm,amssymb,amsmath,amsfonts,tikz,mathrsfs}
\usepackage[colorinlistoftodos]{todonotes}
\usepackage{stackrel}
\usepackage{biblatex}
 \usetikzlibrary{positioning}
\usepackage{tikz-cd}
\usepackage{lipsum}
\usepackage{adjustbox}
\usepackage{enumitem}
\usepackage{float}

\newcommand{\qv}{\mathcal{M}}
\newcommand{\dv}{\mathsf{v}}
\newcommand{\dw}{\mathsf{w}}

\newcommand{\bT}{\mathsf{T}}
\newcommand{\stackqv}{\mathfrak{X}}
\newcommand{\ver}[1]{V^{(#1)}}
\newcommand{\signver}[1]{\mathsf{V}^{(#1)}}
\newcommand{\nver}[1]{\widetilde{V}^{(#1)}}
\newcommand{\cpver}[1]{\widehat{V}^{(#1)}}
\newcommand{\cp}{\Psi}
\newcommand{\fundwt}{\omega}
\newcommand{\qm}{\mathsf{QM}}
\newcommand{\ns}{\text{ns} \,}
\newcommand{\rel}{\text{rel} \,}
\newcommand{\tb}{\mathcal{V}}
\newcommand{\tbw}{\mathcal{W}}
\newcommand{\qmtb}{\mathscr{V}}
\newcommand{\qmtbw}{\mathscr{W}}
\newcommand{\qmpol}{\mathcal{T}^{1/2}}
\newcommand{\vrs}{[\mathsf{QM}^{d}]^{\text{vir}}}
\newcommand{\qmtan}{T_{\text{vir}}}
\newcommand{\qcup}{\star}
\newcommand{\ggamma}{\boldsymbol{\gamma}}
\newcommand{\diffop}{\mathcal{D}}

\newcommand{\ev}{\operatorname{ev}}

\newcommand{\eff}{\operatorname{Eff}}

\DeclareMathOperator{\Ima}{Im}
\DeclareMathOperator{\rpp}{rpp}
\newcommand{\loc}{\text{loc}}

\newcommand{\h}{\mathfrak h}

\newcommand{\Z}{\mathbb Z}

\newcommand{\la}{\lambda}

\newcommand{\GZ}{\mathscr{H}}

\theoremstyle{definition}

\newtheorem{theorem}{Theorem}[section]
\newtheorem{lemma}[theorem]{Lemma}
\newtheorem{proposition}[theorem]{Proposition}
\newtheorem{definition}[theorem]{Definition}
\newtheorem{remark}[theorem]{Remark}
\newtheorem{corollary}[theorem]{Corollary}
\newtheorem{conjecture}[theorem]{Conjecture}
\newtheorem{example}[theorem]{Example}
\newtheorem{assumption}[theorem]{Assumption}
\newtheorem{warning}[theorem]{Warning}
\newtheorem{maintheorem}{Theorem}

\newif\ifShowLabels
\ShowLabelsfalse

\newcommand{\on}{\operatorname}

\newcommand{\BC}{{\mathbb{C}}}

\newcommand{\sF}{{\mathsf{F}}}

\newcommand{\sff}{{\mathsf{f}}}

\newcommand{\ol}{\overline}

\newcommand{\brho}{{\boldsymbol{\rho}}}

\newcommand{\Hom}{{\mathop{\operatorname{Hom}}}}

\newcommand{\ul}{\underline}

\newcommand{\La}{\Lambda}

\newcommand{\tGh}{\widetilde{G}_{\dv}}

 \usepackage{hyperref}

\newcommand\iso{\,\vphantom{j^{X^2}}\smash{\overset{\sim}{\vphantom{\rule{0pt}{0.20em}}\smash{\longrightarrow}}}\,}

\title{The quantum Hikita conjecture via quasimaps}
\author{Hunter Dinkins, Ivan Karpov, and Vasily Krylov}
\date{}

\begin{document}

\maketitle

\begin{abstract}
  We propose a refinement of the quantum Hikita conjecture of Kamnitzer, McBreen, and Proudfoot that bridges the representation theory of Coulomb branches with the enumerative geometry of Higgs branches. We also introduce a general framework for proving it, which we carry out for ADE quiver gauge theories with minuscule framings and for the gauge theory corresponding to the Jordan quiver. As an application, we use the resulting quantum Hikita isomorphisms to give a geometric description of graded traces on quantized Coulomb branches.
\end{abstract}

\setcounter{tocdepth}{2}
\tableofcontents

\section{Introduction}\label{sec:introduction}

In this paper, we will propose and, in some cases, prove a correspondence between the enumerative geometry of Higgs branches of quiver gauge theories and the representation theory of the associated quantized Coulomb branches. Our conjectures and results are refinements of the so-called quantum Hikita conjecture of Kamnitzer, McBreen, and Proudfoot \cite{KMBP}.

\subsection{Quiver varieties and vertex functions}\label{intro:subsec quiver var and vertex functions}

Fix a finite quiver $\Gamma=(I,E)$ and dimension vectors $\dv,\dw \in \mathbb{N}^{I}$.  We write
$$
X=\widetilde{\mathcal M}_H=\mathcal M(\dv,\dw)
$$
for the corresponding Nakajima quiver variety.  It parametrizes stable
framed representations of the doubled quiver satisfying the moment-map
relation.  In particular, $X$ is a smooth symplectic variety, usually
non-compact, equipped with tautological bundles $\mathcal V_i$, $i\in I$.
It also carries an action of the flavor torus $\mathsf F$ and of
$\mathbb C^\times_\hbar$, which scales the symplectic form. We denote by
$\hbar$ the inverse of the standard character of $\mathbb C^\times_\hbar$
and set $\mathsf T=\mathsf F\times\mathbb C^\times_\hbar$. Thus, the
Calabi--Yau specialization used below is $q=2\hbar$. We refer to
\cite{GinzburgLectures,NakALE,NakQv} and Section \ref{ssec:Nakajima quiver var defin} for this construction.

The curve-counting theory used in the paper is the theory of quasimaps to
$X$, see \cite{qm}, \cite[Section 4.3]{OkLec}, and Section \ref{quasimaps}.  A quasimap from $\mathbb P^1$ is allowed to land in the unstable locus at
finitely many points; equivalently, it is a map to the stacky quotient
$\mathfrak X$ which generically lands in $X$.  Virtual integration over the
spaces of such maps gives, for every descendant
$\tau\in H^*_{\mathsf T_q}(\mathfrak X)$, a series
$$
V^{(\tau)}=\sum_d z^d V_d^{(\tau)},
\qquad \mathsf T_q=\mathsf F\times\mathbb C^\times_\hbar \times \mathbb{C}^\times_q,
$$
called the \emph{vertex function}.  Its coefficients are localized equivariant
cohomology classes on $X$ obtained by pushing forward virtual classes of
quasimap spaces.

The vertex $V^{(1)}$ is the analogue of Givental's $I$-function
\cite{Oko15,OkLec}.  Such functions are useful to symplectic geometers because
they package genus-zero curve counts into an explicit series, often computable
by localization, and provide distinguished solutions of the quantum
differential equation.  In the present paper the vertex has a further role:
at the Calabi--Yau specialization $q=2\hbar$, the same series becomes a
character of a module on the $3d$ mirror dual side.

Let us briefly recall the quantum differential equation.  The Pushkar--Smirnov--Zeitlin (or simply PSZ) quantum $D$-module is
the equivariant cohomology of $X$, spread over the formal K\"ahler torus and
equipped with the connection
$$
\nabla_i^{\mathrm{PSZ}}
=qz_i\frac{\partial}{\partial z_i}+M_i(z),
$$
where $M_i(z)$ is quantum multiplication by the quantum tautological class
associated with $c_1(\mathcal V_i)$, and $q$ is the parameter of $\mathbb{C}^{\times}_{q}$ rotating the
source $\mathbb P^1$. A basic result of quasimap theory says that, for every
fixed point $p\in X^{\mathsf T}$, with inclusion $\iota_p\colon p\hookrightarrow X$, the assignment
$\tau\mapsto\iota_p^*V^{(\tau)}$ factors through $Q^{\mathrm{PSZ}}$ and
defines a morphism from the PSZ $D$-module to the rank-one $D$-module
$H^{*,  \on{loc.}}_{\mathsf T_q}(p)[[{\boldsymbol{z}}]]$ endowed with its classical
$D$-module structure. We call such a morphism a solution; see
\cite{KoroteevPushkarSmirnovZeitlin2021QuantumK,OkLec,PSZ} and
Theorem~\ref{thm: Dmodhomom}.

\begin{remark}
Thus, ``solution'' refers here to a morphism of $D$-modules, rather than to
a scalar-valued function annihilated by each operator
$\nabla_i^{\mathrm{PSZ}}$. The $D$-module structure on the target is
nontrivial: $D$-linearity expresses the compatibility between the PSZ
action on the source and the classical action on the target.
\end{remark}

The {\emph{capped}} vertex is a surjective deformation of the usual restriction from the quotient
stack $\mathfrak X$ to its stable locus $X$: 
\begin{equation}\label{eq:intro_capped}
H^*_{\mathsf T_q}(\mathfrak X)[[{\boldsymbol{z}}]] \ni \tau\longmapsto \widehat{V}^{(\tau)} \in H^*_{\mathsf{T}_q}(X)[[{\boldsymbol{z}}]].
\end{equation}
Its $q=0$ specialization is the  \emph{quantum Kirwan map}.  
In practice, it allows one to
work with descendants in the explicit $D$-module
$H^*_{\mathsf T_q}(\mathfrak X)[[{\boldsymbol{z}}]]$ and then pass to the quantum $D$-module.  Compare
\cite{kirv,xu} and Section~\ref{sec: vertex}.

\subsection{Quantized Coulomb branches and graded traces}\label{ssec:intro Coulomb twisted and graded}

There is a second space naturally associated with the same quiver data: the
Coulomb branch $\mathcal M_C$ of Braverman, Finkelberg, and Nakajima \cite{BFN1, Nakajima2016CoulombI}.  From the
viewpoint of $3d$ mirror symmetry, $X=\widetilde{\mathcal{M}}_H$ is the resolved Higgs branch and
$\mathcal M_C$ is its {\emph{symplectic dual}}, see \cite{BLPW}.  This point of view predicts an
exchange of equivariant and K\"ahler data.  The flavor parameters for the
$\mathsf F$-action on $\widetilde{\mathcal{M}}_H$ become deformation parameters on the Coulomb side,
while the K\"ahler torus of $\widetilde{\mathcal{M}}_H$, determined by $\operatorname{Pic}(\mathfrak{X})$,
corresponds to a torus $\mathsf A$ acting on $\mathcal M_C$.

We denote by $\mathcal A$ the Braverman--Finkelberg--Nakajima quantization
of the coordinate ring $\mathbb C[\mathcal M_C]$; see Section
\ref{sec:overview} for its construction.

Fix $t\in\mathsf A$ and regard it as an automorphism of $\mathcal A$. A {\emph{$t$-twisted trace}} is a functional
\begin{equation*}
\operatorname{Tr}_t\colon \mathcal{A} \rightarrow \mathbb{C}
\end{equation*}
such that
\begin{equation}\label{eq:def relation for twisted traces intro}
\on{Tr}_t(a \cdot b) = \on{Tr}_t(b \cdot t(a))\qquad\text{for all }a,b \in \mathcal{A}.
\end{equation}

This notion was first introduced and studied by Etingof and Stryker in \cite{etingof-stryker}. Following an idea of Kontsevich, they used twisted traces to parametrize so-called
nondegenerate 
{\emph{short star-products}}  introduced by Beem, Peelaers and Rastelli \cite{BPR} and appearing naturally in $3d$ $\mathcal{N}=4$ superconformal field theories.

Twisted traces on quantized Higgs and Coulomb branches of $3d$ $\mathcal{N}=4$ supersymmetric gauge theories have also recently attracted the attention of Gaiotto, Okazaki and collaborators, see \cite{bullimore-crew-zhang, gaiotto-sphere-quantization, GO, GHRYWZ}. 
For them, twisted traces provide the correct mathematical language to package the {\emph{correlation functions}} of the corresponding theories.

The same structure naturally arises when one studies the representation theory of the algebra $\mathcal{A}$. Namely, consider a suitable $\mathsf{a}$-weight module $M=\bigoplus_\eta M_\eta$ over
$\mathcal A$.  Its character, and more generally the trace
of an $\mathsf{A}$-invariant element $a \in \mathcal{A}^{\boldsymbol{0}}$, may be recorded as a functional
\begin{equation}\label{eq:trace_module_formula_intro}
\mathcal{A}^{\boldsymbol{0}} \rightarrow \mathbb{C}[[{\boldsymbol{z}}]],\quad a \mapsto 
\operatorname{tr}_M(a)=\sum_\eta z^\eta
\operatorname{tr}\bigl(a|_{M_\eta}\bigr).
\end{equation}
Here the same variables $z$ which record curve classes on $X$ are, under the
duality above, characters of the torus $\mathsf A$, see \cite[Section 3.6]{KMBP} and Section \ref{sec:graded traces} for the details.

\begin{warning}\label{intro:warning target tr}
We are not careful with specifying the exact target of (\ref{eq:trace_module_formula_intro}), compare with the spaces $\mathcal{Z}_\xi$ and $\mathcal{Z}_{\Xi}$ appearing in Section \ref{ssec: D mod of graded traces}. It may vary depending on the context.
\end{warning}

For a character $\eta\colon \mathsf{A} \rightarrow \mathbb{C}^\times$ let $\mathcal{A}^\eta$ be the corresponding $\mathsf{A}$-weight component.
A twisted trace vanishes on $\mathcal{A}^\eta$ for $\eta \neq 0$ (see Lemma \ref{lem:twisted trace determined by A 0}), so it is determined by its restriction to $\mathcal{A}^{\boldsymbol{0}}$. Whenever the power series (\ref{eq:trace_module_formula_intro}) admits a well-defined evaluation at $z=t \in \mathsf{A}$, for example through a rational function represented by this series, we obtain the actual $t$-twisted trace by specializing $z=t$.

An important class of modules $M$ as above consists of those lying in category $\mathcal{O}$ for $\mathcal{A}$ (\cite[Section 3]{BLPW}, \cite[Section 3.4]{catO_coulomb}, and Section \ref{ssec:cat O for CB}) which may be regarded as a broad generalization of the usual BGG category $\mathcal{O}$.  
In this case, $\operatorname{tr}_M(a)$ is a linear combination of traces of {\emph{Verma modules}}, which by \cite[Theorem 4.9]{etingof-stryker} are, up to normalization, {\emph{rational}} functions of $z$. Thus,  $\operatorname{tr}_M(a)$ admits a well-defined evaluation at sufficiently generic $t\in\mathsf A$, and these evaluations determine a family of twisted traces.

The algebras $\mathcal{A}$ are isomorphic to quotients of so-called {\emph{shifted Yangians}} (\cite[Appendix B]{BFN19}, \cite[Section 4]{cat_O_slices_and_categor}, \cite{BT26,JindalNegut2026LoopNilpotent,MuthiahWeekes2026CoulombZastava}). If $\Gamma$ is of type ADE,  
their categories $\mathcal{O}$ are known to have a rich representation theory and have been  studied extensively from several perspectives, see \cite{BrundanKleshchev2008ShiftedYangians, hernandez_zhang, KLLPW, KTWWY, cat_O_slices_and_categor} 
and references therein. 
Moreover, the graded traces of these modules package the same information as so-called Frenkel--Reshetikhin $q$-characters and are known to determine the class of the module in the Grothendieck group.

In \cite{KMBP} Kamnitzer, McBreen, and Proudfoot proposed an object $\operatorname{GrTr}(\mathcal{A})$ over $\mathbb{C}[[{\boldsymbol{z}}]]$ which
is the universal receptacle for functionals as in (\ref{eq:trace_module_formula_intro}): its defining relations encode the
rule that, when two homogeneous elements are interchanged inside a trace, the
result is multiplied by the corresponding monomial in $z$; this is precisely the ``formal'' version of (\ref{eq:def relation for twisted traces intro}).  Continuous linear
functionals on $\operatorname{GrTr}(\mathcal{A})$ will be called {\emph{graded}} traces. The functional (\ref{eq:trace_module_formula_intro}) is an example of a graded trace.

Thus graded traces arise naturally from several directions: from representation theory of category $\mathcal{O}$, from twisted traces and correlation functions in $3d$ gauge theories, and, at least in the ADE setting, from $q$-characters of shifted Yangian modules. This leads to the basic question motivating the paper:

\medskip
\noindent\textbf{Question.}
\emph{Is it possible to describe graded traces on a quantized Coulomb branch in terms of the geometry of the symplectic-dual Higgs branch?}
\medskip

The refined quantum Hikita conjecture answers  this question.

\subsection{The quantum Hikita conjecture}

Kamnitzer, McBreen and Proudfoot conjectured that, if a theory $(\Gamma,{\mathsf{v}},{\mathsf{w}})$ is ``good'' in physics terminology, then the $D$-module $\operatorname{GrTr}(\mathcal{A})$ controlling graded traces on the quantized Coulomb branch is {\emph{isomorphic}} to the $q=2\hbar$
specialization  of the {\emph{Gromov--Witten}} quantum $D$-module of the resolved  Higgs branch $X$:

\begin{equation*}
Q^{\mathrm{GW}}_{q=2\hbar} \simeq \operatorname{GrTr}(\mathcal{A}).
\end{equation*}

We propose
a refinement in which $Q^{\mathrm{GW}}_{q=2\hbar}$ is replaced by the Pushkar--Smirnov--Zeitlin $D$-module $Q^{\mathrm{PSZ}}_{q=2\hbar}$. We expect this refined conjecture to hold without any goodness assumption. 
Our refinement includes additional structure: both sides are canonically quotients of the same explicit $D$-module. On the geometric side, the quotient map is given by the $q=2\hbar$-specialized capped vertex (\ref{eq:intro_capped}) and this explicit $D$-module is simply the space 
\begin{equation*}
\mathscr{H}_\hbar[[{\boldsymbol{z}}]]=H^*_{\mathsf{T}}(\mathfrak{X})[[{\boldsymbol{z}}]]
\end{equation*}
of all possible descendants $\tau$. The same space is also very natural from the Coulomb branch perspective:  $\mathscr{H}_\hbar \subset \mathcal{A}^{\boldsymbol{0}}$ is a commutative subalgebra of $\mathcal{A}$ whose classical limit determines the integrable system on $\mathcal{M}_C$.

So, we make the following conjecture.
\begin{conjecture}[Refined quantum Hikita]\label{intro:refined quantum Hikita}
There exists a commutative diagram
\begin{equation}\label{eq:intro_or_comm_diagram}
\begin{tikzcd}
    & {\GZ_{\hbar}[[{\boldsymbol{z}}]]} & \\
    {Q_{q=2\hbar}^{\mathrm{PSZ}}} && {\operatorname{GrTr}(\mathcal{A}),}
    \arrow["\tau \mapsto \widehat{V}^{(\tau)}_{q=2\hbar}"',
        two heads, from=1-2, to=2-1]
    \arrow["\tau \mapsto \tau",
        two heads, from=1-2, to=2-3]
    \arrow["\simeq", from=2-1, to=2-3]
\end{tikzcd}    
\end{equation}
which {\emph{uniquely}} determines the desired horizontal isomorphism. 
\end{conjecture}

In particular, solutions on the left-hand side should match those on the right-hand side, considered as functionals on the space of descendants $\tau \in \mathscr{H}_\hbar[[{\boldsymbol{z}}]]$. 

As discussed in Section \ref{intro:subsec quiver var and vertex functions}, solutions of $Q^{\mathrm{PSZ}}$ are given by $\tau \mapsto \iota_p^*V^{(\tau)}$, where $\iota_p\colon p \hookrightarrow X^{\mathsf{T}}$ is a ${\mathsf{T}}$-fixed point, and the resulting restriction $\iota_p^*V^{(\tau)}$ depends only on the connected component $Z_p\subset X^{\mathsf T}$ of the point $p$ by Proposition~\ref{nonlocvertex}.

The usual (non-quantum) Hikita conjecture \cite[Conjecture 8.9]{KTWWY} predicts a bijection between the connected components of $X^{\mathsf{T}}$ and Verma modules over $\mathcal{A}$.

 This, together with the discussions of \cite[Remark 1.10]{HKW23}, \cite{Liu1},  \cite{BDGHK}, leads us to the following conjecture claiming that the isomorphism (\ref{eq:intro_or_comm_diagram}) {\emph{intertwines natural solutions}}: 
\begin{conjecture}[Conjecture \ref{eq:extended_conj_tr_Verma}]\label{conj:intro_verma_is_vertex}
\begin{equation}\label{eq:intro_verma_is_vertex}
\Big(\iota_p^*V^{(\tau)}\Big)_{q=2\hbar} = \widetilde{\operatorname{tr}}_{\Delta(Z_p)}(\tau).  
\end{equation}
Here $\Delta(Z_p)$ is the Verma module corresponding to $Z_p$, and $\widetilde{\operatorname{tr}}_{\Delta(Z_p)}(\tau)$ is its normalized trace (see Section \ref{sec: our approach} for the details).
\end{conjecture}

This is the form of the
trace–vertex correspondence which we expect to be (at least partially) explained by the forthcoming work of Botta and Tamagni \cite{BT26}. Namely, in \cite{BT26}, the authors construct {\emph{some}} modules over the Coulomb branch algebra whose normalized characters are given by  $\iota^*_pV^{(1)}_{q=2\hbar}$. In some cases it is not hard to see that they indeed must coincide with Verma modules (see Remark \ref{rem: BT construction}).

\subsection{Approach}

We are now ready to explain our approach for proving Conjecture \ref{intro:refined quantum Hikita} whenever the points of $X^{\mathsf{F}}$ are isolated. This is an adaptation of \cite{krylov_shykov} to the quantum setting.
Our point of view is that one should not try to compare the two $D$-modules directly. In particular, it is difficult to explicitly compute the kernels of the maps in (\ref{eq:intro_or_comm_diagram}) which we want to identify. Instead of doing so, we first identify solutions (compare with (\ref{eq:intro_verma_is_vertex})) and then conclude that Conjecture \ref{intro:refined quantum Hikita} holds. Let $\mathcal{M}_C^{\mathsf{A}}$ denote the {\emph{schematic}} fixed points of $\mathsf{A} \curvearrowright \mathcal{M}_C$.
\begin{maintheorem}[Theorem \ref{prop: numerical implies Hikita}]\label{mainthm:reduction}
Assume $\widetilde{\mathcal{M}}_H^{\mathsf{F}}$ is {\emph{finite}}. Assume moreover that:
\begin{enumerate}
    \item We have $\operatorname{dim}\mathbb{C}[\mathcal{M}_C^\mathsf{A}] \leqslant |X^{\mathsf{F}}|$.
    \item There exists {\emph{some}} collection of $\mathcal{A}$-modules $M_p$ labeled by $p \in X^{\mathsf{F}}$ such that $\Big(\iota^*_pV^{(\tau)}\Big)_{q=2\hbar} = \widetilde{\operatorname{tr}}_{M_p}(\tau).$
\end{enumerate}
Then Conjecture \ref{intro:refined quantum Hikita} holds.
\end{maintheorem}
\begin{remark}\label{rem: what need for approach}
 The first assumption has nothing to do with quantum cohomology and, for example, holds whenever the {\emph{usual}} Hikita conjecture \cite{H17}
 \begin{equation*}
 \mathbb{C}[\mathcal{M}_C^{\mathsf{A}}] \simeq H^*(X)
 \end{equation*}
 is known, in particular, for ADE and Jordan quivers.
 The second assumption does {\emph{not}} require $M_p$ to be Verma modules, so, once established, the results of \cite{BT26} would provide us with the collection of modules with normalized characters given by $\Big(\iota^*_pV^{(1)}\Big)_{q=2\hbar}$. Authors of \cite{BT26} expect that the normalized graded traces of their modules are indeed given by $\Big(\iota^*_pV^{(\tau)}\Big)_{q=2\hbar}$.
\end{remark}

For quivers of ADE type, we verify these assumptions by studying a certain family of $\mathcal{A}$-modules obtained via the comultiplication from so-called ``chamber'' or ``extremal'' modules (see Section \ref{ssec: modules obtained via comult}). Using the factorization property of specialized vertex functions (Proposition~\ref{prop: limver}), as well as the computations of Section \ref{sec:vertpoint} combined with \cite{min_chamber_mod, krylov-klyuev-minu}, we prove:
\begin{maintheorem}[Section \ref{sec:ADE-Hilbert}]\label{maintheorem: ADE}
The assumptions of Theorem \ref{mainthm:reduction} are satisfied for $\Gamma$ of type ADE with $\mathsf{w}$ supported at minuscule nodes; for example, for an arbitrary type $A$ quiver theory. They are also satisfied for the quiver theories corresponding to the Jordan quiver.
Hence, the refined quantum Hikita conjecture is true for these theories.
\end{maintheorem}

Theorem \ref{maintheorem: ADE} gives a geometric description of graded traces on quantized Coulomb branches. Recall that $\mathcal{A}$ is an algebra over $\mathbb{C}[\mathsf{t}]$, where $\mathsf{t}=\mathsf{f} \oplus \mathbb{C}\hbar$ is the Lie algebra for the torus $\mathsf{T}$. The subalgebra $\mathbb{C}[\mathsf{t}] \subset \mathcal{A}$ is {\emph{central}} and its generators should be thought of as ``Casimirs''. We can then specialize these parameters to numbers. Namely, every cocharacter $\gamma\colon \mathbb{C}^\times \rightarrow \mathsf{F}$ determines an element $f \in \mathsf{f}$ and we denote by $\mathcal{A}_{f}$ the central quotient corresponding to $\hbar=1$ and $f \in \mathsf{f}$.

The representation theory of the algebras $\mathcal{A}_f$ changes substantially as  $f$ crosses the singular hyperplanes. As for the usual BGG category $\mathcal{O}$, if $f$ is generic enough, then the number of Verma modules in $\mathcal{O}(\mathcal{A}_f)$ is equal to $|X^{\mathsf{F}}|$, assuming $X^{\mathsf{F}}$ are isolated. If $f$ is {\emph{singular}}, then the number of Verma modules in $\mathcal{O}(\mathcal{A}_f)$ drops while the rank of $\operatorname{GrTr}(\mathcal{A}_f)$ over $\mathbb{C}[[{\boldsymbol{z}}]]$ remains the same.
So, the Verma traces no longer span all graded traces.

Theorem \ref{maintheorem: ADE} nevertheless allows us to describe graded traces on $\mathcal{A}_f$ for an arbitrary parameter $f$ as above. After specializing to $(f,1) \in \mathsf{t}$ and applying the localization theorem, we obtain the following commutative diagram: 
\begin{equation*}
\begin{tikzcd}
	& {\GZ_{\hbar}[[{\boldsymbol{z}}]]} & \\
	H^*(\widetilde{\mathcal{M}}_H^{\tilde{\gamma}(\mathbb{C}^\times)})[[{\boldsymbol{z}}]] && {\operatorname{GrTr}(\mathcal{A}_f),}
	\arrow["\tau \mapsto \widehat{V}^{(\tau)}_{q=2\hbar}"', from=1-2, to=2-1]
	\arrow["\tau \mapsto \tau", from=1-2, to=2-3]
	\arrow["\simeq", from=2-1, to=2-3]
    \end{tikzcd}    
\end{equation*}
where $\tilde{\gamma}\colon \mathbb{C}^\times \rightarrow \mathsf{F} \times \mathbb{C}^\times_\hbar$ is the cocharacter $t \mapsto (\gamma(t),t)$.

The parameter $f$ is singular precisely if $\widetilde{\mathcal{M}}_H^{\tilde{\gamma}(\mathbb{C}^\times)} \neq \widetilde{\mathcal{M}}_H^{\mathsf{T}}$. Then, graded traces on $\mathcal{A}_f$ are precisely $\mathbb{C}[[{\boldsymbol{z}}]]$-linear functionals on $H^*(\widetilde{\mathcal{M}}_H^{\tilde{\gamma}(\mathbb{C}^\times)})[[{\boldsymbol{z}}]]$. By \cite[Theorem 7.3.5]{Nak_quiver_and_fd_reps}, every such functional is given by the {\emph{integration pairing}} with a class in the homology of the $\tilde{\gamma}$-fixed part of the  {\emph{Lagrangian core}} $\mathfrak{L} \subset \widetilde{\mathcal{M}}_H$; by  definition, $\mathfrak{L}$ is the fiber over zero of the natural morphism $\widetilde{\mathcal{M}}_H \rightarrow \mathcal{M}_H$. It is a projective, generally singular, Lagrangian subvariety.

We obtain the following corollary. 
\begin{corollary}[Corollary \ref{cor: description of graded traces}]\label{intro:corollary descr grtr}
Under the assumptions of Theorem \ref{maintheorem: ADE} or, more generally, whenever Conjecture \ref{intro:refined quantum Hikita} holds, graded traces on $\mathcal{A}_f$ are given by:
\begin{equation}\label{eq:intro_graded_via_capped}
\tau \mapsto \int_{\widetilde{\mathcal{M}}_H^{\tilde{\gamma}(\mathbb{C}^\times)}} (\widehat{V}^{(\tau)}_{q=2\hbar}) \cap a, \quad   a \in H_*(\mathfrak{L}^{\tilde{\gamma}(\mathbb{C}^\times)})[[{\boldsymbol{z}}]].
\end{equation}
In particular, every graded trace has a natural $q$-deformation given by (\ref{eq:intro_graded_via_capped}) without imposing the $q=2\hbar$ specialization.
\end{corollary}

\subsection{Computation of vertex functions}
Corollary \ref{intro:corollary descr grtr} provides a geometric description of graded traces via the {\emph{capped}} vertex. It would still be desirable to have a similar description in terms of the {\emph{bare}} vertex $V^{(\tau)}$, as it is much simpler to compute. A complication is that the $q=2\hbar$ specialization of $V^{(\tau)}$ is {\emph{not}} well-defined in general; see Example \ref{ex:vertex has pole at hbar q} and the discussion at the end of Section \ref{sec:application of our results}.

On the other hand, we prove that the fixed-point restrictions $\iota_pV^{(\tau)}$ {\emph{do}} have a well-defined specialization to $q=2\hbar$ (see Proposition \ref{nonlocvertex})
and thus define {\emph{distinguished}} graded traces 
\begin{equation}\label{eq:intro vertex graded traces}
\tau \mapsto (\iota^*_pV^{(\tau)})_{q=2\hbar}.
\end{equation}
As explained in Warning \ref{war: specializations do not commute}, the functionals (\ref{eq:intro vertex graded traces}) {\emph{depend}} on the choice of the
torus $\mathsf{T}$ we start with. This reflects the fact that Verma modules for singular $f$ {\emph{do not}} have the same normalized character as those for regular $f$.

The upshot is that, assuming Conjecture \ref{conj:intro_verma_is_vertex}, the computation of traces of Verma modules for different choices of parameter $f$ reduces to the computation of $\iota_p^{*} V^{(\tau)}$ for various 
tori $\mathsf{T}$.

For a class of $\mathsf{T}$-fixed points $p \in \widetilde{\mathcal{M}}_H$
arising from dominant-minuscule skew heaps, we compute the specialized vertex function (\ref{eq:intro vertex graded traces}) with arbitrary
descendant as a sum over reverse plane partitions over the corresponding skew heap. See Section \ref{sec:vertpoint} for the details.
\begin{maintheorem}[Theorem \ref{thm: pointvertex}]
In the notations of Section \ref{sec:vertpoint}, if $U$ is a dominant-minuscule skew heap and $p_U \in \widetilde{\mathcal{M}}_H^{\mathsf{T}}$ is the corresponding $\mathsf{T}$-fixed point then:
\begin{equation}\label{intro:rpp formula}
(\iota^*_{p_U}V^{(\tau)})_{q=2\hbar} = \sum_{\phi \in \rpp(U)} \tau(\phi) \prod_{x \in U} z^{\phi(x)}_{{\boldsymbol{\pi}(x)}},
\end{equation}
where the sum runs over the set of {\emph{reverse plane partitions}} over $U$.
\end{maintheorem}

The computation is via the localization theorem on quasimap spaces. In Section \ref{sssec_more_general} we provide  examples where explicit formulas such as (\ref{intro:rpp formula}) hold but {\emph{do not}} come from the dominant-minuscule skew heap; they nevertheless correspond to other heaps.
We then pose the question of describing these heaps combinatorially.

Finally, in Section \ref{sec:ADE-point-vertices} we restrict to types $A$ and $D$ and $\mathsf{w}$ supported at one minuscule node. We define differential operators which reconstruct all descendant vertex functions from the vertex with trivial descendant, see Theorems \ref{thm: diffop vertex}, \ref{thm: D rationality}. The explicit form of these operators immediately constrains the possible poles of the descendant vertex functions. Representation-theoretically, these operators recover the $q$-character of the corresponding module from its ordinary character. We speculate that similar operators should exist for arbitrary dominant-minuscule heaps and perhaps more generally.

\begin{maintheorem}[Theorems \ref{thm: diffop vertex}, \ref{thm: D rationality}]\label{mainthm diff operators}
Let $\Gamma$ be of type $A$ or $D$ and $\mathsf{w}$ be a delta-function of one minuscule node. Then:
\begin{itemize}
    \item The specialized vertex function $V^{(\tau)}_{q=2\hbar}$ can be recovered from $V^{(1)}_{q=2\hbar}$ by applying a certain differential operator (depending on $\tau$) in $z_i$.
    \item $V^{(\tau)}_{q=2\hbar}$ is a {\emph{rational}} function in $z_i$'s with poles only at $z^\alpha=1$ for some root $\alpha$.
\end{itemize}
\end{maintheorem}
As we discuss in Section \ref{ssec:deformation operators}, Theorem \ref{mainthm diff operators} admits a deformation outside the $q=2\hbar$ specialization and also a version in $K$-theory.

\subsection{Further directions}
We hope that the present paper opens several directions for further research. 

Let us list some of them.

\subsubsection{More general oriented cohomology theories}
The classical Hikita conjecture admits a $K$-theoretic generalization. This was studied by Dumanski and the third author in \cite{dk}. A quantum version in the hypertoric setting is discussed in \cite[Appendix B]{Zhou}; see also \cite[Section 6]{BaiLee2025QuantumAdams}.

There is a natural prospect of extending the results of the present paper to $K$-theory. All of the conjectures discussed above make sense in the $K$-theoretic setting, and some of the proofs go through as well.

The main point that becomes more delicate is the dimension assumption in Theorem \ref{mainthm:reduction} above.

It would also be interesting to go beyond $K$-theory and consider more general oriented cohomology theories. Given such a theory $\mathbb{E}$, one may ask both for $\mathbb{E}$-theoretic vertex functions and quantum $D$-modules, constructed from $\mathbb{E}$-theoretic virtual fundamental classes on moduli spaces of quasimaps, and for a corresponding notion of an $\mathbb{E}$-theoretic Coulomb branch.

One possible approach to the problem 
is to generalize the results of Teleman to $\mathbb E$; see \cite{T21,T22}.

\subsubsection{The Kähler setting}
As we already said, our main result is a quantum upgrade of the classical Hikita conjecture.

The classical Hikita conjecture is fundamentally a hyperKähler statement. Interestingly, there is a similar story in Kähler context.

Nakajima quiver variety $\widetilde{\mathcal M}_H$ admits a  Kähler analogue $\mathcal M_H^{\mathrm{Kah}}$,
in which no complex moment map equation is imposed. If $\mathcal M_C^{\mathrm{trHilb}}$ denotes the transverse $W$-Hilbert scheme (it is known, see \cite{BF23} and \cite{Web26}, that Coulomb branch is the underlying reduced subscheme of one of the irreducible components of $\mathcal M_C^{\mathrm{trHilb}}$) with its natural $\mathbb C^{\times}$-action, then Kifung Chan and Conan Leung \footnote{Unpublished work and private communication.} proved that there is an isomorphism of graded $\mathbb C$-algebras
$$
H^*\left(\mathcal M_H^{\mathrm{Kah}},\mathbb C\right)
\simeq
\mathbb C\left[
\left(\mathcal M_C^{\mathrm{trHilb}}\right)^{\mathbb C^{\times}}
\right].
$$

It would be interesting to understand the place of this in our quantum Hikita picture.

One may also mention a theory developed by Hausel and Rychlewicz that resembles the Hikita conjecture; see \cite{HR25}. In this setting, given a \textit{projective} variety $X$ with an action of an algebraic group $G$, they construct an isomorphism 
\begin{equation*}
\mathbb C[\mathcal Z_S] = H_G^*(X, \mathbb C)
\end{equation*}
where $\mathcal Z_S$ is a certain version of the fixed-point scheme \textit{on the same variety $X$} (i.e., not on a symplectically dual one).

When $X$ is additionally equipped with a proper $G$-equivariant map to a $G$-representation $\mathbf N$, the authors of \cite{CCL25}  construct an  action of the BFN quantization $\mathcal{A}=\mathcal A_{G,\mathbf N}$ on  the equivariant quantum $D$-module $QDM_G(X)$ which \textit{commutes with the quantum connection.} The operators coming from this action are known as \textit{shift operators.}

Note the contrast with the statement in our note: we work with a \textit{pair} of symplectically dual varieties, and do \textit{not} have an action of the BFN quantization for the Coulomb side on the quantum $D$-module for the Higgs one.

Since the authors of \cite{CCL25} put their work into a larger physical framework, it would be interesting to understand the precise connection between the two constructions (maybe using \cite{CL24a}).

Let us finally add that, ideologically, our paper produces a correspondence between a certain Givental $I$-function and a certain graded trace on the $3d$-mirror side. In the $2d$-mirror context, the shift operators appearing in \cite{CCL25} were related by Iritani to period integrals (see \cite[Prop.~3.20 and Thm.~3.21]{I17a}), and, further, to the $I$-functions (see \cite[Rem.~11]{I20}). This coincidence may deserve some further discussion.

\subsubsection{Characteristic $p$ version}
There is also a characteristic $p$ version of the quantum Hikita conjecture; see \cite{BaiLee2026MirrorPositiveCharacteristic}. 
As Lee explained to us, the refined quantum Hikita conjecture is expected to be compatible with the positive characteristic version of the quantum Hikita conjecture as studied in \cite{BaiLee2026MirrorPositiveCharacteristic}. Namely, one expects that the canonical map (resp. capped vertex map) from the upstairs $D$-module to the $D$-module of graded traces (resp. PSZ quantum $D$-module) should intertwine the classical Steenrod operations with the action of Lonergan's Frobenius center (resp. quantum Steenrod operations). The corresponding statement for Gromov--Witten quantum $D$-module for monotone GIT quotients is shown by \cite{xu}. To realize this claim, one either needs to first develop a theory of $\mathbb{F}_p$-valued quasimap quantum cohomology, or work entirely within the $K$-theoretic setting as in \cite{BaiLee2025QuantumAdams} where the relevant quasimap quantum power operations are already developed.

\subsubsection{Twisted cocenter}

The space $\operatorname{GrTr}(\mathcal{A})$ is a module over $\mathbb{C}[[{\boldsymbol{z}}]]$ classifying graded traces. It is obtained by completing a module over the partial compactification $\ol{\mathsf{A}}=\operatorname{Spec}(\mathbb{C}R^-)$ of the torus $\mathsf{A}$; in the paper, we denote this module by $\operatorname{GrTr}^{\mathrm{rat}}(\mathcal{A})$ (see \cite[Section 3.3]{KMBP} and Section \ref{ssec: D mod of graded traces}). 
The completion is taken at the point $0 \in \ol{\mathsf{A}}$ corresponding to the augmentation ideal of $\mathbb{C}R^-$.
The fiber of $\operatorname{GrTr}^{\mathrm{rat}}(\mathcal{A}_{\hbar=1})$ over $1 \in \mathsf{A}$ is the usual {\emph{cocenter}} of $\mathcal{A}_{\hbar=1}$; see \cite[Proposition 3.12]{KMBP}. Recall also that $\mathcal{A}$ is the quantized Coulomb branch algebra, whose multiplication is defined by convolution.

In \cite[Section 3]{BenZviNadlerPreygel2017Spectral}, the  authors develop a general theory describing cocenters of convolution algebras, as well as categorical cocenters of the corresponding categories. It would be very interesting to understand whether  their methods can be adapted to our ``twisted'' setting  ``over'' $\overline{\mathsf{A}}$. Such an adaptation might provide a general approach to the refined quantum Hikita conjecture and perhaps also to its categorification. A related question is to determine the correct categorical notion corresponding to a twisted trace.

Another interesting observation, communicated to us by Andrei Ionov, is that $\operatorname{GrTr}^{\mathrm{rat}}(\mathcal{A}_{\hbar=1})$ can be viewed as a quantized, multi-parameter analogue of the ``trace'' of the Drinfeld--Gaitsgory interpolation \cite[Section 2.2]{DrinfeldGaitsgoryBraden}. Namely, in the notation of loc. cit., starting with an affine space $Z$ equipped with an action of $\mathbb{C}^\times$, the authors construct a space $\widetilde{Z} \subset Z \times Z \times \mathbb{A}^1$. From this perspective, $\operatorname{GrTr}^{\mathrm{rat}}(\mathcal{A}_{\hbar=1})$ should be thought of as an analogue of the schematic relative intersection:
\begin{equation}\label{eq:int DG diagonal}
\widetilde{Z} \cap (\Delta_Z \times \mathbb{A}^1) := \widetilde{Z} \times_{Z \times Z \times \mathbb{A}^1} (Z \times \mathbb{A}^1).
\end{equation}

 Let us finally mention that there is a natural candidate for the ``quantized Drinfeld--Gaitsgory interpolation'' in our setting. Let us make $\hbar$ invertible.
 In the notations of Section \ref{ssec: D mod of graded traces}, consider a cyclic $\mathcal{A}[\hbar^{-1}]$-bimodule $\mathscr{B}[\hbar^{-1}]$ over $\mathbb{C}R^-$, generated by a symbol $e$, with relations, for $\xi \in R^-$:
\begin{equation}\label{eq:quantized DG}
ea=z^\xi ae, \quad a \in \mathcal{A}^\xi[\hbar^{-1}];  \quad be=z^{\xi} eb, \quad b \in \mathcal{A}^{-\xi}[\hbar^{-1}].
\end{equation}
Then 
\begin{equation*}
HH_0(\mathcal{A}[\hbar^{-1}],\mathscr{B}[\hbar^{-1}]) \simeq \operatorname{GrTr}^{\mathrm{rat}}(\mathcal{A}[\hbar^{-1}]).
\end{equation*}
So, in particular, $\mathscr{B}[\hbar^{-1}]$ is the noncommutative analogue of the Drinfeld--Gaitsgory interpolation and $\operatorname{GrTr}^{\mathrm{rat}}(\mathcal{A}[\hbar^{-1}])$ is nothing else but the cocenter of this bimodule.

It would be interesting to understand whether the full Hochschild homology $HH_\bullet(\mathcal{A}[\hbar^{-1}],\mathscr{B}[\hbar^{-1}])$ 
provides the appropriate derived enhancement of $\operatorname{GrTr}^{\mathrm{rat}}(\mathcal{A}[\hbar^{-1}])$, and whether the appearance of the quantized Drinfeld--Gaitsgory interpolation has a  conceptual explanation in the quantum Hikita story.

Note that the definition (\ref{eq:quantized DG}) makes sense without inverting $\hbar$. The only difference is that then:
\begin{equation*}
\operatorname{GrTr}^{\mathrm{rat}}(\mathcal{A}) \simeq HH_0(\mathcal{A},\mathscr{B})^{\boldsymbol{0}}
\end{equation*}
so we have to {\emph{manually}} pass to the degree zero component. 

On the other hand, if we work with the {\emph{$\mathbb{C}[[{\boldsymbol{z}}]]$-completed}} versions $\operatorname{GrTr}(\mathcal{A})$, $\widehat{\mathscr{B}}$, then it is still true that 
\begin{equation*}
\operatorname{GrTr}(\mathcal{A}) \simeq HH_0^{\operatorname{cont.}}(\mathcal{A},\widehat{\mathscr{B}}).
\end{equation*}

In general, we expect that $\operatorname{GrTr}^{\mathrm{rat}}(\mathcal{A}[\hbar^{-1}])$ is a better-behaved object than $\operatorname{GrTr}^{\mathrm{rat}}(\mathcal{A})$, and the fact that the quantum Hikita conjecture holds for $\hbar=0$ is a phenomenon specific for quiver gauge theories: indeed, in \cite[Section 6]{HKM}, it's explained that the classical ($\hbar=0$) Hikita conjecture may fail outside of the quiver gauge theories.

\subsubsection{Sphere trace}

Another interesting direction is to understand, in our framework, the sphere trace functional predicted by Gaiotto--Okazaki \cite{gaiotto-sphere-quantization}; see also \cite{GHRYWZ}. It was rigorously constructed by 
Zhang \cite{zhang-analytical-traces-coulomb} in type A, and by Klyuev \cite{Klyuevres} in general, under the assumption that the theory is good or ugly. See Section \ref{sec:graded traces vs twisted traces} for further discussion.

\subsection{Acknowledgments}

We gratefully acknowledge helpful conversations with Shaoyun Bai, David Ben-Zvi, Roman Bezrukavnikov, Alexei Borodin,  Tommaso Maria Botta, Alexander Braverman, Ivan Danilenko, Ilya Dumanski, Pavel Etingof, Michael Finkelberg, Dan Freed, Victor Ginzburg, Andrei Ionov, David Kazhdan, Daniil Klyuev, Henry Liu, Jae Hee Lee, Ivan Losev, Davesh Maulik, Michael McBreen, Dinakar Muthiah, Sujay Nair, Nikita Nekrasov, Andrei Okounkov, Andrei Smirnov, Yan Soibelman, Spencer Tamagni, Richard Thomas, Keke Zhang, and  Tianqing Zhu. 

We would like to especially thank Artem Kalmykov, Joel Kamnitzer, Alexis Leroux-Lapierre, Th\'eo Pin\'et, and Alex Weekes for very helpful conversations, as well as for generously sharing their unpublished results, some of which we use in the proofs. 

We thank Kifung Chan, Tommaso Maria Botta, Joel Kamnitzer, Spencer Tamagni, and Richard Thomas for useful and kind comments on the present text.

H. Dinkins was supported by NSF grant DMS-2303286 at MIT and the NSF RTG grant Algebraic Geometry and  Representation Theory at Northeastern University DMS–1645877. 

I. Karpov was supported by D. Maulik’s steering weekly inquiries as to whether this work would ever see the light of day before his 60th birthday conference (that is, in the next 15 years), and by the third author's kind permission to spend a night on the grass in his Paris courtyard.

V. Krylov was supported by the Simons Foundation Award
888988 as part of the Simons Collaboration on Global Categorical Symmetries. 

V. Krylov apologizes for making an incorrect claim in a couple of talks he gave on this work; see footnote \ref{vasya_correction} below.

We used ChatGPT 5.6 Sol (High and Pro versions) to assist with editing the manuscript, providing examples, references, and proofreading.

\section{Hikita Conjectures, Graded Traces, and Quasimaps}\label{sec:overview}

\subsection{Quiver gauge theories}\label{ssec: quiver gauge theories}
Let $\Gamma=(I,E)$ be a finite quiver with vertex set $I$ and arrow set $E$. Set $r:=|I|$ and label the elements of $I$ by $1,\ldots,r$. For an arrow $e \in E$, we denote by $t(e)$ and $h(e)$ the vertices at the tail and head of $e$, i.e. $t(e) \xrightarrow{e} h(e)$. Let $\dv \in \mathbb{Z}_{\geqslant 0}^{r}$ and $\dw \in \mathbb{Z}_{\geqslant 0}^{r}$. Choose vector spaces $V_i,W_i$ with $\dim V_i=\mathsf{v}_i$ and $\dim W_i=\mathsf{w}_i$. Set
$$
{\bf{N}}={\bf{N}}(\dv,\dw):=\bigoplus_{e \in E} \Hom(V_{t(e)},V_{h(e)}) \oplus \bigoplus_{i \in I} \Hom(W_{i},V_{i}).
$$
The group $G_{\dv}=\prod_{i \in I} GL(V_{i})$ acts on ${\bf{N}}$. It is expected that a pair $(G_{\dv},{\bf{N}})$ defines a $3d$ $\mathcal{N}=4$ {\emph{gauge theory}}. There is no mathematical definition of this theory, but it is known how to define its Higgs and Coulomb branches $\mathcal{M}_H$, $\mathcal{M}_C$, which are (usually singular) affine Poisson varieties.

\subsubsection{Higgs branch}\label{ssec:Nakajima quiver var defin}
The variety $\mathcal{M}_H$ is the Hamiltonian reduction 
\begin{equation*}
\mathcal{M}_H := T^*{\bf{N}}/\!\!/\!\!/ G_{\dv} = \operatorname{Spec}\bigl(\mathbb{C}[\mu^{-1}(0)]^{G_{\dv}}\bigr),
\end{equation*}
where $\mu\colon T^*{\bf{N}} \rightarrow \operatorname{Lie}(G_{\dv})^*$ is the moment map for $G_{\dv} \curvearrowright T^*{\bf{N}}$. It was introduced by Nakajima and will be called an {\emph{affine Nakajima quiver variety}}.

For $\theta \in \mathbb{Z}^{r}$, we define a character
\begin{align}\label{eq: def of chi theta}
\chi_{\theta}\colon & G_{\dv} \to \mathbb{C}^{\times} \\
&(g_{i}) \mapsto \prod_{i = 1}^{r} \det(g_{i})^{\theta_{i}}.
\end{align}

Fixing a {\emph{generic}} character $\chi_\theta\colon G_{\dv} \rightarrow \mathbb{C}^\times$, we can consider the open subset $\mu^{-1}(0)^{\theta-\mathrm{st}} \subset \mu^{-1}(0)$ of $\theta$-stable points. It is known that the action of $G_{\dv}$ on $\mu^{-1}(0)^{\theta-\mathrm{st}}$ is free, and we can define the ``GIT version'' of the Hamiltonian reduction:
\begin{equation*}
\widetilde{\mathcal{M}}_{H,\theta} = T^*{\bf{N}}/\!\!/\!\!/_\theta G_{\dv} := \mu^{-1}(0)^{\theta-\mathrm{st}}/G_{\dv}.
\end{equation*}
The variety $\widetilde{\mathcal{M}}_{H,\theta}$ was also introduced by Nakajima and will be called a {\emph{Nakajima quiver variety}}. It is a smooth symplectic variety equipped with a natural map $\widetilde{\mathcal{M}}_{H,\theta} \rightarrow \mathcal{M}_H$ inducing a resolution of singularities of its image.

One natural choice for $\theta$, which is always generic, is $\theta=(1,1,\ldots,1)\in \mathbb{Z}^{r}$. So $\chi_\theta$ is the product of determinants. In the main body of the paper, we will always assume $\theta=(1,1,\ldots,1)$ and will frequently omit it from the notation. 

Let $G_{\sF}$ be an arbitrary connected reductive group acting on ${\bf{N}}$ and commuting with $G_{\dv}$. It acts naturally on $T^*{\bf{N}}$, and this action descends to
\begin{equation*}
G_{\sF} \curvearrowright \mathcal{M}_H,\, \widetilde{\mathcal{M}}_H
\end{equation*}
by Poisson (resp. symplectic) automorphisms. 

\subsubsection{Coulomb branch}
The Coulomb branch $\mathcal{M}_C=\mathcal{M}_C(G_{\dv},{\bf{N}})$ corresponding to a pair $(G_{\dv},{\bf{N}})$ was rigorously introduced by Braverman, Finkelberg and Nakajima in \cite{BFN2}. Let us recall their definition as well as some objects we will be using throughout the paper.

Let $\operatorname{Gr}_{G_{\dv}}:=G_{\dv}((z))/G_{\dv}[[z]]$ be the affine Grassmannian of $G_{\dv}$. 
Let $\La_{\dv}:=\operatorname{Hom}(\mathbb{C}^\times,T_{\dv})$ be the cocharacter lattice of $G_{\dv}$ and let $\La_{\dv}^+ \subset \La_{\dv}$ be the submonoid of dominant cocharacters. For ${\boldsymbol{\eta}} \in \La_{\dv}^+$ let $z^{\boldsymbol{\eta}}$ be the corresponding point of $\operatorname{Gr}_{G_{\dv}}$. Set $\operatorname{Gr}^{\boldsymbol{\eta}} := G_{\dv}[[z]] \cdot z^{\boldsymbol{\eta}}$ and let $\operatorname{Gr}^{\leqslant {\boldsymbol{\eta}}}=\ol{\operatorname{Gr}^{\boldsymbol{\eta}}} \subset \operatorname{Gr}_{G_{\dv}}$ be its closure. It is a projective variety of dimension $\langle 2{\boldsymbol{\rho}},{\boldsymbol{\eta}}\rangle$ containing $\operatorname{Gr}^{\boldsymbol{\eta}}$ as an open nonsingular subvariety. Here $2{\boldsymbol{\rho}} \in \La^\vee_{\dv}$ is the sum of positive roots for $G_{\dv}$, and $\La^\vee_{\dv}=\operatorname{Hom}(T_{\dv},\mathbb{C}^\times)$ is the character lattice of $G_{\dv}$.

Set
\begin{equation*}
\mathcal{T} = \mathcal{T}_{G_{\mathsf{v}},{\bf{N}}} := G_{\dv}((z)) \times^{G_{\dv}[[z]]} {\bf{N}}[[z]],~ \mathcal{R}=\mathcal{R}_{G_{\mathsf{v}},{\bf{N}}} := \{[(g,n)] \in \mathcal{T}\,|\, g \cdot n \in {\bf{N}}[[z]]\}.
\end{equation*}
Both $\operatorname{Gr}_{G_\dv}$ and $\mathcal{R}$ are equipped with the actions of $G_{\dv}[[z]]$ induced by the left multiplication. 
The natural morphism $\mathcal{R} \rightarrow \operatorname{Gr}_{G_{\dv}}$ is $G_{\dv}[[z]]$-equivariant.
Let $\mathcal{R}_{\boldsymbol{\eta}}$, $\mathcal{R}_{\leqslant \boldsymbol{\eta}}$ be the preimages of $\operatorname{Gr}^{\boldsymbol{\eta}}$, $\operatorname{Gr}^{\leqslant \boldsymbol{\eta}}$, respectively. The morphism $\mathcal{R}_{\boldsymbol{\eta}} \rightarrow \operatorname{Gr}^{\boldsymbol{\eta}}$ is an affine fibration (with infinite-dimensional fibers).

In \cite{BFN2}, the authors define the Coulomb branch $\mathcal{M}_C$ as
\begin{equation*}
\mathcal{M}_C = \operatorname{Spec}\Big(H_*^{G_{\dv}[[z]]}(\mathcal{R}),\star\Big).
\end{equation*}
The vector space $H_*^{G_{\dv}[[z]]}(\mathcal{R})$ has a commutative algebra structure $\star$ arising from a convolution diagram for $\mathcal{R}$, and the Coulomb branch $\mathcal{M}_C$ is the spectrum of this commutative algebra. 

\begin{remark}
Here, as well as everywhere else in this paper (unless explicitly stated otherwise), the notation $H_*$ stands for the \textit{Borel--Moore} homology (see \cite[2(ii)]{BFN1} for the details).
\end{remark}

One can similarly define a {\emph{quantization}} $\mathcal{A}_C$ of the Poisson algebra $\mathbb{C}[\mathcal{M}_C]$:
\begin{equation*}
\mathcal{A}_C = \Big(H_*^{G_{\dv}[[z]] \rtimes \mathbb{C}^\times}(\mathcal{R}),\star\Big),
\end{equation*}
Here the semidirect product is formed using the loop-rotation action
\begin{equation*}
(t\cdot g)(z)=g(t^2z), \qquad t\in\mathbb{C}^\times,\quad g(z)\in G_{\dv}[[z]],
\end{equation*}
which also induces the loop-rotation action on $\mathcal{R}$. We denote by $\hbar$ the inverse of the standard character of this $\mathbb C^\times$. Thus, in additive notation, the quantization parameter associated with the above loop rotation is $2\hbar$. In this paper, we will only consider quantizations of Coulomb branches, so we will omit ``$C$'' from the notation. The algebra $\mathcal{A}$ is an associative (noncommutative) algebra over $\mathbb{C}[\hbar]=H^*_{\mathbb{C}^\times}(\operatorname{pt})$ such that $\mathcal{A}/(\hbar) \simeq \mathbb{C}[\mathcal{M}_C]$.

Recall the group $G_{\sF}$ acting on ${\bf{N}}$. Let $\sF \subset G_{\sF}$ be a maximal torus and let $\sff:=\operatorname{Lie}\sF$, $W_{\sff}:=N_{G_{\sF}}(\sF)/\sF$. The group $G_{\sF}$ acts naturally on the space $\mathcal{R}$, allowing one to introduce {\emph{deformations}} of $\mathcal{A}$, $\mathcal{M}_C$ over $H^*_{G_{\sF}}(\operatorname{pt})=\mathbb{C}[\sff]^{W_{\sff}}$, $\operatorname{Spec}H^*_{G_{\sF}}(\operatorname{pt})=\sff/W_{\sff}$:
\begin{equation}\label{eq: deformations of Coulomb}
\mathcal{A}_{\sff/W_{\sff}} = H^{(G_{\dv}[[z]] \rtimes \mathbb{C}^\times)\times G_{\sF}}_*(\mathcal{R}),~\mathcal{M}_{C,\sff/W_{\sff}}=\operatorname{Spec}\Big(H_*^{G_{\dv}[[z]] \times G_{\sF}}(\mathcal{R})\Big).
\end{equation}

\begin{remark}
Note that the deformations (\ref{eq: deformations of Coulomb}) corresponding to $G_{\sF}$ can be recovered from the deformations corresponding to $\sF$ by taking $W_{\sff}$-invariants. 
\end{remark}

Let $\mathcal{A}_{\sff}$ and $\mathcal{M}_{C,\sff}$ denote the pullbacks of these deformations along $\sff\to\sff/W_{\sff}$. From now on, abusing notation, we will abbreviate them to $\mathcal{A}$ and $\mathcal{M}_C$. We also set $\tGh := (G_{\dv}[[z]] \rtimes \mathbb{C}^\times) \times \sF$.

The algebra $\mathcal{A}$ contains a certain {\emph{commutative subalgebra}} to be denoted $\mathscr{H}_{\hbar}$. It is defined as follows. Let $T_{\dv} \subset G_{\dv}$ be the maximal torus and let $\mathfrak{t}_{\dv}$ be its Lie algebra. Set $W_{\dv}:=N_{G_{\dv}}(T_{\dv})/T_{\dv}$. Then
\begin{equation*}
\GZ_{\hbar} = H^*_{\tGh}(\operatorname{pt})=\mathbb{C}[\mathfrak{t}_{\dv}]^{W_{\dv}} \otimes \mathbb{C}[\hbar] \otimes \mathbb{C}[\mathsf{f}],
\end{equation*}
and the embedding $\GZ_\hbar \hookrightarrow \mathcal{A}$ is given by 
\begin{equation*}
\GZ_\hbar \ni x \mapsto x \cap [\mathcal{R}_{\boldsymbol{0}}] \in H_*^{\tGh}(\mathcal{R}_{\boldsymbol{0}})\subset \mathcal{A},
\end{equation*}
where $[\mathcal{R}_{\boldsymbol{0}}]$ is the fundamental class of $\mathcal{R}_{\boldsymbol{0}}$ (representing the identity element in $\mathcal{A}$).
Specializing $\hbar=0$ we obtain the embedding $\GZ:=\GZ_{\hbar=0} \hookrightarrow \mathbb{C}[\mathcal{M}_C]$.

Both $\mathcal{A}$ and $\mathbb{C}[\mathcal{M}_C]$ admit a natural $\mathbb{Z}$-grading (see \cite[Sections 2(ii) and 3(v)]{BFN2}). We normalize this grading in such a way that upon restriction to $\GZ_\hbar \subset \mathcal{A}$ it becomes the grading on $H^*_{\tGh}(\operatorname{pt})$ by the cohomological degree. Because of this, we will call the aforementioned grading on $\mathcal{A},\,\mathbb{C}[\mathcal{M}_C]$ the {\emph{cohomological}} grading. The cohomological grading induces the action $\mathbb{C}^\times \curvearrowright \mathcal{M}_C$. 

For ${\boldsymbol{\eta}} \in \Lambda_{\dv}^+$ and after the identification $H_*^{\tGh}(\mathcal{R}_{\boldsymbol{\eta}}) \simeq H^*_{\tGh}(\operatorname{Gr}^{\boldsymbol{\eta}})$, the degree on the LHS is always {\emph{even}} and is equal to the $-\langle 4\boldsymbol{\rho},{\boldsymbol{\eta}}\rangle+2d_{\boldsymbol{\eta}}$-shifted cohomological degree on the RHS, where $d_{\boldsymbol{\eta}}=\operatorname{rank}(\mathcal{T}_{\boldsymbol{\eta}}/\mathcal{R}_{\boldsymbol{\eta}})$ is given by the formula in \cite[end of Section 2(i)]{BFN1} for $\mathcal T_{\boldsymbol{\eta}} = \mathcal T \times_{\operatorname{Gr}_{G_{\mathsf{v}}}} \operatorname{Gr}^{\boldsymbol{\eta}}$.

Following physics terminology, one often separates quiver gauge theories into
good, ugly and bad theories. In the present language, this trichotomy can be
read from the so-called ``twisted cohomological grading'', see \cite[Remark 2.8 and (2.10)]{BFN1}:
$$
\mathbb C[\mathcal M_C]=\bigoplus_{d\in \mathbb Z}\mathbb C[\mathcal M_C]_d .
$$
We will say that the theory is non-bad if
$$
\mathbb C[\mathcal M_C]_d=0 \quad \text{for } d<0,
\qquad
\mathbb C[\mathcal M_C]_0=\mathbb C .
$$
Equivalently, the induced action $\mathbb C^\times\curvearrowright \mathcal M_C$
is conical. A non-bad theory is called good if, moreover,
$$
\mathbb C[\mathcal M_C]_1=0,
$$
and ugly otherwise. Thus ugly theories are still conical, but have degree one
functions. Finally, we will call the theory bad if the
non-bad condition above fails. In what follows we will use the word conical
whenever only non-negativity of the grading and $\mathbb C[\mathcal M_C]_0=\mathbb C$
are needed, and reserve the word good for the stronger condition excluding
degree one functions.

It is important that the resulting criterion for goodness in terms of the twisted cohomological grading is \textit{equivalent} to the conicity of the usual cohomological grading. 

Set $\mathsf{A}:=\pi_1(G_{\dv})^{\vee}$, where by $\bullet^{\vee}$ we mean the Pontryagin dual. The torus $\mathsf{A}$ canonically identifies with the {\emph{dual}} torus to $G_{\mathsf{v}}/[G_{\mathsf{v}},G_{\mathsf{v}}]$.    
Recall that $\theta \in \mathbb{Z}^I$ defines a character $\chi_\theta\colon G_{\dv} \rightarrow \mathbb{C}^\times$ and hence also defines a cocharacter $\theta\colon \mathbb{C}^\times \rightarrow \mathsf{A}$, which we denote by the same symbol.

Recall that $\pi_0(\operatorname{Gr}_{G_{\dv}})=\pi_1(G_{\mathsf{v}})$ and $\mathcal{R}$ maps to $\operatorname{Gr}_{G_{\dv}}$. It follows that the algebras $\mathbb{C}[\mathcal{M}_C]$, $\mathcal{A}$ are $\pi_1(G_{\dv})$-graded, so the torus $\mathsf{A}$ acts naturally on $\mathcal{M}_C$ (see \cite[Section 3(v)]{BFN2} for the details).

We now give an explicit description of the identification 
\begin{equation*}
\pi_0(\operatorname{Gr}_{G_{\mathsf{v}}})= \pi_1(G_{\mathsf{v}}) \simeq \mathbb{Z}^r 
\end{equation*} 
inducing the isomorphism $\mathsf{A} \simeq (\mathbb{C}^\times)^r$.
Pick any cocharacter ${\boldsymbol{\eta}} \in \La_{\dv}$. We have ${\boldsymbol{\eta}}=(\eta_i)_{i=1,\ldots,r}$, where $\eta_i$ is a  cocharacter of $GL_{\mathsf{v}_i}$. Let $|\eta_i|$ denote the sum of coordinates of $\eta_i$. Then, the identification $\pi_0(\operatorname{Gr}_{G_{\mathsf{v}}}) \simeq \mathbb{Z}^r$ is given by:
\begin{equation}\label{eq: ident pi 0 lattice}
[z^{\boldsymbol{\eta}}] \mapsto (|\eta_i|)_{i=1,\ldots,r},
\end{equation}
where $[z^{\boldsymbol{\eta}}]$ is the connected component of $z^{\boldsymbol{\eta}} \in \operatorname{Gr}_{G_{\mathsf{v}}}$.

The algebra $\mathcal{A}$ is generated over $\GZ_\hbar$ by so-called {\emph{dressed minuscule monopole operators}} $M_{\boldsymbol{\eta},f}$. Let us recall their definition.

Let $\boldsymbol{\eta} \in \La^+_{\dv}$ be a dominant minuscule coweight for $G_{\mathsf{v}}$ and $f \in H^*_{L_{\boldsymbol{\eta}} \times \mathbb{C}^\times\times\sF}(\operatorname{pt})$, where $L_{\boldsymbol{\eta}} \subset G_{\mathsf{v}}$ is the centralizer of ${\boldsymbol{\eta}}$ in $G_{\mathsf{v}}$. Recall $\operatorname{Gr}^{\boldsymbol{\eta}} \subset \operatorname{Gr}_{G_{\mathsf{v}}}$. It is a closed $G_{\mathsf{v}}[[z]] \rtimes \mathbb{C}^\times$-invariant subvariety of $\operatorname{Gr}_{G_{\mathsf{v}}}$ isomorphic to $G_{\dv}/P_{\boldsymbol{\eta}}$, where $P_{\boldsymbol{\eta}} \subset G_{\mathsf{v}}$ is a certain parabolic subgroup corresponding to the Levi $L_{\boldsymbol{\eta}}$ (see \cite[Lemma 2.1.13]{zhu_aff_gr}). 
The natural morphism $\mathcal{R}_{\boldsymbol{\eta}} \rightarrow \operatorname{Gr}^{\boldsymbol{\eta}}$ is an affine fibration, 
so $H^*_{L_{\boldsymbol{\eta}} \times \mathbb{C}^\times\times\sF}(\operatorname{pt}) = H_*^{(G_{\mathsf{v}}[[z]] \rtimes \mathbb{C}^\times)\times\sF}(\mathcal{R}_{\boldsymbol{\eta}})$. Taking the pushforward of 
\begin{equation*}
f \in H^*_{L_{\boldsymbol{\eta}} \times \mathbb{C}^\times\times\sF}(\operatorname{pt}) = H_*^{(G_{\mathsf{v}}[[z]] \rtimes \mathbb{C}^\times)\times\sF}(\mathcal{R}_{\boldsymbol{\eta}}) \longrightarrow H_*^{(G_{\mathsf{v}}[[z]] \rtimes \mathbb{C}^\times) \times \sF}(\mathcal{R}) = \mathcal{A}
\end{equation*}
we obtain the desired monopole operator $M_{\boldsymbol{\eta},f} \in \mathcal{A}$.

\begin{lemma}\label{lem: degrees minuscule}
For $f \in H^m_{L_{\boldsymbol{\eta}} \times \mathbb{C}^\times\times\sF}(\operatorname{pt})$ we have: 
\begin{equation*}
\operatorname{deg}_{\mathbb{Z}^r \times \mathbb{Z}}(M_{\boldsymbol{\eta},f}) = ((|\eta_i|)_{i=1,\ldots,r}, m + 2d_{\boldsymbol{\eta}}-\langle 4\boldsymbol{\rho}, \boldsymbol{\eta}\rangle).
\end{equation*}
\end{lemma} 
\begin{proof}
The lemma follows from the explicit formula (\ref{eq: ident pi 0 lattice}) for the identification $\pi_0(\operatorname{Gr}_{G_{\mathsf{v}}})=\pi_1(G_{\dv}) \simeq \mathbb{Z}^r$.
\end{proof}

For $i \in I$, $1 \leqslant k \leqslant \mathsf{v}_i$, let $\omega_{i,k}$, $\omega_{i,k}^*$ be the following coweights of $GL_{\mathsf{v}_i}$, considered as coweights of $G_{\mathsf{v}}$ via the natural inclusion $GL(V_i) \hookrightarrow G_{\dv}$:
\begin{equation*}
\omega_{i,k}=(\underbrace{1,1,\ldots,1}_{k},0,\ldots,0),~\omega_{i,k}^*=(0,\ldots,0,\underbrace{-1,-1,\ldots,-1}_{k}).
\end{equation*}

The following proposition goes back to \cite[Section 4.3]{BDG}, \cite[Proposition 6.2 and Remark 6.7]{BFN1}, and \cite[Theorem 4.29]{FT}. It describes the generators of $\mathcal{A}$ considered as an algebra over $\GZ_{\hbar}$.
\begin{proposition}\label{prop: generators Coulomb}
The algebra $\mathcal{A}$ is generated over $\GZ_\hbar$ by the operators $M_{{\boldsymbol{\lambda}},f}$, where $\boldsymbol{\lambda}$ is of the form $\sum_{i \in J} \omega_{i,k_i}$ or $\sum_{i \in J}\omega_{i,k_i}^*$, $J \subset I$ is any nonempty subset and $1 \leqslant k_i \leqslant {\mathsf{v}}_i$. 
\end{proposition}
\begin{proof}
This follows from the proof of \cite[Theorem 4.29]{FT}; see also the proof of \cite[Proposition 3.1]{weekes_generators}.   
\end{proof}

We will need a slightly stronger statement of the Poincar\'e--Birkhoff--Witt type. Namely, Proposition \ref{prop: generators Coulomb} claims that $\mathcal A$ is generated by monomials in elements of $\GZ_{\hbar}$ and monopole operators. The same actually holds for \textit{ordered} monomials.

More precisely, let $\mathcal P^{\mathrm{pos}},\mathcal P^{\mathrm{neg}}\subset \Lambda_{\dv}^+$ be the sets of dominant coweights
of the form
$$
 \mathcal P^{\mathrm{pos}}=\left\{\sum_{i\in J}\omega_{i,k_i}\right\},
 \quad
 \mathcal P^{\mathrm{neg}}=\left\{\sum_{i\in J}\omega^{*}_{i,k_i}\right\},
 \qquad
 J\subset I,
 \quad J\neq\varnothing,
 \quad 1\leqslant k_i\leqslant \dv_i .
$$
Let $\mathcal{A}^{\mathrm{pos}} \subset \mathcal{A}$ be the subalgebra generated by $M_{\boldsymbol{\eta},f}$ with $\boldsymbol{\eta} \in \mathcal{P}^{\mathrm{pos}}$. Similarly define $\mathcal{A}^{\mathrm{neg}} \subset \mathcal{A}$.

\begin{proposition}\label{PBW}
The  multiplication morphisms 
\begin{equation}\label{eq:mult_morphism_+_0_-}
\mathcal{A}^{\mathrm{pos}} \otimes \mathscr{H}_{\hbar} \otimes \mathcal{A}^{\mathrm{neg}} \rightarrow \mathcal{A},
\end{equation}
\begin{equation}\label{eq:mult_morphism_-_0_+}
    \mathcal{A}^{\mathrm{neg}} \otimes \mathscr{H}_{\hbar} \otimes \mathcal{A}^{\mathrm{pos}} \rightarrow \mathcal{A}
\end{equation}
are {\emph{surjective}}.
\end{proposition}

\begin{proof}
We prove the claim for the morphism (\ref{eq:mult_morphism_+_0_-}).
Let $\mathcal B$ be its image.

As in \cite[Section 6]{BFN1}, consider the following $\mathbb Z_{\geqslant 0}$-filtration on $\mathcal A$:
$$
 F^k\mathcal A=
 H^{\widetilde{G}_{\dv}}_{\ast}\left(
 \bigcup_{\boldsymbol{\eta}\in\Lambda_{\dv}^+,\;\langle 2\brho,\boldsymbol{\eta}\rangle\leqslant k}\mathcal R_{\boldsymbol{\eta}}
 \right).
$$
The associated graded algebra identifies with
$$
 \operatorname{gr}\mathcal A=\bigoplus_{\boldsymbol{\eta}\in\Lambda_{\dv}^+}
 H^{\widetilde{G}_{\dv}}_{\ast}(\mathcal R_{\boldsymbol{\eta}}).
$$
It is enough to check that every
summand $H^{\widetilde{G}_{\dv}}_{\ast}(\mathcal R_{\boldsymbol{\eta}})\subset \operatorname{gr}\mathcal A$ belongs to
$\operatorname{gr}\mathcal{B}$. We prove this by induction on the pair 
\begin{equation}\label{eq:pair_order_to_eta}
(\langle 2\boldsymbol{\rho}, \boldsymbol{\eta}\rangle, \sum_{i,s}|\eta_{i,s}|),
\end{equation} 
ordered lexicographically.

The case $\boldsymbol{\eta}={\boldsymbol{0}}$ is clear as $H_*^{\widetilde{G}_{\mathsf{v}}}(\mathcal{R}_{{\boldsymbol{0}}})=\mathscr{H}_\hbar$.

Pick, then, $0\neq\boldsymbol{\eta}\in\Lambda_{\mathsf v}^+$, and let $C\subset\Lambda_{\dv}^+$ be a (closed) generalized chamber 
containing $\boldsymbol{\eta}$ and consisting of elements $(\zeta_{i,s}) \in \mathfrak{t}_{\mathsf{v},\mathbb{R}}=\prod_{j \in I}\mathfrak{t}_{v_j,\mathbb{R}}$ of the form
$$
 \zeta_{i_1,s_1}\geqslant \zeta_{i_2,s_2} \geqslant \cdots \geqslant
 \zeta_{i_p,s_p} \geqslant 0 \geqslant
 \zeta_{i_{p+1},s_{p+1}} \geqslant \cdots \geqslant
 \zeta_{i_{|\mathsf v|},s_{|\mathsf v|}},
$$
where $|\mathsf{v}|=\sum_{i \in I}v_i$, $i_k \in I$, $s_k \in \{1,\ldots,v_{i_k}\}$.

First assume that $\boldsymbol{\eta}$ has a positive coordinate. For each $i\in I$, set
$$
 k_i^+=\#\{s = 1,\ldots v_i\mid \eta_{i,s}>0\},
 \qquad
 \boldsymbol\omega=\sum_{i:\,k_i^+>0}\boldsymbol\omega_{i,k_i^+}.
$$
Then $\boldsymbol\omega\in\mathcal P^{\mathrm{pos}}$, both $\boldsymbol\omega$ and $\boldsymbol\eta-\boldsymbol\omega$ lie
in $C$, and the pair (\ref{eq:pair_order_to_eta}) associated with $\boldsymbol\eta-\boldsymbol\omega$ is lexicographically smaller than the pair associated with $\boldsymbol\eta$.

By \cite[Proposition 6.2]{BFN1}, the multiplication map 
\begin{equation}\label{eq: mult in chamber}
H_*^{\widetilde{G}_{\dv}}(\mathcal{R}_{{\boldsymbol{\omega}}}) 
\otimes
H_*^{\widetilde{G}_{\dv}}(\mathcal{R}_{\boldsymbol{\eta}-{\boldsymbol{\omega}}})   \rightarrow H_*^{\widetilde{G}_{\dv}}(\mathcal{R}_{{\boldsymbol{\eta}}}) 
\end{equation}
is given by 
\begin{equation}\label{eq:expl formula mult in chamber}
f[\mathcal R_{\boldsymbol{\omega}}] * g[\mathcal R_{\boldsymbol{\eta}-\boldsymbol{\omega}}]
=
fg(\,\cdot+2\hbar\boldsymbol{\omega}\,) [\mathcal R_{\boldsymbol{\eta}}],
\end{equation}
where\footnote{There is a typo in the version of this formula in \cite[Proposition 6.2]{BFN1} involving the $\hbar$-shift.}
\begin{equation*}
f \in \mathbb{C}[\mathfrak{t}_{\dv}]^{W_{\boldsymbol{\omega}}} \otimes \mathbb{C}[\sff,\hbar],~g \in \mathbb{C}[\mathfrak{t}_{\dv}]^{W_{\boldsymbol{\eta}-\boldsymbol{\omega}}} \otimes \mathbb{C}[\sff,\hbar],~fg(\,\cdot+2\hbar\boldsymbol{\omega}\,) \in \mathbb{C}[\mathfrak{t}_{\dv}]^{W_{\boldsymbol{\eta}}} \otimes \mathbb{C}[\sff,\hbar].
\end{equation*}
Let $p\colon \widetilde{\operatorname{Gr}}^{\boldsymbol{\omega},\boldsymbol{\eta}-\boldsymbol{\omega}} \rightarrow \operatorname{Gr}^{\leqslant \boldsymbol{\eta}}$ be the convolution diagram over $\operatorname{Gr}^{\leqslant \boldsymbol{\eta}}$ as in \cite[Section 4]{MirVil}. 
The map $p$ is an isomorphism over an open subset $\operatorname{Gr}^{\boldsymbol{\eta}} \subset \operatorname{Gr}^{\leqslant \boldsymbol{\eta}}$ and so the restriction map
\begin{equation}\label{eq: restriction convolution}
H_*^{\widetilde{G}_{\mathsf{v}}}(\operatorname{Gr}^{\boldsymbol{\omega}}) \otimes_{H^*_{\widetilde{G}_{\mathsf{v}}}(\operatorname{pt})} H_*^{\widetilde{G}_{\mathsf{v}}}(\operatorname{Gr}^{\boldsymbol{\eta-\omega}}) = H_*^{\widetilde{G}_{\dv}}(\widetilde{\operatorname{Gr}}^{\boldsymbol{\omega},\boldsymbol{\eta}-\boldsymbol{\omega}}) \twoheadrightarrow H_*^{\widetilde{G}_{\dv}}(p^{-1}(\operatorname{Gr}^{\boldsymbol{\eta}}))
\end{equation}
is surjective. 
It then follows from (\ref{eq:expl formula mult in chamber}) and (\ref{eq: restriction convolution})  that the map
(\ref{eq: mult in chamber}) is also surjective. 

By induction, $H^{\widetilde{G}_{\dv}}_{\ast}(\mathcal R_{\boldsymbol{\eta}-\boldsymbol{\omega}})$ belongs to
$\operatorname{gr}\mathcal{B}$, while $H^{\widetilde{G}_{\dv}}_{\ast}(\mathcal R_{\boldsymbol{\omega}})$ belongs to
$\operatorname{gr}\mathcal{A}^{\mathrm{pos}}$ by the definition of $\mathcal{A}^{\mathrm{pos}}$. Note now that $\mathcal{B}$ is stable under the left multiplication by $\mathcal{A}^{\mathrm{pos}}$. It follows that $\operatorname{gr}\mathcal{B}$ is stable under the left multiplication by $\operatorname{gr}\mathcal{A}^{\mathrm{pos}}$.
The surjectivity of the
 map (\ref{eq: mult in chamber}) therefore implies that
$H^{\widetilde{G}_{\dv}}_{\ast}(\mathcal{R}_{\boldsymbol{\eta}})\subset \operatorname{gr}\mathcal{B}$.

This finishes the case in which $\boldsymbol{\eta}$ has a positive coordinate. Now suppose that $\boldsymbol{\eta}$ has a negative coordinate.

For each $i\in I$, set
$$
 k_i^- =\#\{s\mid \eta_{i,s}<0\},
 \qquad
 \boldsymbol\omega =\sum_{i:\,k_i^->0}\boldsymbol\omega^{*}_{i,k_i^-}.
$$

The argument is similar, now using the map $$H_*^{\widetilde{G}_{\dv}}(\mathcal{R}_{\boldsymbol{\eta}-{\boldsymbol{\omega}}}) 
\otimes 
H_*^{\widetilde{G}_{\dv}}(\mathcal{R}_{\boldsymbol{\omega}})  \rightarrow H_*^{\widetilde{G}_{\dv}}(\mathcal{R}_{{\boldsymbol{\eta}}}).$$
\end{proof}

\subsection{Hikita conjecture}
The Hikita conjecture predicts an isomorphism of $\mathbb{Z}_{\geqslant 0}$-graded algebras over $\GZ$ compatible with the natural action of $W_{\sff}$ on both sides:
\begin{equation}\label{eq: Hikita conj}
H^*_{\sF}(\widetilde{\mathcal{M}}_H) \simeq \mathbb{C}[\mathcal{M}_{C,\sff}^{\theta}],
\end{equation}
where on the RHS we consider the algebra of functions on {\emph{schematic fixed points}} (see \cite{FO}, \cite[Section 1.2]{Drinfeld} for the definition) of $\theta(\mathbb{C}^\times)$ acting on $\mathcal{M}_C$. The $\mathbb{Z}_{\geqslant 0}$-gradings on both sides of (\ref{eq: Hikita conj}) are the cohomological gradings.

It is important to mention that the action of $\GZ$ on both the left- and right-hand sides of (\ref{eq: Hikita conj}) is {\emph{cyclic}} and is induced by the natural surjections:
\begin{multline}\label{eq: kirwan quiver}
\GZ = H^*_{G_{\dv} \times \sF}(\mu^{-1}(0))=\\
=H^*_{\sF}(\mu^{-1}(0)/G_{\dv}) \twoheadrightarrow H^*_{\sF}(\mu^{-1}(0)^{\theta-\mathrm{st}}/G_{\dv})=H^*_{\sF}(\widetilde{\mathcal{M}}_H),
\end{multline}
\begin{equation}\label{eq: kirwan on coulomb}
\GZ \hookrightarrow \mathbb{C}[\mathcal{M}_C]^{\theta(\mathbb{C}^\times)} \twoheadrightarrow \mathbb{C}[\mathcal{M}_C^{\theta}],
\end{equation}
where (\ref{eq: kirwan quiver}) is surjective by \cite{kirv} and (\ref{eq: kirwan on coulomb}) is surjective by Proposition \ref{prop: generators Coulomb} combined with Lemma \ref{lem: degrees minuscule} (see \cite[Proposition 8.7]{krylov_shykov} and references therein). In particular, it follows that the isomorphism (\ref{eq: Hikita conj}) with the properties above is {\emph{unique}} if it exists. Moreover, since both (\ref{eq: kirwan quiver}) and (\ref{eq: kirwan on coulomb}) are $\mathbb{Z}_{\geqslant 0}$-graded and $W_{\sff}$-equivariant, these additional properties of the isomorphism (\ref{eq: Hikita conj}) would follow whenever one checks that it is $\GZ$-equivariant.

\begin{remark}
If the action $\mathbb{C}^\times \curvearrowright \mathcal{M}_C$ is not conical, then it is not obvious that the cohomological grading on $\mathbb{C}[\mathcal{M}_C^{\theta}]$ is nonnegative. This follows from the surjectivity of (\ref{eq: kirwan on coulomb}). 
\end{remark}

\begin{proposition} Passing to $W_{\sff}$-invariants in (\ref{eq: Hikita conj}), one would recover the following isomorphism of $\mathbb{Z}_{\geqslant 0}$-graded algebras over $H^*_{G_{\sF} \times G_{\dv}}(\operatorname{pt})$:
\begin{equation}\label{eq: hikita deformed general}
H^*_{G_{\sF}}(\widetilde{\mathcal{M}}_H) \simeq \mathbb{C}[\mathcal{M}_{C,\sff/W_{\sff}}^{\theta}].
\end{equation}
\end{proposition}
\begin{proof}
By the standard result of Borel, $H^*_{G_{\sF}}(\widetilde{\mathcal{M}}_H)=H^*_{{\sF}}(\widetilde{\mathcal{M}}_H)^{W_{\mathsf{f}}}$. 

On the Coulomb-branch side,  $\mathcal M_{C,\mathsf f}$  identifies with the pullback of $\mathcal M_{C,\mathsf f/W_{\mathsf{f}}}$ along the quotient map
$
\mathsf f \longrightarrow \mathsf{f}/W_\mathsf{f}.
$
Since the action of $\theta(\mathbb C^\times)$ commutes with the $W_\mathsf{f}$-action and is trivial on the deformation base $\mathsf{f}$, the formation of the schematic $\theta(\mathbb{C}^\times)$-fixed-point subscheme commutes with this base change. 
\end{proof}

The following proposition and its proof will appear in a forthcoming work of Kamnitzer and Weekes. 
\begin{proposition}[Kamnitzer-Weekes]\label{propos_fixed_Coulomb_fin_dim}
The set $\mathcal{M}_C^{\theta}(\mathbb{C})$ consists of {\emph{at most}} one point. In other words, $\mathbb{C}[\mathcal{M}_C^{\theta}]$ is either zero or a finite-dimensional local  algebra. 
\end{proposition}
The idea of the proof is to show that $\mathcal{M}_C^\theta(\mathbb{C}) \subset \mathcal{M}_C(\mathbb{C})$ lies in the fiber over zero of the map
\begin{equation}\label{eq: int system map}
\mathcal{M}_C \rightarrow \operatorname{Spec}\GZ
\end{equation}
induced by the embedding $\GZ \subset \mathbb{C}[\mathcal{M}_C]$.
They show this by analyzing the fibers of (\ref{eq: int system map}) via the results of \cite[Section 5]{BFN1}. The claim then follows from the fact that the composition $\mathcal{M}_C^\theta \hookrightarrow \mathcal{M}_C \rightarrow \operatorname{Spec}\GZ$ is a closed embedding (use the surjectivity of (\ref{eq: kirwan on coulomb})).

\begin{remark}
Using \cite[Theorem 9.14]{KWWY_categorical_actions}, one can show that $\mathcal{M}_C^{\theta}(\mathbb{C})$ consists of {\emph{one}} point if $I$ has no self-loops and the corresponding Higgs branch $\widetilde{\mathcal{M}}_H$ is nonempty. We are grateful to Joel Kamnitzer for explaining this to us.  Conjecturally, $\mathcal{M}_C^{\theta}(\mathbb{C})$ is a point iff $\widetilde{\mathcal{M}}_H$ is nonempty.
\end{remark}

\subsection{Quantized Hikita conjecture}
The quantized Hikita conjecture was proposed by Nakajima and predicts an isomorphism: 
\begin{equation}\label{eq: quantized Hikita}
H^*_{\sF \times \mathbb{C}^\times}(\widetilde{\mathcal{M}}_H) \simeq \mathsf{C}_{\theta}(\mathcal{A}),
\end{equation}
where $\mathsf{C}_{\theta}(\bullet)$ is the so-called {\emph{Cartan subquotient}}, also known as the $B$-algebra (see \cite[Section 5.1]{BLPW}), and is defined as follows:
\begin{equation*}
\mathsf{C}_{\theta}(\mathcal{A}) := \mathcal{A}^0/\sum_{k>0}\mathcal{A}^{-k}\mathcal{A}^k,
\end{equation*}
Here $\mathcal{A}=\bigoplus_{k \in \mathbb{Z}}\mathcal{A}^k$ is the $\mathbb{Z}$-grading induced by $\theta\colon \mathbb{C}^\times \rightarrow \mathsf{A}$.

\begin{lemma}\label{prop: surj onto cartan subquot}
The natural homomorphisms 
\begin{equation*}
\GZ \rightarrow \mathbb{C}[\mathcal{M}_{C,\mathsf{f}}^\theta], \quad\GZ_{\hbar} \rightarrow \mathsf{C}_{\theta}(\mathcal{A})
\end{equation*}
are surjective. 
\end{lemma}
\begin{proof}
This follows from Proposition \ref{PBW}.
\end{proof}

\begin{remark}
For more general Coulomb branches (namely outside of quiver gauge theories and possibly of non-cotangent type as in 
\cite{coulomb_noncotangent}), we expect that the $\hbar=0$-version of Lemma \ref{prop: surj onto cartan subquot} may fail (compare with \cite[Section 6.4]{HKM}).
\end{remark}

\begin{corollary}\label{cor: basic properties of Cartan subquot}
The algebra $\mathsf{C}_{\theta}(\mathcal{A})$ is commutative and nonnegatively graded.
\end{corollary}
\begin{proof}
Both properties follow from Lemma \ref{prop: surj onto cartan subquot} combined with the fact that $\GZ_\hbar$ is commutative and nonnegatively graded. 
\end{proof}

There is another version of a Cartan subquotient that will be useful when we start discussing the $D$-module of graded traces. 
Recall that $\mathcal{A}$ is graded by $\pi_1(G_{\mathsf{v}})$. To simplify notation, we set $R:=\pi_1(G_{\mathsf{v}})$. 
Recall the identification $R \simeq \mathbb{Z}^I$ given by (\ref{eq: ident pi 0 lattice}).
For $i\in I$, let $\alpha_i\in R\simeq\mathbb{Z}^I$ denote the $i$th standard basis vector. Under the identification $R \simeq \pi_0(\operatorname{Gr}_{G_{\mathsf{v}}})$ it corresponds to the connected component of $z^{\omega_{i,1}}$. 
Via the identification $R \simeq \operatorname{Hom}(\mathsf{A},\mathbb{C}^\times)$ we will treat $\alpha_i$ as {\emph{characters}} of the torus $\mathsf{A}$.
Then, $\langle \theta,\alpha_i\rangle=1>0$.

\begin{example}
Assume that $\Gamma$ is of type ADE and let $G_{\Gamma}$ be the adjoint group with the Lie algebra $\mathfrak{g}_{\Gamma}$. Let $T_{\Gamma}\subset G_{\Gamma}$ be a maximal torus.
Under the  identification  $\mathsf{A} \simeq T_{\Gamma}$ sending $\alpha_i$ to the corresponding simple root of $\mathfrak{g}_{\Gamma}$, the cocharacter $\theta$ corresponds to $\rho^\vee_{\Gamma}$ (the sum of the fundamental coweights of $G_{\Gamma}$).  
\end{example}

Set
\begin{equation*}
R^- := \operatorname{Span}_{\mathbb{Z}_{\geqslant 0}}(\{-\alpha_i\mid i\in I\}).
\end{equation*}

Accordingly, we have the direct-sum decomposition
\begin{equation*}
\mathcal{A} = \bigoplus_{\xi \in R}\mathcal{A}^{\xi}.
\end{equation*}
We then define
\begin{equation*}
\mathsf{C}_\theta'(\mathcal{A}) := \mathcal{A}^{\boldsymbol{0}}/\sum_{0\neq\xi \in R^-}\mathcal{A}^{\xi}\mathcal{A}^{-\xi}.
\end{equation*}

The following lemma states that the algebras $\mathsf{C}_\theta(\mathcal{A})$ and $\mathsf{C}_\theta'(\mathcal{A})$ are actually isomorphic. 
\begin{lemma}\label{lem: two cartan subquotients are iso}
The natural map 
\begin{equation}\label{eq: map between cartans}
\mathsf{C}_\theta'(\mathcal{A}) \rightarrow \mathsf{C}_\theta(\mathcal{A})
\end{equation}
is an isomorphism.    
\end{lemma}
\begin{proof}
The map $\GZ_\hbar \rightarrow \mathsf{C}_\theta(\mathcal{A})$ factors through the map $\GZ_\hbar \rightarrow \mathsf{C}'_\theta(\mathcal{A})$, so it follows from Lemma \ref{prop: surj onto cartan subquot} that the map (\ref{eq: map between cartans}) is surjective. To prove injectivity, we need to show that 
\begin{equation*}
\mathcal{A}^{\boldsymbol{0}} \cap \Big(\sum_{l>0} \mathcal{A}^{-l}\mathcal{A}^{l} \Big) = \sum_{0\neq\xi \in R^-}\mathcal{A}^{\xi}\mathcal{A}^{-\xi}.
\end{equation*}
It is clear that the RHS is contained in the LHS; let us check the other inclusion. Pick $a \in \mathcal{A}^{\boldsymbol{0}}$ and assume that $a=\sum_{j>0} b_jc_j$, where $b_j \in \mathcal{A}^{-j}$ and $c_j \in \mathcal{A}^{j}$. Decomposing $b_j$, $c_j$ into $R$-weight components, retaining only the components contributing to weight zero, we can assume that $b_j \in \mathcal{A}^{-\xi_j}$, $c_j \in \mathcal{A}^{\xi_j}$ for some $\xi_j \in R$ such that $\langle \xi_j,\theta\rangle > 0$. 

Proposition \ref{PBW} implies that we can decompose $b_j=\sum_l x_{l,-}x_{l,+}$, where $x_{l,-} \in \mathcal{A}^{\mathrm{neg}} \cap \mathcal{A}^{\zeta_l}$, $x_{l,+} \in (\mathscr{H}_\hbar \star \mathcal{A}^{\mathrm{pos}}) \cap \mathcal{A}^{-\xi_j-\zeta_l}$ for some $\zeta_l \in R^-$. Note that $\zeta_l \neq 0$ since otherwise $x_{l,+} \in \mathcal{A}^{-\xi_j}$ and the $\theta$-degree of $x_{l,+}$ is negative while the $\theta$-degree of any element in $\mathscr{H}_\hbar \star \mathcal{A}^{\mathrm{pos}}$ is nonnegative.

As in the proof of \cite[Proposition 3.8]{KMBP} we conclude that 
\begin{equation*}
b_jc_j = \sum_{l}x_{l,-}(x_{l,+}c_j) \in \sum_{0\neq\xi\in R^-}\mathcal{A}^\xi \mathcal{A}^{-\xi}
\end{equation*}
as $x_{l,-} \in \mathcal{A}^{\zeta_l}$, $x_{l,+}c_j \in \mathcal{A}^{-\zeta_l}$.
\end{proof}

As before, the identification (\ref{eq: quantized Hikita}) is expected to be $\GZ_\hbar$-equivariant (cf. \cite[Conjecture 8.10]{KTWWY}), and hence $\mathbb{Z}_{\geqslant 0}$-graded and $W_{\sff}$-equivariant.

\subsection{Category \texorpdfstring{$\mathcal{O}$}{O} for Coulomb branches}\label{ssec:cat O for CB}

Consider the decomposition of $\mathcal{A}$ according to the pairing of its eigenweights with $\theta$:
\begin{equation*}
\mathcal{A} = \mathcal{A}^+ \oplus \mathcal{A}^{0} \oplus \mathcal{A}^{-}.
\end{equation*}

Recall that $\mathcal{A}$ is an algebra over $H^*_{\sF \times \mathbb{C}^\times}(\operatorname{pt})=\mathbb{C}[\sff,\hbar]$. It is equipped with an action of $\mathsf{A}$, which is known to be inner as we now recall. Set $\mathsf{a}:=\operatorname{Lie}\mathsf{A}$. We use the natural identification 
\begin{equation}\label{eq:ident_a_with_H_2}
H^2_{G_{\mathsf{v}}}(\operatorname{pt}) \iso  (\mathfrak{g}_{\mathsf{v}}/[\mathfrak{g}_{\mathsf{v}},\mathfrak{g}_{\mathsf{v}}])^* = \mathsf{a}
\end{equation}
sending $c_1(V_i)$ to the trace along the $i$th factor.
Clearly, 
\begin{equation}\label{eq:pairing_of_chern_with_simple}
\langle \alpha_i,c_1(V_j)\rangle=\delta_{ij}.
\end{equation}

\begin{example}
It follows from (\ref{eq:pairing_of_chern_with_simple}) that for $\Gamma$ of type ADE, the basis of $\mathsf{a} \simeq \mathfrak{t}_{\Gamma}$ consisting of $c_1(V_i)$ coincides with the basis of {\emph{fundamental coweights}}.
\end{example}

The identification (\ref{eq:ident_a_with_H_2}) allows us to define the map
\begin{equation}\label{eq:our_quantum_comoment_map}
\mathsf{a} = H^2_{G_{\mathsf{v}}}(\operatorname{pt}) \hookrightarrow H^2_{\widetilde{G}_{\mathsf{v}}}(\operatorname{pt}) \hookrightarrow \mathscr{H}_\hbar \subset \mathcal{A}^{\boldsymbol{0}}, \quad  c^{G_{\mathsf{v}}}_1(V_i) \mapsto c_1^{\widetilde{G}_{\mathsf{v}}}(V_i) \cap [\mathcal{R}_{\boldsymbol{0}}]
\end{equation}
to be denoted 
\begin{equation*}
\iota_{\mathsf{a}}\colon \mathsf{a} \rightarrow \mathcal{A}^{\boldsymbol{0}}.
\end{equation*}
It follows from \cite[Lemma 3.20]{BFN1} that $\iota_{\mathsf{a}}$ is a {\emph{quantum comoment map}}, i.e.:
\begin{equation}\label{eq:moment_map_cond}
[\iota_{\mathsf {a}}(a),x]=2\hbar\langle \xi,a\rangle x, \qquad a\in \mathsf{a},\, x\in\mathcal A^\xi.
\end{equation}

\begin{warning}
Condition (\ref{eq:moment_map_cond}) does {\emph{not}} uniquely determine the comoment map $\iota_{\mathsf{a}}$. So, (\ref{eq:our_quantum_comoment_map}) is some {\emph{distinguished}} choice corresponding to the embedding  induced by the natural surjection $\widetilde{G}_{\mathsf{v}} \twoheadrightarrow G_{\mathsf{v}}$. The compatible choice appears at the quiver variety side when we choose generators of $D$, see Section \ref{quantDmodrev}. Compare with \cite[Section 5.1]{KMBP}.    
\end{warning}
    
For $f \in \sff$, we can consider the specialization $\mathcal{A}_{(f,1)}$. To simplify notation, we denote it by $\mathcal{A}_f$. In the present convention, its quantization parameter is $2$.

We set
$$
\delta=\delta_\theta:=\frac{1}{2}(\iota_{\mathsf a}\circ d\theta)(1).
$$
Thus $[\delta,x]=kx$ for $x\in\mathcal A_f^k$.

By \cite[Lemma 3.24]{catO_coulomb} combined with Proposition \ref{propos_fixed_Coulomb_fin_dim}, the Cartan subquotient $\mathsf{C}_{\theta}(\mathcal{A}_{f})$ is {\emph{finite-dimensional}}.

As in \cite[Definition 3.17]{catO_coulomb} (see also \cite[Definition 3.10]{BLPW}), one can define the category $\mathcal{O}_{\theta}(\mathcal{A}_{f})$. It follows from \cite[Lemma 3.27]{catO_coulomb} that the following is one of the equivalent definitions of $\mathcal{O}_{\theta}(\mathcal{A}_{f})$.

\begin{definition}
The category $\mathcal{O}_{\theta}(\mathcal{A}_f)$ is the full subcategory of finitely generated $\mathcal{A}_{f}$-modules whose objects satisfy the following conditions:
\begin{itemize}
    \item the action of $\mathcal{A}_f^+$ is locally nilpotent;
    \item $\delta$ acts locally finitely;
    \item the generalized eigenspaces of $\delta$ are finite-dimensional.
\end{itemize}
\end{definition}

The Grothendieck group of $\mathcal{O}_\theta(\mathcal{A}_f)$ has a distinguished basis consisting of classes of {\emph{Verma modules}} (see \cite[Lemma 3.27]{catO_coulomb}). For a one-dimensional representation $S$ of the algebra $\mathsf{C}_{\theta}(\mathcal{A}_f)$, the corresponding Verma module $\Delta(S)$ is defined by
\begin{equation*}
\Delta(S) := \mathcal{A}_f \otimes_{\mathcal{A}_f^{\geqslant 0}} S,
\end{equation*}
where $\mathcal{A}_f^{\geqslant 0}=\mathcal{A}_f^+ \oplus \mathcal{A}_f^0$ and the action $\mathcal{A}_f^{\geqslant 0} \curvearrowright S$ is induced by the natural surjection $\mathcal{A}_f^{\geqslant 0} \twoheadrightarrow \mathcal{A}_f^0 \twoheadrightarrow \mathsf{C}_{\theta}(\mathcal{A}_f)$.

\begin{lemma}
The category $\mathcal{O}_\theta(\mathcal{A}_f)$ is nonempty only if $\mathcal{M}_C^\theta(\mathbb{C})=\{p\}$.
\end{lemma}
\begin{proof}
The category $\mathcal{O}_\theta(\mathcal{A}_f)$ is empty if $\mathsf{C}_\theta(\mathcal{A}_f)=0$, since the classes of Verma modules generate $K_0(\mathcal{O}_\theta(\mathcal{A}_f))$. Note that $\mathsf{C}_\theta(\mathcal{A}_f) \neq 0$ implies that $\mathbb{C}[\mathcal{M}_C^\theta] \neq 0$ (apply \cite[Lemma 3.24]{catO_coulomb}); hence $\mathcal{M}_C^\theta(\mathbb{C})$ must contain at least one point. Proposition \ref{propos_fixed_Coulomb_fin_dim} implies that $\mathcal{M}_C^\theta(\mathbb{C})=\{p\}$.
\end{proof}

For future use, let us describe the possible weights of modules in $\mathcal{O}_\theta(\mathcal{A}_f)$. For $S \in \operatorname{Irr}\mathsf{C}_\theta(\mathcal{A}_f)$, let $\xi_S \in \mathsf{a}^*$ be the weight by which $\mathsf{a} \xrightarrow{\iota_{\mathsf{a}}} \mathcal{A}^{\boldsymbol{0}}_f \rightarrow \mathsf{C}_\theta(\mathcal{A}_f)$ acts on $S$.
\begin{lemma}\label{lem: estimate verma weights}
For any $M \in \mathcal{O}_{\theta}(\mathcal{A}_f)$, its generalized $\mathsf{a}$-weights lie in the union of the sets $\xi_S+2R^-$ for $S \in \operatorname{Irr}(\mathsf{C}_\theta(\mathcal{A}_f))$. 
\end{lemma}
\begin{proof}
It is enough to prove the claim when $M$ is a Verma module $\Delta(S)$. 
By definition, $\Delta(S)$ is a cyclic module over $\mathcal{A}_f$. It follows from Proposition \ref{PBW} that $\Delta(S)$ is a cyclic module over $\mathcal{A}_f^{\mathrm{neg}}$.
The $\mathsf{a}$-weights of $\mathcal{A}_f^{\mathrm{neg}}$ lie in $2R^-$, and the claim follows.
\end{proof}

\subsection{Various ``Verma-type'' modules in the category \texorpdfstring{$\mathcal{O}_\theta(\mathcal{A}_f)$}{OAf}}\label{sec:various verma type modules}
\subsubsection{$\Theta$-modules}

Recall that $\mathcal{O}_\theta(\mathcal{A}_f)$ has a distinguished family of modules $\Delta(S)$ labeled by one-dimensional representations of $\mathsf{C}_\theta(\mathcal{A}_f)$. If $\mathcal{M}_C$ is conical and admits a symplectic resolution $\widetilde{\mathcal{M}}_C$, then, as in \cite[Section 5.3]{BLPW}, every $\theta$-fixed point $p \in \widetilde{\mathcal{M}}_C^{\theta}$ gives rise to two types of modules. First, there is a module $\Theta(p)$ constructed using sheafified quantizations. Second, there is a distinguished homomorphism $\kappa_p\colon \mathsf{C}_\theta(\mathcal{A}_f) \rightarrow \mathbb{C}$ which allows us to define $\Delta_f(p) := \mathcal{A}_f \otimes_{\mathcal{A}_f^{\geqslant 0}} \mathbb{C}$. We believe that the homomorphism $\kappa_p$ and the module $\Theta(p)$ should also make sense for an arbitrary {\emph{nonsingular}} $\theta$-fixed point on a {\emph{partial resolution}} of $\mathcal{M}_C$ (see \cite[Section 3(ix)]{BFN1}, \cite{BFN3} for the discussion of partial resolutions of Coulomb branches). A choice of $p$ as above should correspond to a choice of an isolated torus fixed point $p^{\vee} \in \widetilde{\mathcal{M}}_H^{\mathsf{F}}$ (compare with the discussion in \cite[Section 5.4.2]{quiver_slant}). It would be interesting to rigorously construct these more general modules $\Theta(p)$.

\subsubsection{Modules obtained via comultiplication}\label{ssec: modules obtained via comult}

For $\Gamma$ of type ADE, let $\alpha_i$ and $\omega_i$ denote the simple roots and fundamental weights of $\mathfrak g_\Gamma$. Set
\begin{equation}\label{vwtoweights}
\lambda=\sum_{i\in I}\dw_i\omega_i,
\qquad
\mu=\lambda-\sum_{i\in I}\dv_i\alpha_i.
\end{equation}
By \cite[Theorem~3.10]{BFN19}, the Coulomb branch $\mathcal{M}_C$ is identified, in our conventions, with the generalized affine Grassmannian slice $\overline{\mathcal{W}}^{\la}_\mu$. For $\sF=\sF_{\mathsf w}$, there exists another natural collection of modules labeled by nonsingular fixed points on partial resolutions of $\mathcal{M}_C$. Namely, a choice of a partial resolution of $\mathcal{M}_C=\overline{\mathcal{W}}^{\la}_\mu$ corresponds to a choice $\underline{\la}$ of a decomposition $\la=\la_1+\ldots+\la_\ell$. Nonsingular torus fixed points on $\widetilde{\mathcal{W}}^{\underline{\la}}_\mu$ are in bijection with $\ell$-tuples $\underline{\mu}=(\mu_k)_{k=1,\ldots,\ell}$ such that $\mu_k \in W\lambda_k$ and $\sum_{k=1}^\ell \mu_k=\mu$ (this, for example, follows from \cite[Propositions 5.7 and  5.9]{KrPer}).
We thus obtain modules $M_{\mathsf f}(\underline{\mu})$ whose specializations $M_f(\underline{\mu})$ lie in $\mathcal{O}_{\theta}(\mathcal{A}_f)$ and are labeled by the tuples $\underline{\mu}$ above. These modules are constructed as follows.

Recall that $\mathcal{A}_{\mathsf{f}}$ is the quantization of the Coulomb branch corresponding to $(\Gamma,\mathsf{v},\mathsf{w})$. The data $(\mathsf{v},\mathsf{w})$ are equivalent to the data $(\lambda,\mu)$, so let us denote $\mathcal{A}_{\mathsf{f}}$ by $\mathcal{A}_{\mathsf{f}}(\lambda,\mu)$. The decomposition $\lambda=\la_1+\ldots+\la_\ell$ determines a product decomposition $\sF=\sF^{(1)} \times \ldots \times \sF^{(\ell)}$, and hence $\mathsf f=\mathsf f^{(1)}\oplus\cdots\oplus\mathsf f^{(\ell)}$.
The modules $M_{\mathsf f}(\underline{\mu})$ are constructed using the so-called {\emph{comultiplication}} maps:
\begin{equation}\label{comult_maps}
\mathcal{A}_{\mathsf{f}}(\lambda,\mu) \rightarrow
\mathcal{A}_{\mathsf{f}^{(1)}}(\lambda_1,\mu_1)
\otimes_{\mathbb{C}[\hbar]}
\mathcal{A}_{\mathsf{f}^{(2)}}(\lambda_2,\mu_2)
\otimes_{\mathbb{C}[\hbar]} \ldots
\otimes_{\mathbb{C}[\hbar]}
\mathcal{A}_{\mathsf{f}^{(\ell)}}(\lambda_\ell,\mu_\ell).
\end{equation}
The existence of such a map uses a realization of $\mathcal{A}$ as truncated shifted Yangians. The corresponding coproduct for shifted Yangians was constructed in \cite[Section 4]{comult_shifted}.
It is a nontrivial theorem that this coproduct factors through the truncations, see \cite[Theorem 2 and Section 5.4]{KLLPW}. We do not recall the definitions of these maps here.

\begin{warning}
The map in (\ref{comult_maps}) is {\emph{not}} unique: it depends on the choice of iterated coproduct and its parenthesization. This is related to the fact that comultiplication maps between shifted Yangians are not coassociative in general; see \cite[Remark 4.15]{comult_shifted}.
\end{warning}

By definition,
\begin{equation}\label{def_M_module}
M_{\mathsf{f}}(\underline{\mu}) :=
L_{\mu_1,\mathsf{f}^{(1)}} \otimes_{\mathbb{C}[\hbar]}
L_{\mu_2,\mathsf{f}^{(2)}} \otimes_{\mathbb{C}[\hbar]} \ldots
\otimes_{\mathbb{C}[\hbar]}
L_{\mu_\ell,\mathsf{f}^{(\ell)}}
\end{equation}
where $L_{\mu_k,\mathsf{f}^{(k)}}$ are certain distinguished modules over $\mathcal{A}_{\mathsf{f}^{(k)}}(\la_k,\mu_k)$ with the action of $\mathcal{A}_{\mathsf{f}}(\la,\mu)$ induced by (\ref{comult_maps}). The modules $L_{\mu_k,\mathsf{f}^{(k)}}$ are defined as follows. The condition $\mu_k \in W\lambda_k$ implies that the natural map
\begin{equation*}
\mathbb{C}[\mathsf{f}^{(k)},\hbar] \iso \mathsf{C}_\theta(\mathcal{A}_{\mathsf{f}^{(k)}}(\la_k,\mu_k))
\end{equation*}
is an isomorphism (compare with \cite[Section 6]{quiver_slant}). Thus, we obtain an action $\mathcal{A}_{\mathsf{f}^{(k)}}(\la_k,\mu_k)^{\geqslant 0} \curvearrowright \mathbb{C}[\mathsf{f}^{(k)},\hbar]$. We then define
\begin{equation*}
L_{\mu_k,\mathsf{f}^{(k)}} := \mathcal{A}_{\mathsf{f}^{(k)}}(\la_k,\mu_k) \otimes_{\mathcal{A}_{\mathsf{f}^{(k)}}(\la_k,\mu_k)^{\geqslant 0}} \mathbb{C}[\mathsf{f}^{(k)},\hbar].
\end{equation*}
Note that $L_{\mu_k,\mathsf{f}^{(k)}}$ is free as a module over $\mathbb{C}[\mathsf{f}^{(k)},\hbar]$, with normalized character given by \cite[Equation (22)]{quiver_slant}. The argument is the same as in the proof of \cite[Proposition 6.11]{quiver_slant} and also follows from \cite{min_chamber_mod}.

\begin{remark}
The condition $\mu_k \in W\lambda_k$ implies that, for every $f \in \mathsf{f}^{(k)}$, the category $\mathcal{O}_\theta(\mathcal{A}_f(\la_k,\mu_k))$ contains a {\emph{unique}} irreducible object, the so-called ``extremal'' irreducible module over the shifted Yangian $Y_{\mu_k}$; see \cite{min_chamber_mod} or \cite[Section 6]{quiver_slant} for details. This object is precisely the specialization of $L_{\mu_k,\mathsf{f}^{(k)}}$ at $(f,1)$.
\end{remark}

\subsubsection{Relation between various modules}\label{eq: relation between modules}

Let us assume that $\overline{\mathcal{W}}^\lambda_\mu$ has a symplectic resolution $\widetilde{\mathcal{W}}^{\underline{\lambda}}_\mu$ and $\mu$ is dominant.
The modules $M_f(\underline{\mu})$, $\Theta_f(\underline{\mu})$ are then well-defined and have the same {\emph{normalized character}} (see Definition \ref{eq:def normalized graded trace} below) equal to the character of $S^\bullet(T^-_{\underline{\mu}}\widetilde{\mathcal{W}}^{\underline{\la}}_\mu)$, where $T^-_{\underline{\mu}}\widetilde{\mathcal{W}}^{\underline{\la}}_\mu$ is the $\theta$-negative part of the tangent space at $\underline{\mu}$. We do not know if $\Delta_f(\underline{\mu})$ has the same character. Note that $\Delta_f(\underline{\mu})$ is flat (considered as a family depending on $f$) whenever $f$ is outside of the {\emph{finite}} number of shifts of classical walls (compare with \cite[Proposition 4.15(2)]{LosevmodcatO}).
It is natural to expect that outside of the finite number of shifts of classical walls all of these modules are isomorphic. We expect that this happens whenever $\mathsf{C}_\theta(\mathcal{A}_f)$ is isomorphic to the direct sum of copies of $\mathbb{C}$ equal to the number of $\theta$-fixed points on $\widetilde{\mathcal{W}}^{\underline{\lambda}}_\mu$. We will call $f$ for which this fails {\emph{singular}}. 
Under (\ref{eq: quantized Hikita}),  singular $f$ correspond to the parameters $f \in \mathsf{f}$ such that zeroes $\widetilde{\mathcal{M}}_{H,\theta}^{(f,1)}$ of the vector field $(f,1) \in \mathsf{f} \oplus \mathbb{C}$ are {\emph{not}} equal to $\widetilde{\mathcal{M}}_{H,\theta}^{\mathsf{F} \times \mathbb{C}^\times_\hbar}$.

\begin{remark}
Note that by combining \cite[Conjectures 1.3.1, 1.3.2, 1.6.3]{Losev-loc} one sees that the singular hyperplanes as above should  correspond to the quantization parameters $f$ for which the {\emph{derived localization theorem}} fails.  
\end{remark}

Note that at least if $f$ is {\emph{very}} generic  (namely, after we delete a {\emph{countable}} number of shifts of classical walls, compare with \cite[Section 4.1.1]{LosevmodcatO}), 
all of the modules above become isomorphic and irreducible (compare with \cite[Proposition 4.13]{LosevmodcatO}). It is an interesting question to understand the precise relation between these modules at singular parameters. Lacking better terminology, we will call them all {\emph{Verma-type}} modules. We will use modules $M_f(\underline{\mu})$ for the actual computations needed for our purposes.

\subsection{Graded traces}\label{sec:graded traces}
Recall the identification $\mathsf{a} \simeq H^2_{G_{\mathsf{v}}}(\operatorname{pt})$, and set $\mathfrak C:=\mathbb C[\mathsf f,\hbar]$.

From this point onward, the Novikov variables used in comparisons with
graded traces are the sign-modified variables
$$
\boldsymbol{z}_i^{\mathrm{tr}}
:=
(-1)^{a_i}\boldsymbol{z}_i^{\mathrm{geom}},
\qquad
a_i=
\dw_i+
\sum_{\substack{e\\h(e)=i}}\dv_{t(e)}
-
\sum_{\substack{e\\t(e)=i}}\dv_{h(e)},
\qquad i\in I.
$$
We suppress the superscripts and write $\boldsymbol{z}_i$ for
$\boldsymbol{z}_i^{\mathrm{tr}}$.

Let $\mathbb C[{\boldsymbol{z}}]:=\mathbb C[R^-]$ be the corresponding monoid algebra, with its basis elements written as $z^\xi$, $\xi \in R^-$. We also set $\mathbb{C}[\boldsymbol{z}]_\hbar := \mathbb C[2\hbar R^-]$.
Let
\begin{equation}\label{eq:augm_ideal_R_-}
\mathfrak{m}:=(z^\xi\mid 0\neq\xi\in R^-) \subset \mathbb{C}[{\boldsymbol{z}}], \quad 
\mathfrak{m}_{\hbar}:=(z^\xi\mid 0\neq\xi\in 2\hbar R^-) \subset \mathbb{C}[{\boldsymbol{z}}]_\hbar
\end{equation}
be the corresponding  augmentation ideals.

\begin{definition}\label{[z]}
Let $V$ be a graded vector space and set $V[{\boldsymbol{z}}]:=V\otimes_{\mathbb C}\mathbb C[{\boldsymbol{z}}]$. We denote
$$
V[[{\boldsymbol{z}}]]:=\varprojlim_{n\geqslant 1}^{\mathrm{gr}}V[{\boldsymbol{z}}]/\mathfrak{m}^nV[{\boldsymbol{z}}],
$$
where the inverse limit is taken in the category of graded vector spaces.
\end{definition}

Recall the module $M_{\mathsf{f}}(\underline{\mu})$ defined in (\ref{def_M_module}). The module itself depends both on the ordering of the tuple $\underline{\mu}$ and on the choice of iterated coproduct in (\ref{comult_maps}). The key point is that the {\emph{graded trace}} of $M_{\mathsf f}(\underline{\mu})$ does not depend on these choices (see (\ref{eq:gr_tr_M_mu}) below). Let us define this graded trace. Recall the commutative subalgebra $\mathscr{H}_{\hbar} \subset \mathcal{A}$. It contains $\mathsf{a}=H^2_{G_{\mathsf{v}}}(\operatorname{pt})$ as a commutative Lie subalgebra (see (\ref{eq:our_quantum_comoment_map})), and hence $\mathsf a$ acts on $M_{\mathsf f}(\underline\mu)$. Write
$$
\xi_{\underline\mu}\in
\operatorname{Hom}(\mathsf a,\mathsf f^*\oplus\mathbb C\hbar)
$$
for its highest weight. 
Its generalized $\mathsf a$-weights lie in $\xi_{\underline\mu}+2\hbar R^-$. The graded trace of $M_{\mathsf{f}}(\underline{\mu})$ is the $\mathfrak C$-linear map
$$
\operatorname{tr}_{M_{\mathsf f}(\underline\mu)}\colon
\mathcal A^0\longrightarrow
z^{\xi_{\underline\mu}}\mathfrak C[[{\boldsymbol{z}}]]_{\hbar}
$$
given by
\begin{equation}\label{def_gr_tr_of_module}
\operatorname{tr}_{M_{\mathsf f}(\underline\mu)}(a)
=\sum_{\eta\in\xi_{\underline\mu}+2\hbar R^-}
\operatorname{tr}_{\mathfrak C}\bigl(a|_{M_\eta}\bigr)z^\eta.
\end{equation}

We will show later (see Lemma \ref{lem: H to GrTr is surjective}) that $\operatorname{tr}_{M_{\mathsf f}(\underline\mu)}$ is uniquely determined by its restriction to $\mathscr{H}_\hbar$.

Note that the definition (\ref{def_gr_tr_of_module}) makes sense for any $\mathcal{A}$-module that is free over $\mathfrak C$, decomposes as  the direct sum of its generalized $\mathsf{a}$-eigenspaces, and has eigenspaces of finite rank over $\mathfrak C$. The definition also makes sense for $\mathsf{a}$-locally-finite modules over $\mathcal{A}_{f}$; in particular, the graded trace of any object in $\mathcal{O}_{\theta}(\mathcal{A}_f)$ is well-defined. It follows from the proof of Lemma \ref{lem: estimate verma weights} that the graded trace of any Verma module $\Delta(S)$ defines a functional 
\begin{equation*}
\operatorname{tr}_{\Delta(S)}\colon \mathcal{A}_f^0 \rightarrow z^{\xi_S}\mathbb{C}[[{\boldsymbol{z}}]].
\end{equation*}
For $\Gamma$ without loops, identifying $\mathcal{A}_{f}$ with an appropriate truncated shifted Yangian (see \cite[Section 4]{cat_O_slices_and_categor}) and restricting to the corresponding Cartan subalgebra, one sees that the restriction of the graded trace to $\mathscr{H}_f \subset \mathcal{A}_{f}^0$ carries the same information as the Frenkel--Reshetikhin $q$-character of the corresponding module; this construction goes back to \cite{FRq}.

Recall that $\mathscr{H}_{\hbar} = H^*_{G_{\mathsf{v}} \times \mathbb{C}^\times \times \mathsf{F}}(\operatorname{pt})$. For $i \in I$ and $1 \leqslant p \leqslant \dv_i$ we will denote by $A_i^{(p)} \in \mathcal{A}$ the element $c_p(V_i) \in \mathscr{H}_\hbar$. 

For $\Gamma$ of type ADE, it follows from \cite[Theorem 3.14]{hernandez_zhang}, together with the formula for elementary symmetric polynomials under disjoint unions of variables, that tensor products are compatible with trace functionals in the following sense.

\begin{proposition}\label{tensorprod}
For $i\in I$, let
$$
\mathbf A_i(u):=1+\sum_{p\geqslant 1}A_i^{(p)}u^{-p} \in 1+\mathcal{A}^0[u^{-1}],
$$
where $u$ is a formal parameter. Then
$$
\operatorname{tr}_{M_1\otimes M_2}(\mathbf A_i(u))
=\operatorname{tr}_{M_1}(\mathbf A_i(u))
\operatorname{tr}_{M_2}(\mathbf A_i(u)).
$$
\end{proposition}
In particular, 
\begin{equation}\label{eq:gr_tr_M_mu}
\operatorname{tr}_{M_{\mathsf{f}}(\underline{\mu})}{\bf{A}}_i(u)=\prod_{k=1}^{\ell}\operatorname{tr}_{L_{\mu_k,\mathsf{f}^{(k)}}}{\bf{A}}_i(u).
\end{equation}

\subsection{\texorpdfstring{$D$}{D}-module of graded traces and  the quantum Hikita conjecture}\label{ssec: D mod of graded traces}

In \cite{KMBP}, Kamnitzer, McBreen, and Proudfoot generalized the conjecture of~\eqref{eq: quantized Hikita} to the quantum setting. Let us formulate their conjecture in our convention. On the LHS, they propose to consider the $q=2\hbar$ specialization of Givental's {\emph{quantum}} $D$-module $Q^{\mathrm{Giv}}$ of $\widetilde{\mathcal{M}}_H$. As a vector space,
\begin{equation*}
Q^{\mathrm{Giv}} = H^*_{\mathsf{T}}(\widetilde{\mathcal{M}}_H)[[{\boldsymbol{z}}]][q],
\end{equation*}
where $[[\boldsymbol{z}]]$ has the meaning of Definition~\ref{[z]} but with the monoid $R^{-}$ replaced by the collection of \emph{effective} curve classes $d \in H_2(\widetilde{\mathcal{M}}_H,\mathbb{Z})$ and with the grading given by cohomological degree. As a set, the algebra $D=D^{\mathrm{Giv}}$ acting on $Q=Q^{\mathrm{Giv}}$ is
\begin{equation}\label{eq: giventals algebra D}
D^{\mathrm{Giv}}=S^\bullet(H^2_{\mathsf{T}_q}(\widetilde{\mathcal{M}}_H))[[{\boldsymbol{z}}]]
\end{equation}
where $\mathbb{C}^{\times}_{q}$ acts trivially on $\widetilde{\mathcal{M}}_{H}$. The multiplication is given by 
\begin{equation*}
(1 \otimes c)(z^d \otimes 1) = q\langle \bar{c},d\rangle z^d \otimes 1 + z^d \otimes c,
\end{equation*}
where $\bar c$ denotes the image of $c$ in $H^2(\widetilde{\mathcal M}_H)$ and $\langle \bar{c},d\rangle$ is the natural homology--cohomology pairing. The action of $D^{\mathrm{Giv}}$ on $Q^{\mathrm{Giv}}$ is given by the formula
\begin{equation}\label{eq: action D Giv on Q Div}
z^d \mapsto z^d \cdot -,~ c \mapsto q\partial_{\bar{c}}+(c \star -),
\end{equation}
where $\star$ is the {\emph{quantum}} multiplication on $H^*_{\mathsf{T}}(\widetilde{\mathcal{M}}_H)[[{\boldsymbol{z}}]]$ defined using genus-zero Gromov--Witten invariants, and $\partial_{\bar{c}}$ is the ``Euler'' vector field given by $\partial_{\bar{c}}(z^d)=\langle {\bar{c}},d\rangle z^d$.

Here and below, we use the plus-sign convention $q\partial_{\bar c}+(c\star-)$ for the quantum connection. This differs from the convention in \cite[Section 4.2]{KMBP}.

Now, $Q_{q=2\hbar}^{\mathrm{Giv}}$ is the specialization of $Q$ defined by imposing $q=2\hbar$, where $\hbar$ is the primitive equivariant parameter coming from the $\mathbb{C}^\times$-action. We will still call $Q^{\mathrm{Giv}}_{q=2\hbar}$ a ``$D$-module'' but note that the action of $D$ on it factors through the specialization $D_{q=2\hbar}$.

On the RHS of the quantum Hikita conjecture, the authors propose to consider the $D$-module of {\emph{graded traces}}. 
We denote this $D$-module by $\operatorname{GrTr}(\mathcal{A})$; it is defined as follows.

We first introduce, following \cite[Section 3.3]{KMBP},
\begin{equation}\label{eq:def_ideal_J}
J := \sum_{\xi \in R^-}\BC[{\boldsymbol{z}}] \cdot \{1 \otimes ab - z^\xi \otimes ba,\, a \in \mathcal{A}^{\xi},\, b \in \mathcal{A}^{-\xi}\} \subset \mathcal{A}^0[{\boldsymbol{z}}],
\end{equation}
\begin{equation*}
\on{GrTr}^{\mathrm{rat}}(\mathcal{A}) :=\mathcal{A}^{0}[{\boldsymbol{z}}]/J.
\end{equation*}
The superscript $\mathrm{rat}$ stands for ``rational''. One can define $\operatorname{GrTr}^{\mathrm{rat}}(\mathcal{A}_f)$ by the same formula. 

\begin{lemma}\label{lem:spec_commutes_with_GrTr_rat}
The specialization homomorphism $\mathcal{A} \twoheadrightarrow \mathcal{A}_f$ induces the isomorphism $\operatorname{GrTr}^{\mathrm{rat}}(\mathcal{A})|_{(f,1)} \iso \operatorname{GrTr}^{\mathrm{rat}}(\mathcal{A}_f)$.    
\end{lemma}
\begin{proof}
Specialization is right exact so $\operatorname{GrTr}^{\mathrm{rat}}(\mathcal{A})|_{(f,1)} = \mathcal{A}^0_f[{\boldsymbol{z}}]/J_f$, where $J_f$ is the image of $J$ in $\mathcal{A}_f$. The claim follows.
\end{proof}

We now define the completion

\begin{equation}\label{eq:def_of_grtr}
\on{GrTr}(\mathcal{A}) := \varprojlim_{n\geqslant 1}^{\mathrm gr}\Big(\on{GrTr}^{\mathrm{rat}}(\mathcal{A})/\mathfrak{m}^n\on{GrTr}^{\mathrm{rat}}(\mathcal{A})\Big),
\end{equation}
where $\mathfrak{m}$ is the augmentation ideal defined in \eqref{eq:augm_ideal_R_-} and the limit is taken in the category of $\mathbb{Z}$-graded modules. The grading on $\operatorname{GrTr}^{\mathrm{rat}}(\mathcal{A})$ is induced by the {\emph{cohomological}} grading on $\mathcal{A}^0$  and $\operatorname{deg}z^\xi=0$. The individual terms in (\ref{eq:def_of_grtr}) will be denoted by $\operatorname{GrTr}(\mathcal{A})_n$. 

One can define $\operatorname{GrTr}^{\mathrm{rat}}(\mathcal A_f)$ by the
same formula. More precisely, the cohomological grading on $\mathcal A$ induces an exhaustive increasing
filtration on $\mathcal A_f$; placing each $z^\xi$ in filtration degree
zero, we endow $\operatorname{GrTr}^{\mathrm{rat}}(\mathcal A_f)$ and its
$\mathfrak m$-adic quotients with the corresponding quotient filtrations.
We define
$$
  \operatorname{GrTr}(\mathcal A_f)
  :=
  \varprojlim_{n\geqslant 1}^{\mathrm{fil}}
  \left(
    \operatorname{GrTr}^{\mathrm{rat}}(\mathcal A_f)/
    \mathfrak m^n\operatorname{GrTr}^{\mathrm{rat}}(\mathcal A_f)
  \right),
$$
where the inverse limit is taken in the
category of exhaustively filtered vector spaces.

\begin{lemma}\label{lem:finite-level GrTr}
For every $n\geqslant 1$, the natural homomorphism
$$
\GZ_\hbar[{\boldsymbol z}]
 \big/\mathfrak m^n\GZ_\hbar[{\boldsymbol z}]
\longrightarrow
\operatorname{GrTr}(\mathcal A)_n
$$
is surjective. In particular, every cohomological graded component of
$\operatorname{GrTr}(\mathcal A)_n$ is finite-dimensional.
\end{lemma}

\begin{proof}
Proposition \ref{PBW} implies that
every $x\in\mathcal A^{\boldsymbol0}$ can be written as
$$
x=h+\sum_j a_jb_j,
\qquad
h\in\GZ_\hbar,\quad
a_j\in\mathcal A^{\xi_j},\quad
b_j\in\mathcal A^{-\xi_j},\quad
0\neq\xi_j\in R^-.
$$
In $\operatorname{GrTr}(\mathcal A)_n$ we have
$$
[x]=[h]+\sum_j [z^{\xi_j} \otimes b_ja_j].
$$
Applying the same decomposition to each $b_ja_j$ and iterating $n$
times, the remaining terms acquire a factor in $\mathfrak m^n$ and
therefore vanish. This proves the surjectivity.

Finally, $\mathbb C[{\boldsymbol z}]/\mathfrak m^n$ is finite-dimensional,
and $\GZ_\hbar$ has finite-dimensional cohomological graded components.
The last assertion  follows.
\end{proof}

\begin{lemma}\label{lem:specialization}
\begin{enumerate}[label=(\arabic*)]

\item For $\hbar=0$, we have a natural isomorphism of $\mathbb{C}[\mathsf f][[{\boldsymbol{z}}]]$-modules
\begin{equation*}
\operatorname{GrTr}(\mathcal{A})_{\hbar=0} \simeq \mathbb{C}[\mathcal{M}_{C,\mathsf f}^\theta][[{\boldsymbol{z}}]].
\end{equation*}

\item For $z=0$, we have a natural isomorphism of $\mathfrak C$-modules
\begin{equation}\label{eq:descr_grtr_hbar_0}
\operatorname{GrTr}(\mathcal{A})/\mathfrak{m}\operatorname{GrTr}(\mathcal{A}) \simeq\mathsf C_\theta(\mathcal A).
\end{equation}
\item If $\operatorname{GrTr}(\mathcal{A})_n$ is {\emph{flat}} over $\mathfrak{C}$ for every $n$, then there is a natural isomorphism $\operatorname{GrTr}(\mathcal{A})|_{(f,1)} \iso \operatorname{GrTr}(\mathcal{A}_f)$. 
\end{enumerate}
\end{lemma}
\begin{proof}
We claim that the natural morphism
\begin{equation}\label{eq:comp_morph_lim_hbar}
\operatorname{GrTr}(\mathcal{A})_{\hbar=0} \rightarrow \varprojlim_{n\geqslant 1}^{\mathrm gr} \Big((\operatorname{GrTr}(\mathcal{A})_n)_{\hbar=0}\Big)
\end{equation}
is an isomorphism. 

Fix a cohomological degree $r \in \mathbb{Z}$ and let $K_{n,r}$, $C_{n,r}$ be the degree $r$ components of the kernel and cokernel of 
\begin{equation*}
 \cdot \hbar \colon \operatorname{GrTr}(\mathcal{A})_n \rightarrow  \operatorname{GrTr}(\mathcal{A})_n. 
\end{equation*}
By Lemma \ref{lem:finite-level GrTr}, both $K_{n,r}$ and  $C_{n,r}$  are finite-dimensional. Thus the inverse systems $K_{n,r}$ and
$C_{n,r}$ ($r$ is fixed) satisfy the Mittag--Leffler condition. So, the morphism (\ref{eq:comp_morph_lim_hbar}) is indeed an isomorphism.

It remains to construct an isomorphism 
\begin{equation}\label{eq:identification_want}
(\operatorname{GrTr}(\mathcal{A})_n)_{\hbar=0} \simeq \mathbb{C}[\mathcal{M}^{\theta}_{C,\mathsf{f}}] \otimes (\mathbb{C}[[{\boldsymbol{z}}]]/\mathfrak{m}^n).
\end{equation}
Clearly,  
\begin{equation*}
(\operatorname{GrTr}(\mathcal{A})_n)_{\hbar=0} = \operatorname{GrTr}(\mathcal{A}_{\hbar=0})_n,
\end{equation*} 
compare with the proof of Lemma \ref{lem:spec_commutes_with_GrTr_rat}. Since $\mathcal{A}_{\hbar=0}$ is commutative, the relation in (\ref{eq:def_ideal_J}) corresponding to $\xi \in R^-$ becomes
\begin{equation*}
1 \otimes ab - z^{\xi} \otimes ba = (1-z^\xi) \otimes ab.
\end{equation*}
For $\xi={\boldsymbol{0}}$ this relation is zero. Otherwise, $1-z^{\xi}$ is invertible. The existence of the isomorphism (\ref{eq:identification_want}) follows.

For part~(2), we have: 
$$
\operatorname{GrTr}(\mathcal A)_{z=0}
\simeq
\frac{\mathcal A^{\boldsymbol0}}
{[\mathcal A^{\boldsymbol0},\mathcal A^{\boldsymbol0}]
 +\sum_{{\boldsymbol{0}}\neq\xi\in R^-}\mathcal A^\xi\mathcal A^{-\xi}}.
$$
The commutator relations are redundant because
$\mathsf C'_\theta(\mathcal A)\simeq\mathsf C_\theta(\mathcal A)$ is
commutative (Corollary \ref{cor: basic properties of Cartan subquot}). The result therefore follows from
Lemma~\ref{lem: two cartan subquotients are iso}.

It remains to prove part~(3). By right-exactness of specialization and
Lemma \ref{lem:spec_commutes_with_GrTr_rat}, for every $n\geqslant 1$ there is a natural isomorphism of $\mathbb{Z}_{\geqslant 0}$-filtered
vector spaces
$$
  \left.\operatorname{GrTr}(\mathcal A)_n\right|_{(f,1)}
  \xrightarrow{\sim}
  \operatorname{GrTr}^{\mathrm{rat}}(\mathcal A_f)/
  \mathfrak m^n\operatorname{GrTr}^{\mathrm{rat}}(\mathcal A_f).
$$
Since the transition morphisms are strict surjections, the flatness
hypothesis and the Mittag--Leffler exact sequence applied successively to
the finite regular sequence defining $(f,1)$ show that specialization
commutes with this inverse limit in the category of exhaustively filtered
vector spaces, and the claim follows from the preceding definition.
\end{proof}

\begin{remark}
It is natural to conjecture that $\operatorname{GrTr}(\mathcal{A})_n$ is  always free over $\mathfrak{C}=\mathbb{C}[\mathsf{f},\hbar]$. We will prove (see Lemma \ref{lem: freequot}) that this is indeed the case whenever Assumption~\ref{ass:our_main_ass_modules_dim_est} below holds.
\end{remark}

It is easy to see that there is a natural morphism
\begin{equation}\label{eq: master D mod to GrTr surjective}
\mathfrak V\colon \GZ_{\hbar}[[{\boldsymbol{z}}]] \rightarrow \on{GrTr}(\mathcal{A}),~z^d \otimes \tau \mapsto [z^d \otimes \tau]
\end{equation}
which is a homogeneous homomorphism of $D_{q=2\hbar}$-modules.

\begin{lemma}\label{lem: H to GrTr is surjective}
The homomorphism (\ref{eq: master D mod to GrTr surjective}) is surjective; in particular, the $\mathbb Z$-grading on the target is supported in $\mathbb Z_{\geqslant 0}$.
\end{lemma}
\begin{proof}

It follows from Lemma \ref{lem:finite-level GrTr} that 
$\operatorname{GrTr}(\mathcal{A})_n$
is $\mathbb Z_{\geqslant 0}$-graded, and hence so is $\operatorname{GrTr}(\mathcal A)$.

Thus, one can apply the graded Nakayama lemma to 
$\operatorname{GrTr}(\mathcal{A})$ considered as a $\mathbb{Z}_{\geqslant 0}$-graded module over the nonnegatively graded ring $\mathbb{C}[\hbar]$
and finish the proof using Lemma \ref{prop: surj onto cartan subquot} combined with the identification (\ref{eq:descr_grtr_hbar_0})
from Lemma~\ref{lem:specialization} (2).
\end{proof}

\begin{lemma}\label{lem: GrTr fin gen}
The $\mathfrak C[[{\boldsymbol{z}}]]$-module $\on{GrTr}(\mathcal{A})$ is finitely generated.
\end{lemma}
\begin{proof}
Recall that the $\mathbb{Z}$-grading on $\mathcal{A}$ induces the $\mathbb{Z}_{\geqslant 0}$-grading on $\on{GrTr}(\mathcal{A})$. Thus, by the graded Nakayama lemma, it is enough to show that $\operatorname{GrTr}(\mathcal{A})_{\mathsf f=\hbar=0} = \mathbb{C}[\mathcal{M}_C^\theta] \otimes \mathbb{C}[[\boldsymbol{z}]]$ is finitely generated over $\mathbb{C}[[{\boldsymbol{z}}]]$. This follows from 
Proposition \ref{propos_fixed_Coulomb_fin_dim}.
\end{proof}

For a parameter-valued weight 
\begin{equation*}
\xi \in \operatorname{Hom}_{\mathbb{C}}(\mathsf{a},\mathsf{f}^* \oplus \mathbb{C}\hbar)
\end{equation*}
define
\begin{equation}\label{eq:def of Z xi}
\mathcal{Z}_{\xi} := z^\xi\mathbb{C}[\mathsf{f},\hbar][[2\hbar R^-]]= \Big\{ \sum_{\eta \in \xi+2\hbar R^-} z^{\eta} \otimes p_\eta\,\Big|\, p_\eta \in \mathbb{C}[\mathsf{f},\hbar]\Big\}.
\end{equation}
The space $\mathcal{Z}_{\xi}$ has a natural $D_{q=2\hbar}$-module structure (compare with \cite[Remark 3.16]{KMBP}). More generally, 
to a finite collection $\Xi=\{\xi_1,\ldots,\xi_s\}$, associate $\mathcal{Z}_{\Xi}$ defined in the same way as in (\ref{eq:def of Z xi}) with weights $\eta$ being allowed to live in $\xi_j+2\hbar R^-$ for $j=1,\ldots,s$.

Given a module $M=\bigoplus_{\xi+\eta} M_{\xi+\eta}$ whose generalized $\mathsf a$-weight spaces $\xi+\eta$ are contained in the finite union
\begin{equation*}
\bigcup_{j=1}^s (\xi_j+2\hbar R^-)
\end{equation*}
we consider the map
\begin{equation*}
\operatorname{tr}_M \colon \operatorname{GrTr}(\mathcal{A}) \rightarrow \mathcal{Z}_{\Xi}, \quad x \mapsto  \sum_{\eta \in \bigcup_j \xi_j+2\hbar R^-} z^{\eta} \otimes \operatorname{tr}_{M_{\eta}}(x|_{M_{\eta}}).
\end{equation*}

The following lemma holds by \cite[Proposition 3.15]{KMBP}.
\begin{lemma}
The map $\operatorname{tr}_M$ is a homomorphism of $D_{q=2\hbar}$-modules. 
\end{lemma}

The following Conjecture was proposed by Kamnitzer, McBreen and Proudfoot. We state it in our convention. They proved it for hypertoric varieties and for the Springer resolution. In \cite{CHY-q}, the authors proved it for the pair $(X \rightarrow Y; Y^{\vee})$, where $Y^\vee$ is the closure of the minimal nilpotent orbit in a semisimple Lie algebra $\mathfrak{g}$ and $X$ is a resolution of the corresponding Kleinian singularity $Y$.  
\begin{conjecture}[{\cite[Conjecture 1.1]{KMBP}}]\label{conj: original quantum hikita}
If $\mathcal{M}_C$ is good (cohomological grading is conical), then there exists an isomorphism of $D$-modules
\begin{equation}\label{eq: original quantum Hikita}
Q_{q=2\hbar}^{\mathrm{Giv}} \simeq \operatorname{GrTr}(\mathcal{A})
\end{equation}
sending $1$ to $1$.
\end{conjecture}

\subsection{Refinement of the quantum Hikita conjecture}

We propose a refinement of Conjecture \ref{conj: original quantum hikita}. In (\ref{eq: original quantum Hikita}), $\operatorname{GrTr}(\mathcal{A})$ is compared with the $q=2\hbar$ specialization of Givental's quantum $D$-module $Q^{\mathrm{Giv}}$, defined using curve counts into $\widetilde{\mathcal{M}}_H$. We propose to {\emph{replace}} this $D$-module by the Pushkar--Smirnov--Zeitlin $D$-module (or simply the PSZ $D$-module), defined using curve counts into the {\emph{stacky}} quotient $\mathfrak{X}:=[\mu^{-1}(0)/G_{\mathsf{v}}]$.
Our refined version has two new features:

\begin{itemize}
    \item It does not require a conicity assumption.
    \item It identifies both objects as {\emph{quotients}} of the ``master'' $D$-module $\mathscr H_\hbar[[{\boldsymbol{z}}]]$.
\end{itemize}
In particular, solutions of our $D$-modules, considered as functions on $\mathscr H_\hbar[[{\boldsymbol{z}}]]$, will naturally correspond to each other. As we will see in Equation (\ref{eq: vertex gives solutions}) below,
every $\mathsf{T}$-fixed component $Z \subset \widetilde{\mathcal{M}}_H^{\mathsf{T}}$ determines a solution of $Q^{\mathrm{PSZ}}_{q=2\hbar}$. We conjecture that, under the identification in Conjecture~\ref{conj: our version quantum hikita}, this solution corresponds to the normalized trace of a Verma-type module over $\mathcal{A}_{\mathsf{f}}$.
If $Z=\{p\}$ is a point, then the corresponding Verma-type module is the $\Theta$-module for the {\emph{dual}} $\theta$-fixed point $p^\vee \in \widetilde{\mathcal{M}}^{\theta}_C$; see Conjectures~\ref{conj: Verma traces as vertex functions},~\ref{eq:extended_conj_tr_Verma} below for the details.

We propose the following conjecture.
\begin{conjecture}\label{conj: our version quantum hikita}
There is an isomorphism of $D$-modules
\begin{equation*}
Q_{q=2\hbar}^{\mathrm{PSZ}} \simeq \operatorname{GrTr}(\mathcal{A})
\end{equation*}
as quotients of $\mathscr H_\hbar[[{\boldsymbol{z}}]]$.
\end{conjecture}

The main result of this paper is the proof of Conjecture \ref{conj: our version quantum hikita} for ADE quivers and $\mathsf{v},\mathsf{w}$ such that $\mathsf w_i=0$ if $\omega_i$ is not minuscule. Note that we do not assume that $\mathcal{M}_C$ is conical. The following conjecture relates the original quantum Hikita conjecture of \cite{KMBP} to our refined version.

Let $\Sigma_+\subset X^*(\mathsf A)=\pi_1(G_{\mathsf v})$ be the (finite) set of $\theta$-positive equivariant roots of $\mathcal M_C$ in the sense of \cite[Section 4.1]{KMBP} and \cite[Section 3.1.4, Definition 3.2]{Oko15}. We set
$$
\mathsf A^{\mathrm{reg}}
:=\mathsf A\setminus
\bigcup_{\alpha\in\Sigma_+}\{z^\alpha=1\}
=\{z\in\mathsf A\mid z^\alpha\neq 1\text{ for all }\alpha\in\Sigma_+\}.
$$ 
For good ADE theories (i.e., when $\mu$ is {\emph{dominant}}), symplectic duality identifies $\mathsf A$ with the K\"ahler torus $\mathsf K$ of $\widetilde{\mathcal M}_H$ and $\Sigma_+$ should identify with the set $\Delta_+$ of positive K\"ahler roots; both of them are subsets of the set of positive roots of $\mathfrak{g}_{\Gamma}$. Thus $\mathsf A^{\mathrm{reg}}$ identifies with the regular K\"ahler torus
$$
\mathsf K^{\mathrm{reg}}=\mathsf K\setminus\bigcup_{\alpha\in\Delta_+}\{z^\alpha=1\}
$$
of \cite[Section 4.1 and Remark 4.2]{KMBP}.

\begin{conjecture}\label{conj: Givental vs PSZ}
For $\Gamma$ of type ADE and dominant $\mu$, the $D$-modules $Q^{\mathrm{Giv}}$ and $Q^{\mathrm{PSZ}}$ extend to isomorphic $D$-modules on $\mathsf A^{\mathrm{reg}} \simeq {\mathsf{K}}^{\mathrm{reg}}$.
\end{conjecture}

It should be possible to prove this Conjecture \ref{conj: Givental vs PSZ} by adapting the arguments of \cite{bmo} to the PSZ $D$-module $Q^{\mathrm{PSZ}}$.

\begin{remark}
Note that, for $\Gamma$ of type ADE, $\mathcal{M}_C$ is {\emph{good}} for the natural $\mathbb{C}^\times$-action if and only if $\mu$ is dominant. 
One might expect that Conjecture \ref{conj: Givental vs PSZ} holds for arbitrary {\emph{good}} quiver theories $(G_{\mathsf{v}},{\bf{N}})$.    
\end{remark}

It is instructive to see why the conjecture above may fail to hold already for \textit{ugly} theories.

\begin{example}
Consider the simplest possible ugly theory, namely $\Gamma=A_1$ and $\mathsf v=\mathsf w=1$.
Then the corresponding quiver variety $\widetilde{\mathcal{M}}_H$ is a point.

Thus, $Q^{\mathrm{Giv}}$ is a trivial $D$-module.

On the other hand, an explicit calculation shows that the PSZ algebra of differential operators is
\begin{equation*}
    D = \mathbb C[\hbar,q,c][[z]],
\end{equation*}
with the noncommutative multiplication rule $cz=zc+qz$. Here $c$ denotes the operator corresponding to the class $c_1(\mathcal V)\in H^2_{\mathbb C^\times}(\mathfrak X)$. Its differential-operator interpretation is that $z$ acts by multiplication by the Kähler variable and $c$ is the logarithmic differential operator $qz\frac{d}{dz}$. The PSZ $D$-module in this example is the cyclic left $D$-module
$$
Q^{\mathrm{PSZ}}=D\big/D\bigl(c-z(c+2\hbar)\bigr),
$$
i.e. it is generated by one vector $1$ subject to the single relation
$$
\bigl(c-z(c+2\hbar)\bigr)1=0.
$$

If one attempted to extend Conjecture \ref{conj: Givental vs PSZ} to ugly theories, one would have to extend this $D$-module from $\mathbb C[[z]]$ to the whole $\mathsf A^{\mathrm{reg}}$. In this example, $\Sigma_+=\{1\}$, and hence $\mathsf A^{\mathrm{reg}}=\mathbb C^\times\setminus\{1\}$.

This can be done as follows: in the definition of $Q^{\mathrm{PSZ}}$, replace the completion in the $z$-direction by the localization to the open subset $z \neq 0, 1$. It is easy to see that the resulting object is indeed well-defined. What one obtains is the standard $D$-module with logarithmic singularity at $z=1$.

However, unlike $Q^{\mathrm{Giv}}$, it is certainly nontrivial; the $D$-module with logarithmic singularity has nontrivial monodromy around $z=1$.
\end{example}

Last but not least, we should remark on the notion of effective divisors entering the definitions of the GW and PSZ quantum $D$-modules. For the PSZ quantum $D$-module, one must take effective divisors in the sense described in Section \ref{quasimaps} below. Namely, its definition {\emph{depends}} on $\mathfrak X=[\mu^{-1}(0)/G_{\mathsf{v}}] \supset \widetilde{\mathcal{M}}_H$, and not solely on $\widetilde{\mathcal{M}}_H$.

\subsection{``Master'' and ``classical'' \texorpdfstring{$D$}{D}-modules}
\subsubsection{Master $D$-module}\label{sssec: master D-module}
The PSZ $D$-module will be defined as a quotient of a certain ``master'' $D$-module. Let us define the latter. Recall that $\mathfrak{X}=[\mu^{-1}(0)/G_{\mathsf{v}}]$. Here and below, we set $\mathsf{T}:=\sF\times\mathbb C^\times_\hbar$ and $\mathsf{T}_q:=\mathsf T\times\mathbb C^\times_q$. More generally, for $\mathsf{T}$ one can take an arbitrary subtorus of $\mathsf{F} \times \mathbb{C}^\times_\hbar$ projecting surjectively onto $\mathbb{C}^\times_\hbar$. For example, an arbitrary cocharacter 
\begin{equation*}
\mathbb{C}^\times \rightarrow \mathsf{F} \times \mathbb{C}^\times_\hbar, \quad t \mapsto (\gamma(t),t)
\end{equation*} 
would work. Considering these more general tori would be especially important in Section \ref{sec:application of our results}.

First, the algebra $D=D^{\mathrm{PSZ}}$ acting on our modules is defined as follows (compare with (\ref{eq: giventals algebra D})):
\begin{equation*}
D^{\mathrm{PSZ}} = S^\bullet(H^2_{T_q}(\mathfrak{X}))[[{\boldsymbol{z}}]]
\end{equation*}
with the multiplication given by
\begin{equation}\label{eq: mult on D PSZ}
(z^{d'} \otimes c)(z^d \otimes x) = q\langle \bar{c},d\rangle z^{d+d'} \otimes x + z^{d+d'} \otimes cx,
\end{equation}
where $c\in H^2_{T_q}(\mathfrak X)$, $x\in S^\bullet(H^2_{T_q}(\mathfrak X))$, $\bar c$ denotes the image of $c$ in $H^2(\mathfrak X)$, and $\langle \bar{c},d\rangle$ is the natural homology--cohomology pairing.

We define
\begin{equation*}
\mathscr{H}_{\hbar,q}[[{\boldsymbol{z}}]] := H^*_{T_q}(\mathfrak{X})[[{\boldsymbol{z}}]]
\end{equation*}
with the $D^{\mathrm{PSZ}}$-module structure given by
\begin{equation*}
(z^{d'} \otimes c)\cdot(z^d \otimes x)
=q\langle \bar c,d\rangle z^{d+d'}\otimes x
+z^{d+d'}\otimes(c\cup x),
\end{equation*}
where $c\in H^2_{T_q}(\mathfrak X)$ and $x\in H^*_{T_q}(\mathfrak X)$.

\begin{remark}
If $\mathsf v_i\leqslant 1$ for all $i \in I$, then the group $G_{\mathsf{v}}$ is a product of copies of $\mathbb{C}^\times$, and $\mathscr{H}_{\hbar,q}[[{\boldsymbol{z}}]]$ identifies with $D^{\mathrm{PSZ}}$, considered as a free left module over itself. In general, $\mathscr{H}_{\hbar,q}[[{\boldsymbol{z}}]]$ contains $D^{\mathrm{PSZ}}$ as a submodule generated by $1$.
\end{remark}

We will call $\mathscr{H}_{\hbar,q}[[z]]$ the ``master'' $D$-module, since the other $D$-modules we consider will be obtained as natural quotients of it.

\subsubsection{Classical $D$-modules}
We also define the ``classical'' $D$-module $Q^{\mathrm{Clas}}$ by
\begin{equation*}
Q^{\mathrm{Clas}}=H^*_{T_q}(\widetilde{\mathcal{M}}_H)[[{\boldsymbol{z}}]]
\end{equation*}
with the $D^{\mathrm{PSZ}}$-module structure given by the same formula as above:
\begin{equation*}
(z^{d'} \otimes c)\cdot(z^d \otimes y) = q\langle \bar{c},d\rangle z^{d+d'} \otimes y + z^{d+d'} \otimes (c|_{\widetilde{\mathcal{M}}_H} \cup y),
\end{equation*}
where $c\in H^2_{T_q}(\mathfrak X)$ and $y\in H^*_{T_q}(\widetilde{\mathcal M}_H)$.

Note that $c \in H^2_{T_q}(\mathfrak{X})$ acts on $Q^{\mathrm{Clas}}$ via the operator
\begin{equation*}
q\partial_{\bar{c}}+(c|_{\widetilde{\mathcal{M}}_H} \cup -).
\end{equation*}
Thus, this is precisely the ``classical'' (non-quantum) part of Formula~(\ref{eq: action D Giv on Q Div}) defining the quantum $D$-module.

Let $B$ be an arbitrary commutative algebra. Assume that we are given a homomorphism of algebras
\begin{equation*}
\varphi\colon H^*_{T_q}(\widetilde{\mathcal{M}}_H) \rightarrow B.
\end{equation*}
It induces an action $D \curvearrowright B[[{\boldsymbol{z}}]]$ given by the same formula as before: for $c \in H^2_{T_q}(\mathfrak{X})$ and $x \in B$, we have
\begin{equation*}
(z^{d'} \otimes c) \cdot (z^d \otimes x) = q\langle \bar{c},d\rangle z^{d+d'} \otimes x + z^{d+d'} \otimes (\varphi(c|_{\widetilde{\mathcal{M}}_H})x).
\end{equation*}
Moreover, $\varphi$ induces a homomorphism of $D$-modules
\begin{equation}\label{eq: homom Q clas to B clas}
Q^{\mathrm{Clas}} \rightarrow B[[z]],~z^d \otimes x \mapsto z^d \otimes \varphi(x).
\end{equation}

In particular, the restriction homomorphism 
\begin{equation*}
H^*_{T_q}(\widetilde{\mathcal{M}}_H) \rightarrow H^*_{T_q}(\widetilde{\mathcal{M}}_H^{T})
\end{equation*}
determines the homomorphism of $D$-modules
\begin{equation*}
Q^{\mathrm{Clas}} \rightarrow H^*_{T_q}(\widetilde{\mathcal{M}}_H^{T})[[{\boldsymbol{z}}]].
\end{equation*}
We will call $H^*_{T_q}(\widetilde{\mathcal{M}}_H^{T})[[{\boldsymbol{z}}]]$ the {\emph{fixed-locus classical $D$-module}} and will denote it by 
\begin{equation}
 \label{eq:fixed-locus-classical-module}
 Q^{\mathrm{Clas,fix}}
 :=H_{T_q}^*(\widetilde{\mathcal{M}}_H^{\mathsf{T}})[[{\boldsymbol{z}}]]
 =\bigoplus_{Z\in\pi_0(\widetilde{\mathcal{M}}_H^T)}H_{T_q}^*(Z)[[{\boldsymbol{z}}]].
\end{equation}

\subsubsection{Normalized solutions of classical $D$-modules}\label{sssec: normalized classical solutions}
Fix $Z \in \pi_0(X^{T})$ and pick a point $p \in Z$. It determines a homomorphism of $D$-modules 
$$
Q^{\mathrm{Clas},\mathrm{fix}}\longrightarrow H^*_{T_q}(p)[[{\boldsymbol{z}}]]
$$
induced by the restriction homomorphism
$$
\iota_p^*\colon H^*_{T_q}(\widetilde{\mathcal M}_H)\longrightarrow H^*_{T_q}(p),
$$
where $\iota_p\colon p\hookrightarrow Z$ is the inclusion.

Let us describe the $D$-module $H^*_{T_q}(p)[[{\boldsymbol{z}}]]$. To a $T$-fixed point $p \in Z$ we associate the homomorphism
$$
\eta_p\colon \mathsf a=H^2_{G_{\mathsf v}}(\operatorname{pt})
\longrightarrow H^2_{T}(p),
\qquad
\eta_p(a)=\iota_p^*\bigl(a|_{\widetilde{\mathcal M}_H}\bigr).
$$
Here $a|_{\widetilde{\mathcal M}_H}$ denotes the image of $a$ under the natural maps
$$
H^2_{G_{\mathsf v}}(\operatorname{pt}) 
\longrightarrow
H^2_{T}(\mathfrak X)\longrightarrow
H^2_{T}(\widetilde{\mathcal M}_H).
$$
The element $\eta_p(a)$ lies in $H^2_T(p)\subset H^2_{T_q}(p)$.
Concretely, for $a_i=c_1^{G_{\mathsf v}}(V_i)\in\mathsf a$, one has
$$
\eta_p(a_i)=c_1^T(\mathcal V_i|_p)
=\sum_{\alpha\in\operatorname{Wt}_T(\mathcal V_i|_p)}\alpha,
$$
where $\mathcal V_i$ is the tautological bundle associated with $V_i$ and the weights are counted with multiplicity.

We use $z^{\eta_p}$ for the formal symbol characterized by
$$
q\partial_{a}z^{\eta_p}
=\eta_p(a)z^{\eta_p},\qquad a\in\mathsf a.
$$

\begin{lemma}
Multiplication by $z^{\eta_p}$ induces an isomorphism of $D$-modules $H^*_{T_q}(p)[[{\boldsymbol{z}}]] \iso z^{\eta_p} \cdot \mathbb{C}[\mathsf{f},\hbar,q][[{\boldsymbol{z}}]]$.
\end{lemma}
\begin{proof}
This is clear.
\end{proof}

Summarizing, every $T$-fixed point $p \in \widetilde{\mathcal{M}}_H$ induces a homomorphism of $D$-modules
\begin{equation}\label{eq: fixed point defines formal solution}
\iota_p^*\colon Q^{\mathrm{Clas}} \longrightarrow Q^{\mathrm{Clas},\mathrm{fix}} \longrightarrow  H^*_{T_q}(p)[[{\boldsymbol{z}}]] \simeq z^{\eta_p} \cdot \mathbb{C}[\mathsf{f},\hbar,q][[{\boldsymbol{z}}]]
\end{equation}
which only depends on the choice of the connected $Z$ component of $\widetilde{\mathcal{M}}_H^{\mathsf{T}}$ in which $p$ lives.

The map (\ref{eq: fixed point defines formal solution}) should be regarded as a {\emph{solution}} of the $D$-modules $Q^{\mathrm{Clas}}$, $Q^{\mathrm{Clas},\mathrm{fix}}$ normalized by $z^{\eta_p}$.

\subsection{Definition of the PSZ \texorpdfstring{$D$}{D}-module}\label{ssec_def_of_PSZ}

The ``hyperk\"ahler Kirwan map''
\begin{equation}\label{eq: restr determines algebra}
\mathscr H_\hbar=H^*_{\mathsf{T}}(\mathfrak{X}) \twoheadrightarrow H^*_{\mathsf{T}}(\widetilde{\mathcal{M}}_H)
\end{equation}
is surjective (\cite[Corollary 1.5]{kirv}). In particular,  the algebra structure on $H^*_T(\widetilde{\mathcal{M}}_H)$ is {\emph{uniquely}} determined by the algebra structure on $\mathscr H_\hbar$ together with the map (\ref{eq: restr determines algebra}). In \cite[Section 7.4.1]{Oko15}, Okounkov defines a so-called {\emph{capped vertex function with descendant}}, which is a surjective $\mathbb{C}[\mathsf{f},\hbar,q][[{\boldsymbol{z}}]]$-linear map
\begin{equation}\label{eq: capped vertex}
\mathscr H_{\hbar,q}[[{\boldsymbol{z}}]] \twoheadrightarrow H^*_{T_q}(\widetilde{\mathcal{M}}_H)[[{\boldsymbol{z}}]], \quad \tau \mapsto \widehat{V}^{(\tau)}
\end{equation}
which at $z=0$ identifies with the restriction homomorphism (\ref{eq: restr determines algebra}). Note that the map (\ref{eq: capped vertex}) is surjective (see Proposition \ref{prop: vertsurj} below).

The PSZ quantum product $\star$ is defined as the {\emph{unique}} multiplication on $H^*_{T}(\widetilde{\mathcal{M}}_H)[[{\boldsymbol{z}}]]$ such that the map $\tau \mapsto \widehat{V}^{(\tau)}_{q=0}$ is a {\emph{homomorphism}} of algebras, where the algebra structure on $\mathscr H_{\hbar}[[{\boldsymbol{z}}]]$ is the standard one. Thus, $\tau \mapsto \widehat{V}^{(\tau)}_{q=0}$ should be regarded as a {\emph{quantum Kirwan map}} (compare with \cite{xu}, \cite{Zhang-Zhou} and references therein). We will denote $\widehat{V}^{(\tau)}_{q=0}$ simply by $\widehat{\tau}$.

The $D$-module $Q^{\mathrm{PSZ}}$ has a very similar realization. Namely, one way of defining it is by saying that the $D$-module structure on $Q^{\mathrm{PSZ}}=H^*_{T_q}(\widetilde{\mathcal{M}}_H)[[{\boldsymbol{z}}]]$ is uniquely determined by the fact that the morphism (\ref{eq: capped vertex}) is a {\emph{homomorphism of $D$-modules}}. See Section \ref{quantDmodrev} and, in particular, Definition \ref{quantDmodrev:def} for details.

In Section~\ref{sec: vertex}, we will see that $Q^{\mathrm{PSZ}}$ has an alternative realization that is closer to the standard approach to the quantum $D$-module, as in (\ref{eq: action D Giv on Q Div}) above. Namely, we will prove that the action of $c \in H^2_{T_q}(\mathfrak{X})$ on $Q^{\mathrm{PSZ}}$ is given by the operator
\begin{equation*}
q\partial_{\bar{c}}+(\widehat{c} \star -).
\end{equation*}

\subsection{From \texorpdfstring{$Q^{\mathrm{PSZ}}$}{QPSZ} to \texorpdfstring{$Q^{\mathrm{Clas}}_{\mathrm{loc}}=Q^{\mathrm{Clas},\mathrm{fix}}_{\mathrm{loc}}$}{Qfix} via vertex functions}
In \cite[Section 7.2.3]{OkLec}, Okounkov defines a {\emph{vertex function with descendant}}
\begin{equation}\label{eq: vertex map}
V^{(-)}\colon \mathscr H_{\hbar,q}[[{\boldsymbol{z}}]] \rightarrow H^*_{T_q}(\widetilde{\mathcal{M}}_H^{\mathsf{T}})_{\mathrm{loc}}[[{\boldsymbol{z}}]], \quad \tau \mapsto V^{(\tau)},
\end{equation}
where by $\bullet_{\mathrm{loc}}$ we mean the tensor product $\bullet \otimes_{\mathbb{C}[\mathsf{f},\hbar,q]} \operatorname{Frac}(\mathbb{C}[\mathsf{f},\hbar,q])$. 
Clearly, the RHS of (\ref{eq: vertex map}) is $Q^{\mathrm{Clas},\mathrm{fix}}_{\mathrm{loc}}$ which one can identify with $Q^{\mathrm{Clas}}_{\mathrm{loc}}$.

The following lemma follows from \cite[Equation (7.2.5)]{OkLec}.
\begin{lemma}
The map (\ref{eq: vertex map}) is a homomorphism of $D$-modules.
\end{lemma}

The following proposition will be proved in Section~\ref{sec: vertex} and goes back to \cite[Equation (7.4.6)]{OkLec} and \cite[proof of Theorem 5.16]{SZ}.

\begin{proposition}\label{prop: vertex factors through kernel of capped}
The homomorphism $V^{(-)}$ factors through the quotient $\mathscr{H}_{\hbar,q}[[{\boldsymbol{z}}]] \twoheadrightarrow Q^{\mathrm{PSZ}}$, inducing a homomorphism of $D$-modules 
\begin{equation*}
\ol{V}^{(-)}\colon Q^{\mathrm{PSZ}} \rightarrow Q^{\mathrm{Clas},\mathrm{fix}}_{\mathrm{loc}}.
\end{equation*}
\end{proposition}

Combining this with (\ref{eq: fixed point defines formal solution}), we conclude that any $p \in \widetilde{\mathcal{M}}_H^{\mathsf{T}}$ determines a (normalized) solution of the $D$-module $Q^{\mathrm{PSZ}}$ depending {\emph{only}} on the connected component $Z \ni p$:
\begin{equation}\label{eq: vertex gives solutions}
\ol{V}^{(-)}_{p}\colon Q^{\mathrm{PSZ}} \rightarrow     H^*_{\mathsf{T}_q}(\operatorname{pt})_{\mathrm{loc}}[[{\boldsymbol{z}}]] \iso  z^{\eta_p} \cdot \operatorname{Frac}(\mathbb{C}[\mathsf{f},\hbar,q])[[{\boldsymbol{z}}]],
\end{equation}
where the identification in (\ref{eq: vertex gives solutions}) is given by $1 \mapsto z^{\eta_p}$.

\subsection{The \texorpdfstring{$q=2\hbar$}{q=2h} specialization of the PSZ \texorpdfstring{$D$}{D}-module}
It turns out that (normalized) solutions (\ref{eq: vertex gives solutions}) are {\emph{regular}} along  $q=2\hbar$ and specialize to elements of $H^*_{\mathsf{T}}(\operatorname{pt})[[{\boldsymbol{z}}]]$, i.e., {\emph{no}} localization occurs.  Namely, the following proposition will be proven in Section \ref{CYvertex}.

\begin{proposition}
 \label{nonlocvertex}
Let $\mathsf T\subset\mathsf F\times\mathbb C_\hbar^\times$ be any subtorus, and set $\mathsf T_q:=\mathsf T\times\mathbb C_q^\times$. We continue to denote by $\hbar$ the restriction to $\mathsf T$ of the inverse of the standard character of $\mathbb C_\hbar^\times$, as well as its first Chern class. Denote by $V_{\mathsf T}^{(-)}$ the uncapped vertex defined directly in $\mathsf T_q$-equivariant cohomology.
For every
 $\tau\in H^*_{\mathsf T_q}(\mathfrak X)[[{\boldsymbol z}]]$, every connected component
 $Z\subset\widetilde{\mathcal M}_H^{\mathsf T}$, and every $p\in Z$, the  restriction
 $\iota_p^*V_{\mathsf T}^{(\tau)}$ is regular along $q=2\hbar$, and
\begin{equation}\label{eq: spec vertex point}
  \left.\iota_p^*V_{\mathsf T}^{(\tau)}\right|_{q=2\hbar}
  \in H^*_{\mathsf T}(\operatorname{pt})[[{\boldsymbol z}]]
\end{equation}
is nonlocalized and depends only on $Z$.
\end{proposition}

To simplify notation, we will omit the subscript $\mathsf T$ and sometimes denote $\iota_p^*V^{(\tau)}$ by $V^{(\tau)}_{p}$; the corresponding $q=2\hbar$ specialization will be denoted $V^{(\tau)}_{p,q=2\hbar}$.

\begin{warning}\label{war: specializations do not commute}
A very important caveat is that (\ref{eq: spec vertex point}) {\emph{strongly depends}} on the choice of the torus $\mathsf{T}$ in the following sense. Assume that $\mathsf{T}' \subset \mathsf{T}$ is some subtorus. Let 
\begin{equation*}
V^{(\tau),\mathsf{T}} \in H^*_{\mathsf{T}}(\widetilde{\mathcal{M}}_H^{\mathsf{T}})_{\mathrm{loc}}[[{\boldsymbol{z}}]], \quad  V^{(\tau),\mathsf{T}'} \in H^*_{\mathsf{T}}(\widetilde{\mathcal{M}}_H^{\mathsf{T}'})_{\mathrm{loc}}[[{\boldsymbol{z}}]]
\end{equation*} 
be the corresponding vertex functions.
Then the corresponding $q=2\hbar$ specializations are {\emph{not}} compatible   in general:
\begin{equation}\label{eq:two spec are not equal}
\Big(V^{(\tau),\mathsf{T}}_{p,q=2\hbar}\Big)\Big|_{\mathsf{t}'} \neq V^{(\tau),\mathsf{T}'}_{p,q=2\hbar}.
\end{equation}
 The LHS of (\ref{eq:two spec are not equal}) also defines a solution of $Q_{q=2\hbar}^{\mathrm{PSZ}}=H^*_{\mathsf{T}'}(\widetilde{\mathcal{M}}_H)[[{\boldsymbol{z}}]]$, but a {\emph{different}} one from $V^{(\tau),\mathsf{T}'}_{p,q=2\hbar}$. This is compatible with the discussion in Section \ref{eq: relation between modules}: namely, the Verma-type modules we consider become {\emph{distinct}} at singular parameters. See Section \ref{sssec_kleinian_example} for an explicit example confirming (\ref{eq:two spec are not equal}).
 Note that without taking the $q=2\hbar$ specialization, it is true that the specialization $\Big(V^{(\tau),\mathsf{T}}_p\Big)\Big|_{\mathsf{t}'}$ {\emph{is}} well-defined and
 \begin{equation*}
\Big(V^{(\tau),\mathsf{T}}_p\Big)\Big|_{\mathsf{t}'}=V^{(\tau),\mathsf{T}'}_p.
 \end{equation*} 
Putting it differently, the $q=2\hbar$ vertex functions for all subtori of $\mathsf{T}$ can be {\emph{recovered}} from the unique function $V^{(\tau),\mathsf{T}}_p$ by taking limits in {\emph{appropriate}} order. 
\end{warning}

To simplify notation we set $\iota=\iota_{\widetilde{\mathcal{M}}_H}$.
Proposition~\ref{nonlocvertex}, in particular, implies that whenever the $\mathsf{T}$-fixed points of $\widetilde{\mathcal{M}}_H$ are {\emph{isolated}}, the $D$-module $Q^{\mathrm{PSZ}}_{q=2\hbar}$, considered as a quotient of $\mathscr H_\hbar[[{\boldsymbol{z}}]]$ via (\ref{eq: capped vertex}), can be described as a {\emph{submodule}} of $Q^{\mathrm{Clas},\mathrm{fix}}_{q=2\hbar}$; note that the latter is {\emph{not}} localized.

\begin{proposition}
 \label{prop:fixed-locus-vertex-map}
If $\widetilde{\mathcal{M}}_H^{\mathsf{T}}$ are isolated, the formula
 \begin{equation}
  \label{eq:fixed-locus-vertex-map}
  V^{\mathrm{fix},(-)}_{q=2\hbar}\colon
  \mathscr H_\hbar[[{\boldsymbol{z}}]]\longrightarrow
  Q^{\mathrm{Clas,fix}}_{q=2\hbar},
  \qquad
  \tau\longmapsto
  \left.\iota^*V^{(\tau)}\right|_{q=2\hbar},
 \end{equation}
 defines a homomorphism of $D$-modules.  At
 $z=0$, it is the natural restriction homomorphism
 $$
  \mathscr H_\hbar\longrightarrow H_T^*(\widetilde{\mathcal{M}}_H^T),
  \qquad
  \tau\longmapsto\iota^*(\tau|_X).
 $$
 Moreover,~\eqref{eq:fixed-locus-vertex-map} factors uniquely through the
 capped-vertex quotient, giving a homomorphism of $D$-modules
 \begin{equation}
  \label{eq:PSZ-to-fixed-locus-classical}
  \overline V_{q=2\hbar}\colon
  Q^{\mathrm{PSZ}}_{q=2\hbar}
  \longrightarrow Q^{\mathrm{Clas,fix}}_{q=2\hbar}.
 \end{equation}
\end{proposition}

\begin{proof}
 Proposition~\ref{nonlocvertex} shows that the right-hand side
 of~\eqref{eq:fixed-locus-vertex-map} is well-defined over $\mathfrak C$.
 Before the specialization $q=2\hbar$, the uncapped vertex is a homomorphism
 of $D$-modules by \cite[Equation~(7.2.5)]{OkLec}.  
Proposition~\ref{nonlocvertex}
 permits us to specialize this to $q=2\hbar$; hence
 \eqref{eq:fixed-locus-vertex-map} is a homomorphism of $D$-modules.

The $z=0$ term of
 $V^{(\tau)}$ is $\tau|_{\widetilde{\mathcal{M}}_H}$: this is precisely the assertion about the constant
 term. This follows from the definition of quasimaps in Section \ref{quasimaps}.

 It remains to justify the factorization after specialization.  Set
 $\mathfrak{C}_q:=\mathbb C[\mathsf f,\hbar,q]$ and let
 $N:=\ker(\widehat V^{(-)})$.  Surjectivity of $\widehat{V}^{(-)}$ gives an
 exact sequence
 $$
  0\longrightarrow N\longrightarrow\mathscr H_{\hbar,q}[[{\boldsymbol{z}}]]
  \xrightarrow{\ \widehat V^{(-)}\ }Q^{\mathrm{PSZ}}
  \longrightarrow0.
 $$
 The $\mathfrak{C}_q$-module $H_{T_q}^*(\widetilde{\mathcal{M}}_H)$ is finite free by
 \cite[Theorem~7.3.5]{Nak_quiver_and_fd_reps}.  Consequently, $Q^{\mathrm{PSZ}}=H_{T_q}^*(\widetilde{\mathcal{M}}_H)[[{\boldsymbol{z}}]]$ is finite free over
 $\mathfrak{C}_q[[{\boldsymbol{z}}]]$, so the displayed sequence splits as a sequence of
 $\mathfrak{C}_q[[{\boldsymbol{z}}]]$-modules.  It therefore remains exact under the base change
 $\mathfrak{C}_q\to \mathfrak{C}_q/(q-2\hbar)=\mathfrak C$, and we obtain
 \begin{equation}
  \label{eq:specialized-capped-kernel}
  \ker\bigl(\widehat V^{(-)}|_{q=2\hbar}\bigr)=\operatorname{Im}\bigl(N\otimes_{\mathfrak{C}_q}\mathfrak C
    \longrightarrow\mathscr H_\hbar[[{\boldsymbol{z}}]]\bigr).
 \end{equation}
 Before specialization, Proposition~\ref{prop: vertex factors through kernel of capped} 
 implies that the restriction of the bare
 vertex to $X^T$ annihilates $N$.  Its specialization is regular by Proposition~\ref{nonlocvertex}; hence
 $V^{\mathrm{fix},(-)}_{q=2\hbar}$ annihilates the right-hand side of
 \eqref{eq:specialized-capped-kernel}.  It therefore induces
 \eqref{eq:PSZ-to-fixed-locus-classical}.  Uniqueness follows from the
 surjectivity of the specialized capped vertex.
\end{proof}

\begin{proposition}
 \label{prop:PSZ-to-fixed-locus-classical-injective}
 \begin{itemize}
     \item[(a)] The homomorphism
 \begin{equation}\label{eq:vertex PSZ to Clas fix}
  \overline V_{q=2\hbar}\colon
  Q^{\mathrm{PSZ}}_{q=2\hbar}
  \longrightarrow Q^{\mathrm{Clas,fix}}_{q=2\hbar}
 \end{equation}
 in (\ref{eq:PSZ-to-fixed-locus-classical}) is injective.
 \item[(b)] The homomorphism (\ref{eq:vertex PSZ to Clas fix}) is an isomorphism at $x \in \operatorname{Spec}\mathfrak{C}$ iff the restriction homomorphism $H^*_{T}(\widetilde{\mathcal{M}}_H)|_x \rightarrow H^*_T(\widetilde{\mathcal{M}}_H^T)|_x$ is an isomorphism.\footnote{The third author has incorrectly claimed in a couple of talks that the map $\ol{V}_{q=2\hbar}$ induces the isomorphism of $Q^{\mathrm{PSZ}}_{q=2\hbar}$ with $Q^{\mathrm{Clas}}_{q=2\hbar}$.  Proposition \ref{prop:PSZ-to-fixed-locus-classical-injective} is the correct replacement of this statement.\label{vasya_correction}}
 \end{itemize}
\end{proposition}

\begin{proof}
 Let us prove part (a). We regard both sides of (\ref{eq:vertex PSZ to Clas fix}) as families over $\operatorname{Spec}\mathfrak C$. Set $\mathfrak{K}:=\operatorname{Frac}(\mathfrak{C})$.
 The equivariant cohomology $H_T^*(\widetilde{\mathcal{M}}_H)$ is free over $\mathfrak C$ by
 \cite[Theorem~7.3.5]{Nak_quiver_and_fd_reps}.  Consequently, the
 coefficientwise homomorphism
 \begin{equation}
  \label{eq:PSZ-to-completed-generic-fiber}
  H_T^*(\widetilde{\mathcal{M}}_H)[[{\boldsymbol{z}}]]\longrightarrow
  \bigl(H_T^*(\widetilde{\mathcal{M}}_H) \otimes_{\mathfrak{C}} \mathfrak{K}\bigr)[[{\boldsymbol{z}}]]
 \end{equation}
 is injective.  It is therefore enough to prove that the completed
 extension of~\eqref{eq:PSZ-to-fixed-locus-classical} to $\mathfrak{K}$ is injective.

 At ${\boldsymbol{z}}=0$, the capped vertex is the natural surjection
 $\mathscr H_\hbar\twoheadrightarrow H_T^*(\widetilde{\mathcal{M}}_H)$ in
 \eqref{eq: restr determines algebra}.  The constant-term assertion in
 Proposition~\ref{prop:fixed-locus-vertex-map} therefore shows that the
 reduction of $\overline V_{q=2\hbar}$ at ${\boldsymbol{z}}=0$ is the
 restriction homomorphism
 $$
  \iota^*\colon H_T^*(\widetilde{\mathcal{M}}_H)\longrightarrow H_T^*(\widetilde{\mathcal{M}}_H^T).
 $$
 By the equivariant localization theorem, this homomorphism induces an
 isomorphism
 \begin{equation}
  \label{eq:generic-fixed-locus-localization}
  H_T^*(\widetilde{\mathcal{M}}_H)\otimes_{\mathfrak C} \mathfrak{K}
  \xrightarrow{\ \sim\ }
  H_T^*(\widetilde{\mathcal{M}}_H^T)\otimes_{\mathfrak C} \mathfrak{K}.
 \end{equation}

 We use completed extension of coefficients in the Novikov direction; thus,
 for example,
 $$
  \left(Q^{\mathrm{PSZ}}_{q=2\hbar}\right)_{\mathfrak{K}}
  :=\bigl(H_T^*(\widetilde{\mathcal{M}}_H)\otimes_{\mathfrak C} \mathfrak{K}\bigr)[[\boldsymbol{z}]].
 $$
 The reduction of the resulting morphism
 $\left(\overline V^{\mathrm{fix}}_{q=2\hbar}\right)_{\mathfrak{K}}$ at ${\boldsymbol{z}}=0$ is
 precisely~\eqref{eq:generic-fixed-locus-localization}.

 Since ${\mathsf{T}}$ acts trivially on $X^{\mathsf{T}}$, one has
 $$
  H_T^*(\widetilde{\mathcal{M}}_H^T)=H^*(\widetilde{\mathcal{M}}_H^T)\otimes_{\mathbb C}\mathfrak C;
 $$
 in particular, $H_T^*(\widetilde{\mathcal{M}}_H^T)$ is a finite free $\mathfrak C$-module.  Let
 $\mathfrak{m}_{\mathfrak{K}}\subset \mathfrak{K}[\boldsymbol{z}]$ be the ideal generated by the
 nonconstant Novikov monomials.  For every $n\geqslant1$, reduce the source
 and target modulo $\mathfrak{m}_{\mathfrak{K}}^n$.  These are finite free modules of the same rank
 over $\mathfrak{K}[\boldsymbol z]/\mathfrak{m}_{\mathfrak{K}}^n$, and the reduction of the induced map modulo
 $\mathfrak{m}_{\mathfrak{K}}$ is the isomorphism~\eqref{eq:generic-fixed-locus-localization}.
 Nakayama's lemma shows that the map modulo $\mathfrak{m}_{\mathfrak{K}}^n$ is surjective; since its
 source and target are free of the same finite rank, it is an isomorphism.
 Passing to the inverse limit gives an isomorphism
 $$
  \left(Q^{\mathrm{PSZ}}_{q=2\hbar}\right)_{\mathfrak{K}}
  \xrightarrow{\ \sim\ }
  \left(Q^{\mathrm{Clas,fix}}_{q=2\hbar}\right)_{\mathfrak{K}}.
 $$

 If $h$ lies in the kernel of
 \eqref{eq:PSZ-to-fixed-locus-classical}, its image under
 \eqref{eq:PSZ-to-completed-generic-fiber} lies in the kernel of the
 completed generic morphism, which is an isomorphism.  Hence $h$ maps to
 zero in the completed generic fiber, and the injectivity of
 \eqref{eq:PSZ-to-completed-generic-fiber} gives $h=0$.

 Part (b) is clear. 
\end{proof}

\begin{corollary}
 \label{cor:fixed-locus-realization-PSZ}
 The image
 $$
  Q^{\mathrm{fix}}_{q=2\hbar}
  :=\operatorname{Im}\bigl(V^{(-)}_{q=2\hbar}\bigr)
  \subset Q^{\mathrm{Clas,fix}}_{q=2\hbar}
 $$
 is a $D$-submodule of $Q^{\mathrm{Clas},\mathrm{fix}}_{q=2\hbar}$, and the induced map gives an
 isomorphism
 $$
  Q^{\mathrm{PSZ}}_{q=2\hbar}
  \xrightarrow{\ \sim\ }Q^{\mathrm{fix}}_{q=2\hbar}.
 $$
 Equivalently, one has a commutative diagram of $D$-modules
 $$
 \begin{tikzcd}
  \mathscr H_\hbar[[{\boldsymbol{z}}]]
   \arrow[->>]{r}{\widehat V^{(-)}|_{q=2\hbar}}
   \arrow{dr}[swap]{V^{(-)}_{q=2\hbar}}
  & Q^{\mathrm{PSZ}}_{q=2\hbar}
   \arrow[hookrightarrow]{d}{\overline V^{(-)}_{q=2\hbar}} \\
  & Q^{\mathrm{Clas,fix}}_{q=2\hbar}.
 \end{tikzcd}
 $$
\end{corollary}

Proposition~\ref{prop:PSZ-to-fixed-locus-classical-injective} implies that {\emph{whenever $\widetilde{\mathcal{M}}_H^{\mathsf{T}}$ are isolated}}, Conjecture~\ref{conj: our version quantum hikita} can be reformulated as follows (without mentioning the $D$-module $Q^{\mathrm{PSZ}}_{q=2\hbar}$):

\begin{conjecture}\label{conj: classical reformulation quantum hikita} Assume $\widetilde{\mathcal{M}}_H^{\mathsf{T}}$ are isolated.
Then there exists an isomorphism of $D$-modules $Q_{q=2\hbar}^{\mathrm{fix}} \iso \operatorname{GrTr}(\mathcal{A})$ making the following diagram commutative:

\begin{equation}\label{eq:conj reformulated}
\begin{tikzcd}
	& {\GZ_{\hbar}[[{\boldsymbol{z}}]]} & \\
	{Q_{q=2\hbar}^{\mathrm{fix}}} && {\operatorname{GrTr}(\mathcal{A})}
	\arrow["V^{(-)}_{q=2\hbar}"', from=1-2, to=2-1]
	\arrow[from=1-2, to=2-3]
	\arrow["\simeq", from=2-1, to=2-3]
\end{tikzcd}
\end{equation}
\end{conjecture}
Here the left diagonal arrow is given by the uncapped vertex function, and the right diagonal arrow is the tautological map.

Note that Conjecture~\ref{conj: classical reformulation quantum hikita} gives a  ``classical'' description of the $D$-module $\operatorname{GrTr}(\mathcal{A})$, while the ``quantum'' part is absorbed by the map $V^{(-)}_{q=2\hbar}$.

Conjecture \ref{conj: classical reformulation quantum hikita} {\emph{only}} makes sense as stated with the assumption that the points of $\widetilde{\mathcal{M}}_H^{\mathsf{T}}$ are isolated. The reason is the following: without this assumption, $V^{(-)}$ will {\emph{not}} be regular at $q=2\hbar$. The simplest example is the following. It is an interesting question whether there is still a ``reasonable'' way to specialize $V^{(-)}$ at $q=2\hbar$.

\begin{example}\label{ex:vertex has pole at hbar q}
Let $\Gamma$ be an $A_3$-quiver with 
\begin{equation*}
v_1=v_2=v_3=1, \quad w_1=w_3=1,
w_2=0.
\end{equation*} Then $\widetilde{\mathcal{M}}_H$ is the resolution of
$\mathbb{A}^2/(\mathbb{Z}/4\mathbb{Z})$. Take
${\mathsf{T}}=\mathbb{C}^\times_\hbar$, so
$\widetilde{\mathcal{M}}_H^{\mathbb{C}^\times_\hbar}$ contains a component
$C$ isomorphic to $\mathbb{P}^1$.

Temporarily restore the effective one-dimensional flavor torus and write
\begin{equation*}
\varphi=f_1-f_3
\end{equation*} 
for its parameter. Let $p_L,p_R\in C$ be the two
flavor-fixed points, and denote by $V_L(\varphi)$ and $V_R(\varphi)$ the
corresponding point restrictions of the vertex. Choose
$\widetilde\eta\in H^2_{\mathsf F_{\mathsf w}}(C)$ such that
$$
 \widetilde\eta|_{p_L}=0,
 \qquad
 \widetilde\eta|_{p_R}=\varphi,
$$
and let $\eta\in H^2(C)$ be its image after forgetting the flavor
equivariance. Equivariant interpolation on $C\cong\mathbb P^1$ gives
$$
 \left.V^{(1)}\right|_C
 =
 V_L(\varphi)
 +\frac{\widetilde\eta}{\varphi}
 \bigl(V_R(\varphi)-V_L(\varphi)\bigr).
$$
Consequently, after forgetting the flavor equivariance,
$$
 \left.V^{(1)}\right|_C
 =
 V_U+
 \eta\left.
 \frac{\partial}{\partial\varphi}
 \bigl(V_R(\varphi)-V_L(\varphi)\bigr)
 \right|_{\varphi=0},
 \qquad
 V_U:=V_L(0)=V_R(0).
$$
In degree $(0,0,1)$, the localization calculation of
Subsection~\ref{sssec_kleinian_example} gives
$$
 [z_3]V_L(\varphi)
 =
 \frac{2\hbar\varphi}{q(2\hbar+\varphi-q)},
 \qquad
 [z_3]V_R(\varphi)=0.
$$
It follows that
$$
 [z_3]\left(\left.V^{(1)}\right|_C\right)
 =
 \frac{2\hbar}{q(q-2\hbar)}\,\eta.
$$
Since $\eta\neq0$, the component-valued vertex has a genuine pole at
$q=2\hbar$. This pole is invisible after restriction to any point
$p\in C$: the pullback of $\eta$ to $p$ vanishes, and hence
$\left.V^{(1)}\right|_p=V_U$. Thus it is the restriction to the whole
fixed component, rather than its pointwise restrictions, that obstructs
the Calabi--Yau specialization.
\end{example}

\begin{remark}\label{rem: vertex for  cotangent}
There is, nevertheless, an important class of examples with a
positive-dimensional fixed locus for which the global specialization {\emph{is}}
well-defined. Let
$$
 X=T^*(\operatorname{GL}_n/P)
$$
with its standard type~$A$ Nakajima-quiver presentation, and take
$\mathsf T=\mathbb C_\hbar^\times$, acting by cotangent dilation in the
normalization of~\eqref{action_hbar_quiver}. Then
$$
 X^{\mathsf T}=Y:=\operatorname{GL}_n/P,
$$
but the specialization
$$
 \left.\bigl(V^{(-)}|_Y\bigr)\right|_{q=2\hbar}\colon
 \mathscr H_\hbar[[{\boldsymbol z}]]
 \longrightarrow
 H^*_{\mathbb C_\hbar^\times}(Y)[[{\boldsymbol z}]]
 \cong
 \bigl(H^*(Y)\otimes\mathbb C[\hbar]\bigr)[[{\boldsymbol z}]]
$$
is defined without localization and $\Big(V^{(1)}|_Y\Big)_{q=2\hbar}$ lands in $H^0(Y)[[{\boldsymbol{z}}]]$.
See Section \ref{eq: ssec case of cootangent}
for the details.
\end{remark}

Note that, on the other hand, Conjecture \ref{conj: our version quantum hikita} should be true in general without any assumptions on the quiver and the torus $\mathsf{T}$.

The following proposition will be crucial for our approach as it allows us to reduce the proof of Conjecture \ref{conj: our version quantum hikita} to the case of a ``maximal possible'' torus $\mathsf{T}$. Then, whenever this ``large'' torus has isolated fixed points, Conjecture \ref{conj: our version quantum hikita} for this particular torus reduces to Conjecture \ref{conj: classical reformulation quantum hikita}.

\begin{proposition}
Assume that Conjecture \ref{conj: our version quantum hikita} holds for some torus $\mathsf{T}$. Then it also holds for any subtorus $\mathsf{T}' \subset \mathsf{T}$.
\end{proposition}
\begin{proof}
This follows from Proposition \ref{Vasyapoprosil} below combined with Lemma \ref{lem:specialization}(3). 
\end{proof}

In particular, whenever we know Conjecture \ref{conj: classical reformulation quantum hikita} holds for some $\mathsf{T}=\mathsf{F} \times \mathbb{C}^\times_\hbar$, then it also holds for the {\emph{trivial}} flavor torus and $\hbar=1$ (from the Coulomb branch perspective this corresponds to the most {\emph{singular}} quantization parameter, compare with Section \ref{eq: relation between modules}). In this case Conjecture \ref{conj: our version quantum hikita} boils down to the existence of the following commutative diagram
\begin{equation}\label{eq:conj reformulated after specialization}
\begin{tikzcd}
	& {\GZ_{\hbar}[[{\boldsymbol{z}}]]} & \\
	{H^*(\widetilde{\mathcal{M}}_H^{\mathbb{C}^\times_\hbar})} && {\operatorname{GrTr}(\mathcal{A}_{\hbar=1}),}
	\arrow["\widehat{V}^{(-)}_{\hbar=1,q=2}"', from=1-2, to=2-1]
	\arrow[from=1-2, to=2-3]
	\arrow["\simeq", from=2-1, to=2-3]
\end{tikzcd}
\end{equation}
Proposition \ref{nonlocvertex} then implies that $\operatorname{GrTr}(\mathcal{A}_{\hbar=1})$ has a collection of distinguished solutions labeled by the connected components of $\widetilde{\mathcal{M}}_H^{\mathbb{C}^\times_\hbar}$ and considered as functionals on $\mathscr{H}_\hbar[[{\boldsymbol{z}}]]$.

\subsection{\texorpdfstring{$\on{GrTr}(\mathcal{A})$}{GrTrA} for generic parameters}
In this section we make the following assumption which will be checked in Section \ref{sec:ADE-Hilbert} for the cases of interest to us.

\begin{definition}
We will say that an $\mathcal{A}$-module $M$ is a ``Verma-type module'' if the following conditions hold.
\begin{enumerate}
    \item The action of $\mathsf{a} \curvearrowright M$ decomposes $M$ as a direct sum of eigenspaces, and each of these eigenspaces is free of finite rank over $\mathfrak{C}$.
    \item The graded character of $M$ is an element of $z^{\xi} \cdot (1+\mathfrak{m}_\hbar\mathfrak{C}[[{\boldsymbol{z}}]]_\hbar)$ for some linear functional $\xi\colon \mathsf{a} \rightarrow \mathsf{f}^*\oplus \mathbb{C}\hbar$. We call $\xi$ the {\emph{highest weight}} of $M$.
\end{enumerate}
\end{definition}

If $M$ is a Verma-type module with highest weight $\xi$, then the corresponding $\mathsf{a}$-weight space $M(\xi)$ has rank one over $\mathfrak{C}$. In particular, every element of $\mathscr{H}_\hbar$ acts on $M(\xi)$ via the multiplication by some element of $\mathfrak{C}$.

\begin{assumption}\label{ass:our_main_ass_modules_dim_est}
\begin{enumerate}[label=(\arabic*)]
\item There exists a finite collection of Verma-type modules $\{M_p\mid p \in S\}$ over $\mathcal{A}$ labeled by some finite set $S$, such that the map 
\begin{equation*}
\mathscr{H}_\hbar \rightarrow \bigoplus_{p \in S}\mathfrak{C}
\end{equation*}
induced by the action of $\mathscr{H}_\hbar$ on $M_p(\xi_p)$ becomes {\emph{surjective}} after tensoring by $\mathfrak{K}=\operatorname{Frac}(\mathfrak{C})$.
\item We have $\dim_{\mathbb C}\mathbb{C}[\mathcal{M}_C^\theta] \leqslant |S|$.
\end{enumerate}
\end{assumption}

\begin{remark}
 We expect Assumption \ref{ass:our_main_ass_modules_dim_est}(2) to always hold for $|S|=\operatorname{dim}H^*(\widetilde{\mathcal{M}}_H)$. 
 
 We expect that Assumption \ref{ass:our_main_ass_modules_dim_est}(1) holds whenever the points of $\widetilde{\mathcal{M}}_H^{\mathsf{F}}$ are isolated. The modules that Botta and Tamagni obtain via in \cite{BT26} should hopefully provide the desired family.
\end{remark}

\begin{proposition}\label{prop: tr-iso}
Let $\operatorname{GrTr}(\mathcal{A})_{\mathrm{loc}}$ be the  completed localization to $\mathfrak{K}$. Then the maps $\operatorname{tr}_{M_p}$ induce an isomorphism of $\mathfrak{K}[[{\boldsymbol{z}}]]$-modules:
\begin{equation*}
\bigoplus_{p\in S} \operatorname{tr}_{M_p}\colon \operatorname{GrTr}(\mathcal{A})_{\mathrm{loc}} \iso \bigoplus_{p\in S} z^{\xi_p} \cdot \mathfrak{K}[[{\boldsymbol{z}}]].
\end{equation*}
\end{proposition}
\begin{proof}
We denote by $\mathfrak{m}_{\mathfrak{K}}$ the ideal generated by $\mathfrak{m}$ in $\mathfrak{K}[{\boldsymbol{z}}]$ and, abusing notation, consider all objects modulo $\mathfrak{m}_{\mathfrak{K}}^n$.
It is enough to prove the claim after taking the quotient by $\mathfrak{m}_{\mathfrak{K}}^n$ for $n \geqslant 1$ (use that both $\mathfrak{K}[[{\boldsymbol{z}}]]$-modules are complete).

We claim that the morphism $\bigoplus_{p\in S} \operatorname{tr}_{M_p}$ is surjective. By Nakayama's lemma for the ring $\mathfrak K[{\boldsymbol{z}}]/\mathfrak{m}_{\mathfrak K}^n$, combined with Lemma~\ref{lem: GrTr fin gen}, it is enough to show that the morphism
$$
\mathsf{C}_{\theta}(\mathcal{A})\otimes_{\mathfrak C}\operatorname{Frac}(\mathfrak C)
\longrightarrow \bigoplus_{p\in S}\operatorname{Frac}(\mathfrak C)
$$
is surjective. This follows from Assumption~\ref{ass:our_main_ass_modules_dim_est}(1) and Lemma~\ref{prop: surj onto cartan subquot}.

It remains to prove that $\bigoplus_{p\in S} \operatorname{tr}_{M_p}$ is injective. Consider
$$
\on{GrTr}_{n}(\mathcal{A})_{\mathrm{loc}}:=\operatorname{GrTr}(\mathcal A)_{\mathrm{loc}}/\mathfrak m_{\mathfrak K}^n\operatorname{GrTr}(\mathcal A)_{\mathrm{loc}}
$$
as a finitely generated module over $\operatorname{Frac}(\mathfrak{C})$ by Lemma~\ref{lem: GrTr fin gen}. Then one has
\begin{multline}\label{eq: comparison}
\dim_{\operatorname{Frac}(\mathfrak C)} \on{GrTr}_{n}(\mathcal{A})_{\mathrm{loc}}
\leqslant \dim_{\mathbb{C}}\left(\operatorname{GrTr}(\mathcal A)_{\mathsf f=\hbar=0}/\mathfrak m^n\right)
= {}\\
\dim_{\mathbb C} \mathbb C[\mathcal M_C^{\theta}] \cdot \dim_{\mathbb{C}}(\mathbb{C}[{\boldsymbol z}]/\mathfrak m^n),
\end{multline}
where the last equality follows from Lemma~\ref{lem:specialization}(1).

Equation~\eqref{eq: comparison} and Assumption~\ref{ass:our_main_ass_modules_dim_est}(2) show that the dimension over $\operatorname{Frac}(\mathfrak C)$ of the source is not greater than that of the target. Since $\bigoplus_{p\in S} \operatorname{tr}_{M_p}$ is surjective, it must be an isomorphism.
\end{proof}

Let us denote by $\mathfrak m_{\mathfrak C}$ the ideal generated by $\mathfrak m$ in $\mathfrak C[{\boldsymbol{z}}]$.
\begin{lemma}\label{lem: freequot}
For every $n\geqslant 1$, the module
$$
\operatorname{GrTr}(\mathcal A)_n:=\operatorname{GrTr}(\mathcal A)/\mathfrak m_{\mathfrak{C}}^n\operatorname{GrTr}(\mathcal A)
$$
is free over $\mathfrak C[{\boldsymbol z}]/\mathfrak m_{\mathfrak C}^n$.
\end{lemma}
\begin{proof}

By Lemma~\ref{lem: H to GrTr is surjective}, $\operatorname{GrTr}(\mathcal A)$ is a $\mathbb{Z}_{\geqslant 0}$-graded $\mathfrak{C}$-module.

By the graded Nakayama lemma, it is therefore enough to prove that each $\operatorname{GrTr}(\mathcal A)_n$ has the same dimension after localization to $\operatorname{Frac}(\mathfrak C)$ as after specialization at $\mathsf{f}=\hbar=0$.

From Proposition~\ref{prop: tr-iso}, we know that the dimension of the generic specialization is equal to
$$
|S|\cdot \dim_{\mathfrak K}\left(\mathfrak K[{\boldsymbol z}]/\mathfrak m_{\mathfrak K}^n\right).
$$

Thus the lemma follows from the same estimate as in Equation~\eqref{eq: comparison}.
\end{proof}

\section{Proofs of Theorems \ref{mainthm:reduction}, \ref{maintheorem: ADE} and applications}\label{sec: our approach}

\subsection{Proof of Theorem \ref{mainthm:reduction}}

\begin{lemma}\label{lem: when traces verma determine GrTr}
Assuming Assumption \ref{ass:our_main_ass_modules_dim_est} holds, there is an embedding of $D$-modules:
\begin{equation*}
\bigoplus_{p\in S} \operatorname{tr}_{M_p}\colon \on{GrTr}(\mathcal{A}) \hookrightarrow \bigoplus_{p\in S}\mathcal{Z}_{\xi_p}.
\end{equation*}
\end{lemma}
\begin{proof}
It is enough to prove the corresponding statement for all quotients modulo powers of the ideal $\mathfrak{m}$.

It remains to note that, since these quotients are free by Lemma~\ref{lem: freequot}, injectivity can be checked after localization; hence the lemma follows from Proposition~\ref{prop: tr-iso}.
\end{proof}

\begin{definition}\label{eq:def normalized graded trace} For a Verma-type module $M_p$ with highest weight $\xi$, we define the \textit{normalized} graded trace $\widetilde{\operatorname{tr}}_{M_p}(-)$ by the formula 
\begin{equation*}
\widetilde{\operatorname{tr}}_{M_p}(-) := z^{-\xi} \operatorname{tr}_{M_p}(-).
\end{equation*}
\end{definition}

Assume $\widetilde{\mathcal{M}}_H^{\mathsf{F}}$ is finite and symplectic duality gives a bijection
\begin{equation*}
\widetilde{\mathcal{M}}_H^{\mathsf{F}} \ni p \longleftrightarrow p^\vee \in \widetilde{\mathcal{M}}_C^\theta.
\end{equation*}
\begin{conjecture}\label{conj: Verma traces as vertex functions}
For every $\tau \in \mathscr{H}_\hbar[[{\boldsymbol{z}}]]$
\begin{equation}\label{eq:equality vertex trace theta}
\widetilde{\operatorname{tr}}_{\Theta(p^\vee)}(\tau) = V^{(\tau)}_{q=2\hbar,p}.
\end{equation}
\end{conjecture}

Note that by Proposition \ref{prop: limver} below, the RHS of (\ref{eq:equality vertex trace theta}) factorizes and boils down to the case when our quiver variety is a {\emph{point}}. We expect that the LHS of (\ref{eq:equality vertex trace theta}) has the same factorization property (compare with our proof of Theorem \ref{maintheorem: ADE} in Section \ref{sec:ADE-Hilbert}).

We expect that Conjecture \ref{conj: Verma traces as vertex functions} holds more generally. 
Namely, there is no need to assume that the fixed points $\widetilde{\mathcal{M}}_H^{{\mathsf{T}}}$ are isolated.
For concreteness, let us assume the quiver $\Gamma$ is of type $ADE$ and let ${\mathsf{F}}$ be some flavor torus; for example, $\mathsf{F}$ can be trivial.

As in \cite{KTWWY}, assume that $f \in \mathsf{f}$ is {\emph{integral}}, i.e., it comes from some cocharacter $\gamma\colon \mathbb{C}^\times \rightarrow \mathsf{F}$. Let $\tilde{\gamma}$ be the following cocharacter of $\mathsf{T}$:
\begin{equation*}
\tilde{\gamma}\colon \mathbb{C}^\times \rightarrow \mathsf{T}, \quad t \mapsto (\gamma(t),t).
\end{equation*}

By the results of \cite{varprod}, the set of connected components 
$
\pi_0\big(\bigsqcup_{\mathsf{v}}\widetilde{\mathcal{M}}_H^{\tilde{\gamma}}\big)
$
has a $\mathfrak{g}_{\Gamma}$-crystal structure (so-called {\emph{monomial crystal}}). By \cite{KTWWY}, the set of Verma modules for  $\mathcal{A}_{f}$
is labeled by the same set.

\begin{example}
For $f=0$, the corresponding crystal is the one corresponding to the irreducible representation $V(\lambda)$. When $\mathsf{F}$ is the full framing torus and $\gamma$ is generic, one gets the crystal of the tensor product $\bigotimes_{k} V(\lambda_k)$, where $\la=\sum_k \lambda_k$ is the decomposition of $\lambda$ into the sum of fundamentals. 
\end{example}

Let $b$ be an element of the crystal as above and let $\Delta(b)$ be the corresponding Verma module over $\mathcal{A}_f$. Let $Z_b \subset \widetilde{\mathcal{M}}_H^{\tilde{\gamma}(\mathbb{C}^\times)}$ be the corresponding connected component and fix any point $p \in Z_{b}$.
It is natural to formulate the following refinement of Conjecture \ref{conj: Verma traces as vertex functions}. 
\begin{conjecture}\label{eq:extended_conj_tr_Verma}
\begin{equation}\label{eq:conj verma vertex more general}
\widetilde{\operatorname{tr}}_{\Delta(b)}(\tau) = V^{(\tau),\tilde{\gamma}(\mathbb{C}^\times)}_{q=2\hbar,p}.
\end{equation} 
\end{conjecture}
The index $\tilde{\gamma}(\mathbb{C}^\times)$ in (\ref{eq:conj verma vertex more general}) emphasizes that the $q=2\hbar$ specialization is taken in $\tilde{\gamma}(\mathbb{C}^\times)$-equivariant cohomology (compare with Warning \ref{war: specializations do not commute} above).

Conjecture \ref{conj: Verma traces as vertex functions} is an extreme case of Conjecture \ref{eq:extended_conj_tr_Verma} corresponding to $\mathsf{F}$ such that $p \in \widetilde{\mathcal{M}}_H^{\mathsf{F}}$ is isolated and $\gamma$ is generic.

The other extreme case is $f=0$.
Then Conjecture \ref{eq:extended_conj_tr_Verma} describes graded traces of Verma modules in the category $\mathcal{O}$ for $\mathcal{A}_{0}$, i.e., in the most singular block. Note that if $p \in \widetilde{\mathcal{M}}_H^{\mathsf{F}_{\mathsf{w}}}$ is an {\emph{isolated}} torus fixed point, then the corresponding module $\Theta(p)$ after specialization to $0 \in \mathsf{f}$ and $\hbar=1$ is still a well-defined $\mathcal{A}_0$-module.
This object will {\emph{not}} be a Verma module over $\mathcal{A}_0$. Compare with Warning \ref{war: specializations do not commute} above, see also Equation (\ref{eq:two spec are not equal}) and computations in Section \ref{sssec_kleinian_example} below.

\begin{remark}\label{rem: BT construction}
The authors of \cite{BT26} construct an action of the quantized Coulomb branch $\mathcal{A}$ on the critical cohomology of the moduli space of based quasimaps to $\widetilde{\mathcal{M}}_H$. It is natural to expect that their construction provides a categorification of Conjecture \ref{eq:extended_conj_tr_Verma}. When $\widetilde{\mathcal{M}}_H$ is a {\emph{point}}, it is not hard to identify the module constructed in \cite{BT26} (assuming it is finitely generated) with the unique Verma module in the category $\mathcal{O}$ over the corresponding quantized Coulomb branch (compare with \cite[Sections 5, 6]{quiver_slant}).   
\end{remark}

We are now ready to prove Theorem \ref{mainthm:reduction}.
\begin{theorem}\label{prop: numerical implies Hikita}
Assume Assumption \ref{ass:our_main_ass_modules_dim_est}(2) holds,  the family $\{M_p\}_{p \in S}$ is indexed by $\widetilde{\mathcal{M}}_H^{\mathsf{F}}$, and the vertex-trace identity (\ref{eq:equality vertex trace theta}) holds for this family. Then Conjecture \ref{conj: classical reformulation quantum hikita} holds. So, 
Conjecture~\ref{conj: our version quantum hikita} also holds.
\end{theorem}
\begin{proof}
Whenever the vertex--trace identity holds for a collection of Verma-type
modules indexed by isolated fixed points, part~(1) of
Assumption~\ref{ass:our_main_ass_modules_dim_est} is automatic.  Indeed,
after taking the constant term in ${\boldsymbol z}$, the map defined by the pointwise vertices is the
composite of the hyperk\"ahler Kirwan map~\eqref{eq: restr determines algebra}
with restriction to the fixed points.  After tensoring with
$\mathfrak K=\operatorname{Frac}(\mathfrak C)$, this map is surjective by
\eqref{eq:generic-fixed-locus-localization}.  Under the vertex--trace
identity, its components are precisely the characters by which
$\mathscr H_\hbar$ acts on the highest-weight lines, which gives the map in
Assumption~\ref{ass:our_main_ass_modules_dim_est}(1).

Our goal is to identify $Q^{\mathrm{PSZ}}_{q=2\hbar}$ and $\on{GrTr}(\mathcal{A})$ as quotients of $\mathscr{H}_{\hbar}[[{\boldsymbol{z}}]]$. The fact that this isomorphism is compatible with the $D$-module structures would then follow automatically from Lemma~\ref{lem: H to GrTr is surjective} and the definition of $Q^{\mathrm{PSZ}}_{q=2\hbar}$. By Lemma~\ref{lem: when traces verma determine GrTr}, after normalizing each component, we have an embedding
\begin{equation*}
\bigoplus_{p\in S}\widetilde{\operatorname{tr}}_{M_p}\colon
\operatorname{GrTr}(\mathcal{A}) \hookrightarrow \bigoplus_{p\in S} \mathfrak C[[{\boldsymbol{z}}]].
\end{equation*}
Combining it with the surjection $\mathscr{H}_\hbar[[{\boldsymbol{z}}]] \twoheadrightarrow \operatorname{GrTr}(\mathcal{A})$, we conclude that $\operatorname{GrTr}(\mathcal{A})$ identifies with the image of the map
$$
\mathscr{H}_{\hbar}[[{\boldsymbol{z}}]] \longrightarrow \bigoplus_{p\in S}\mathfrak C[[{\boldsymbol{z}}]],
\qquad \tau\longmapsto\bigl(\widetilde{\operatorname{tr}}_{M_p}(\tau)\bigr)_{p\in S}.
$$
By Corollary \ref{cor:fixed-locus-realization-PSZ}
it remains to identify the maps
\begin{align*}
&\mathscr{H}_{\hbar}[[{\boldsymbol{z}}]] \xrightarrow{V^{(-)}_{q=2\hbar}} Q_{q=2\hbar}^{\mathrm{Clas},\mathrm{fix}}=\bigoplus_{p\in\widetilde{\mathcal M}_H^T}H_T^*(p)[[{\boldsymbol{z}}]],\\
&\mathscr{H}_{\hbar}[[{\boldsymbol{z}}]] \longrightarrow \operatorname{GrTr}(\mathcal{A}) \xrightarrow{\ \bigoplus_{p\in S}\widetilde{\operatorname{tr}}_{M_p}\ } \bigoplus_{p\in S}\mathfrak C[[{\boldsymbol{z}}]],
\end{align*}
where the corresponding fixed points identify the terms in the two rightmost direct sums.

Our assumption claiming that the statement of Conjecture~\ref{conj: Verma traces as vertex functions} holds for our collection of modules precisely implies that these maps are equal.
\end{proof}

To summarize, Theorem~\ref{prop: numerical implies Hikita} allows us to deduce our refined version of the quantum Hikita conjecture from the ``numerical'' Conjecture~\ref{conj: Verma traces as vertex functions}, combined with the dimension estimate (Assumption \ref{ass:our_main_ass_modules_dim_est}(2)), which holds automatically, for example, if the {\emph{original}} (non-equivariant and non-quantized) Hikita conjecture (\ref{eq: Hikita conj}) holds.

\subsection{Proof of Theorem \ref{maintheorem: ADE}}\label{sec:ADE-Hilbert}
By Theorem
\ref{prop: numerical implies Hikita}, it is enough to verify Conjecture
\ref{conj: Verma traces as vertex functions} and Assumption
\ref{ass:our_main_ass_modules_dim_est}(2).  We deal with these two statements
separately.  The calculation of the vertex functions used below is carried
out in Theorem~\ref{thm: pointvertex}.

Write
$$
    \lambda=\lambda_1+\cdots+\lambda_r,
$$
where every $\lambda_a$ is minuscule, and set
$$
 S_{\lambda,\mu}:=
 \left\{\underline{\mu}=(\mu_1,\ldots,\mu_r)\ \middle|\
 \mu_a\in W\lambda_a,\quad \sum_{a=1}^r\mu_a=\mu\right\}.
$$

\begin{proposition}\label{prop:ADE-Verma-vertex}
Let $I$ be an ADE quiver and suppose that the framing is supported at
minuscule vertices.  Then Conjecture
\ref{conj: Verma traces as vertex functions} holds for the modules
$M_{\mathsf f}(\underline{\mu})$, $\underline{\mu}\in S_{\lambda,\mu}$.
\end{proposition}

\begin{proof}
We first reduce to the case $r=1$.  By \eqref{def_M_module},
$M_{\mathsf f}(\underline{\mu})$ is the tensor product of the minuscule
chamber modules $L_{\mu_a,\mathsf f^{(a)}}$.  Proposition
\ref{tensorprod}, together with the additivity of the highest weight, shows
that the normalized graded trace is multiplicative under this tensor
product.  On the Higgs side, the corresponding decomposition, see Example \ref{fixpoints}, is 
$$
 \widetilde{\mathcal M}_{H,\theta}(\mathsf v,\mathsf w)^{\mathsf A}
 =\bigsqcup_{\underline{\mu}\in S_{\lambda,\mu}}
   \prod_{a=1}^r
   \widetilde{\mathcal M}_{H,\theta}
   (\mathsf v^{(a)},\mathsf w^{(a)}).
$$
Proposition~\ref{prop: limver} below gives the same factorization for the
$q=2\hbar$ specialized vertex function.  Notice that the normalization is
compatible with this factorization: if the two highest weights are
$\xi_1,\xi_2$, then the highest weight of the tensor product is
$\xi_1+\xi_2$, and hence
$z^{-(\xi_1+\xi_2)}=z^{-\xi_1}z^{-\xi_2}$.

It remains to consider a minuscule $\lambda$ and $\mu\in W\lambda$.  The
corresponding quiver variety is a point.  If $w$ is the minimal Weyl-group
element such that $\mu=w\lambda$, the normalized trace of the minuscule
chamber module is the sum over {\emph{reverse plane partitions}} on the heap
$H(w)$ (see Section \ref{ssec: dominant minuscule heaps} for the details). This is proven in \cite{min_chamber_mod} by using the KLRW-approach to the representations of Coulomb branches (see \cite{cat_O_slices_and_categor}) and then reducing the question to \cite{kleshchev-ram}.  For $G=AD$, this is also proven in \cite{krylov-klyuev-minu} by completely different techniques, namely by identifying this trace with the sphere trace $T_{\mathsf{sph}}$ (the identification also works for type $E$) and then computing $T_{\mathsf{sph}}$ by residues. 
Theorem~\ref{thm: pointvertex} below gives precisely the
same sum (over reverse plane partitions on the heap $H(w)$) for the $q=2\hbar$ specialization of the descendant vertex function.  The insertion and the
Cartan generators are evaluated on the same collection
$\phi \in\operatorname{rpp}(H(w))$, so the equality holds as an equality
of functionals on $\mathscr{H}_\hbar[[{\boldsymbol{z}}]]$.  This proves the proposition.
\end{proof}

\begin{proposition}\label{prop:ADE-dimension-assumption}
Let $I$ be an ADE quiver and suppose that the framing is supported at
minuscule vertices.  Then Assumption
\ref{ass:our_main_ass_modules_dim_est}(2) holds for the collection
$$
 \bigl\{M_{\mathsf f}(\underline{\mu})\ \bigm|\
       \underline{\mu}\in S_{\lambda,\mu}\bigr\}.
$$
\end{proposition}

\begin{proof}

The Hikita isomorphism for generalized affine-Grassmannian
slices gives
$$
 \mathbb C[\mathcal M_C^{\theta}]
 \simeq H^*(\widetilde{\mathcal M}_H).
$$
When $\mu$ is dominant, this is
\cite[Theorem~8.1]{KTWWY}.  For arbitrary $\mu$, which is needed here
because we do not impose the conicity assumption, it is
\cite[Theorem~A.7]{dk}.  The restrictions in types $E_7$ and $E_8$ in the
 references are automatic when the framing is supported at
minuscule vertices.  The framing torus has isolated fixed points indexed
by $S_{\lambda,\mu}$; hence equivariant localization and the freeness of
equivariant cohomology give
$$
 \dim_{\mathbb C}\mathbb C[\mathcal M_C^{\theta}]
 =\dim_{\mathbb C}H^*(\widetilde{\mathcal M}_H)
 =|S_{\lambda,\mu}|.
$$
This proves part~(2), in fact with equality.
\end{proof}

We now take the Jordan quiver with $\mathsf v=n$ and $\mathsf w=r$, so that
the Higgs branch is the ADHM space $\qv(n,r)$ (also known as the Gieseker variety). Let
$\mathsf P_r(n)$ denote the set of $r$-tuples of partitions
$\ul{\mathsf Y}=(\mathsf Y^{(1)},\ldots,\mathsf Y^{(r)})$ such that
$\sum_s|\mathsf Y^{(s)}|=n$.

\begin{proposition}\label{prop:Gieseker-Verma-vertex}
Conjecture~\ref{conj: Verma traces as vertex functions} holds for the
ADHM space $\qv(n,r)$.
\end{proposition}

\begin{proof}
Both sides are indexed by $\ul{\mathsf Y}\in\mathsf P_r(n)$. If
$\mathsf v(\mathsf Y^{(s)})$ records the number of boxes of each content
in $\mathsf Y^{(s)}$, the component of the flavor-torus fixed locus
containing the fixed point $p_{\ul{\mathsf Y}}$ is the product
$$
\prod_{s=1}^{r}
\qv_{A_\infty}\bigl(\mathsf v(\mathsf Y^{(s)}),\delta_0\bigr);
$$
see Example~\ref{fixpointsGieseker}. Proposition~\ref{prop: limver}
therefore identifies the vertex function at $p_{\ul{\mathsf Y}}$ with
the vertex function of this product.

By \cite[Theorem~1.1]{KN18} (see also \cite[Theorem 4.1]{Webster2019CyclotomicCherednik} and \cite{BravermanEtingofFinkelberg2020CyclotomicDAHA}), the quantized Coulomb branch of the Jordan
quiver with framing $r$ is identified with the spherical cyclotomic
rational Cherednik algebra of type $G(r,1,n)$. 
We
use $S_n \ltimes (\mathbb{Z}/r\mathbb{Z})^n$-invariants of the Verma module $\Delta(\ul{\mathsf Y})$ over the corresponding cyclotomic rational Cherednik algebra (compare with \cite[Section 2]{DG10}). Its generalized
Jack basis and the eigenvalues of the Cartan generators are given by
\cite[Theorem~1.3 and the proof of Theorem~2.2]{DG10}. This basis is
naturally indexed by $r$-tuples of reverse plane partitions on the
Young diagrams $\mathsf Y^{(s)}$. Comparing the corresponding
eigenvalues with Theorem~\ref{thm: pointvertex} identifies, term by
term, the normalized trace of $\Delta(\ul{\mathsf Y})$ with the
$q=2\hbar$ specialized vertex function at
$p_{\ul{\mathsf Y}}$. This proves the claim.
\end{proof}

\begin{remark} The physical incarnation of the results from \cite{DG10} cited above is given in \cite{GO,GO24}. Moreover, as far as we understand, the upcoming results of \cite{BT26} should give a geometric proof.
\end{remark}

\begin{proposition}\label{prop:Hilb-dimension-assumption}
For the Jordan quiver 
Assumption
\ref{ass:our_main_ass_modules_dim_est}(2) holds.
\end{proposition}

\begin{proof}

This is \cite[Corollary A.9]{krylov_shykov} going back to \cite[Lemma 2.1.4]{Hatano2021FramedModuli}.
\end{proof}

\subsection{Application}\label{sec:graded traces vs twisted traces}
\subsubsection{Graded traces vs twisted traces} Let us discuss the relation between $\on{GrTr}(\mathcal{A})$ and so-called twisted traces which we have briefly mentioned in Section \ref{ssec:intro Coulomb twisted and graded}. This will also shed some light on the definition of this object.

In our setting, we fix $t\in\mathsf A$ and regard it as an automorphism of $\mathcal A=\mathcal A_{f}$ for some $f \in \mathsf{f}$. A $t$-twisted trace is a functional
\begin{equation*}
\operatorname{Tr}_t\colon \mathcal{A} \rightarrow \mathbb{C}
\end{equation*}
such that
\begin{equation*}
\on{Tr}_t(a \cdot b) = \on{Tr}_t(b \cdot t(a))\qquad\text{for all }a,b \in \mathcal{A}.
\end{equation*}

The following simple lemma motivates the definition of $\operatorname{GrTr}(\mathcal{A})$ above.
\begin{lemma}\label{lem:twisted trace determined by A 0}
For $a \in \mathcal{A}^\eta$ such that $\eta \neq {\boldsymbol{0}}$, we have $\operatorname{Tr}_t(a)=0$. In other words, $\operatorname{Tr}_t$ is determined by its restriction to $\mathcal{A}^{\boldsymbol{0}}$.
\end{lemma}
\begin{proof}
Pick $x\in\mathsf a$ such that $\langle\eta,x\rangle \neq 0$. Since $t$ fixes the image of the quantum comoment map, we have
\begin{equation*}
\operatorname{Tr}_t(\iota_{\mathsf a}(x) \cdot a) = \operatorname{Tr}_t(a \cdot t(\iota_{\mathsf a}(x)))=\operatorname{Tr}_t(a \cdot \iota_{\mathsf a}(x)),
\end{equation*}
so $\operatorname{Tr}_t([\iota_{\mathsf a}(x),a])=0$. Since $[\iota_{\mathsf a}(x),a]=2\langle\eta,x\rangle a$, we conclude that $\operatorname{Tr}_t(a)=0$, as desired.
\end{proof}

Looking at the definition of $\on{GrTr}(\mathcal{A})$, we see that $\mathbb{C}[[{\boldsymbol{z}}]]$-linear functionals on $\on{GrTr}(\mathcal{A})$ are precisely ``formal'' versions of twisted traces on $\mathcal{A}$. Namely, let us start with a $\mathbb{C}[[{\boldsymbol{z}}]]$-linear functional
\begin{equation}\label{eq: grtr to c z}
\operatorname{Tr}\colon \on{GrTr}(\mathcal{A}) \rightarrow \mathbb{C}[[{\boldsymbol{z}}]],
\end{equation}
The target of (\ref{eq: grtr to c z}) may vary depending on the context; compare with Warning \ref{intro:warning target tr}.

For every $x \in \on{GrTr}(\mathcal{A})$, the value $\operatorname{Tr}(x)$ is a formal power series in ${\boldsymbol{z}}$. Whenever the values of $\operatorname{Tr}$ admit a well-defined evaluation at ${\boldsymbol{z}}=t$, we obtain the actual $t$-twisted trace by specializing ${\boldsymbol{z}}=t$.
By Lemma~\ref{lem:twisted trace determined by A 0}, any $t$-twisted trace is uniquely determined by its restriction to $\mathcal{A}^{\boldsymbol{0}}$. Thus, whenever it is a specialization of a functional on $\operatorname{GrTr}(\mathcal{A})$, this functional determines it {\emph{uniquely}}. Motivated by this, we call functionals on $\operatorname{GrTr}(\mathcal{A})$ {\emph{graded traces}}.

One important twisted trace was predicted by Gaiotto--Okazaki, who call it the {\emph{sphere trace}}. It should be thought of as a {\emph{correlation function}} of the corresponding gauge theory.
It follows from \cite{zhang-analytical-traces-coulomb} for type $A$ quivers and from \cite{Klyuevres} in general that the functional $T_{\mathrm{sph}}$ is indeed well-defined for {\emph{good or ugly}} theories, as Gaiotto--Okazaki predicted. This functional is given by an explicit integral (see \cite[Equation (2.17)]{GO}).

\subsubsection{An application to graded traces}\label{sec:application of our results}
Let us briefly mention an application of the main result of the paper. It would be interesting to study this in more detail.

Recall that one of the main results of the current paper is an isomorphism of $D$-modules $Q_{q=2\hbar}^{\mathrm{PSZ}} \iso \operatorname{GrTr}(\mathcal{A})$ making the following diagram commutative:

 \begin{equation}\label{eq:our main iso section applications}
\begin{tikzcd}
	& {\GZ_{\hbar}[[{\boldsymbol{z}}]]} & \\
	{Q_{q=2\hbar}^{\mathrm{PSZ}}} && {\operatorname{GrTr}(\mathcal{A}).}
	\arrow["\widehat{V}_{q=2\hbar}^{(-)}"', from=1-2, to=2-1]
	\arrow[from=1-2, to=2-3]
	\arrow["\simeq", from=2-1, to=2-3]
\end{tikzcd}
\end{equation}

One may study specializations of this isomorphism to various points $(f, \hbar_0) \in \mathsf f \oplus \mathbb C\hbar$. Our proofs in this paper apply only to specializations at \textit{generic} points, that is, parameters for which the corresponding fixed locus consists of isolated points. In this case the category $\mathcal{O}_{\theta}(\mathcal{A}_{f,\hbar_0})$ is ``large enough''
in the following sense: for such a generic $(f, \hbar_0)$, the natural morphism
\begin{equation}\label{KGrTr}
K_0(\mathcal{O}_{\theta}(\mathcal{A}_{f,\hbar_0}))\otimes_{\mathbb Z}\mathbb C[[{\boldsymbol{z}}]] \rightarrow \operatorname{Hom}^{\mathrm{cont}}_{\mathbb{C}[[{\boldsymbol{z}}]]}\bigl(\operatorname{GrTr}(\mathcal{A}_{f,\hbar_0}),\mathbb C[[{\boldsymbol{z}}]]\bigr),
\end{equation}
given on the classes of Verma modules by the formula 
\begin{equation*}
[M] \mapsto \widetilde{\operatorname{tr}}_{M}(-),
\end{equation*}
is an isomorphism of $\mathbb C[[{\boldsymbol{z}}]]$-modules.
In other words, for generic values of the parameters, every $\mathbb{C}[[{\boldsymbol{z}}]]$-linear functional on $\on{GrTr}(\mathcal{A}_{f,\hbar_0})$ is a linear combination of normalized graded traces of Verma modules with coefficients in $\mathbb C[[{\boldsymbol{z}}]]$.

However, the specialization to $(f, \hbar_0)$ makes sense for \textit{all} pairs $(f, \hbar_0)$. It is especially interesting for \textit{special} pairs, that is, those for which the fixed locus is nonisolated. Among those, the most interesting are the {\emph{integral}} ones, namely those coming from a cocharacter
\begin{equation*}
\tilde{\gamma}\colon \mathbb{C}^\times \rightarrow \mathsf{F} \times \mathbb{C}^\times_\hbar,\quad t \mapsto (\gamma(t),t)
\end{equation*}
for some $\gamma\colon \mathbb{C}^\times \rightarrow \mathsf{F}$ as above.

In this case, the morphism~\eqref{KGrTr} is not surjective, so there is {\emph{no}} way to read off all graded traces from the set of connected components of the fixed locus. Namely, whenever Conjecture \ref{conj: our version quantum hikita} holds, it automatically implies that functionals
\begin{equation}\label{eq:vertex gives solution of grtr application}
\mathscr{H}_\hbar[[{\boldsymbol{z}}]] \ni \tau    \mapsto V^{(\tau),\tilde{\gamma}(\mathbb{C}^\times)}_{2\hbar=q=2,p}, \quad p \in Z \in \pi_0\Big(\widetilde{\mathcal{M}}_H^{\tilde{\gamma}(\mathbb{C}^\times)}\Big)
\end{equation}
considered in Conjecture \ref{eq:conj verma vertex more general} factor through $\operatorname{GrTr}(\mathcal{A}_f)$ and give {\emph{certain}} solutions of this $D$-module. 

Conjecture \ref{conj: our version quantum hikita} actually gives a geometric description of all graded traces on $\operatorname{GrTr}(\mathcal{A}_f)$ via the {\emph{capped}} vertex.  Recall that by \cite[Theorem 7.3.5]{Nak_quiver_and_fd_reps} there exists a perfect ``integration'' pairing
\begin{equation*}
\int\colon H_*(\widetilde{\mathcal{M}}_H^{\tilde{\gamma}(\mathbb{C}^\times)}) \otimes H_*^T(\mathfrak{L}^{\tilde{\gamma}(\mathbb{C}^\times)}) \rightarrow \mathbb{C},
\end{equation*}
where $\mathfrak L\subset\widetilde{\mathcal M}_H$ is the Lagrangian fiber over the cone point $\widetilde{\mathcal{M}}_H \rightarrow \mathcal{M}_H$.

In particular, using the identification (\ref{eq:our main iso section applications}) between $\operatorname{GrTr}(\mathcal{A}_f)$ and the specialized quantum $D$-module, we get the following corollary describing graded traces as functionals on $\mathscr{H}_\hbar$.

 \begin{corollary}\label{cor: description of graded traces}
 Every graded trace on $\mathcal{A}_{f}$ is given by
\begin{equation*}
\mathscr{H}_\hbar \ni \tau \mapsto \int_{\widetilde{\mathcal M}_H^{\tilde{\gamma}(\mathbb{C}^\times)}} \widehat{V}^{(\tau)}_{2\hbar=q=2}\cap a
\end{equation*}
for some $a \in H_*(\mathfrak{L}^{\tilde{\gamma}(\mathbb{C}^\times)})[[{\boldsymbol{z}}]]$. In particular, every graded trace has a natural $q$-deformation given by the same formula but without imposing the $q=2\hbar$ specialization.
 \end{corollary}

More generally, changing the target in (\ref{eq: grtr to c z}) by some other complete $D$-module $\mathcal{Z}$ (such as $\mathcal{Z}_\xi$ or $\mathcal{Z}_\Xi$) simply corresponds to replacing $H_*(\mathfrak{L}^{\tilde{\gamma}(\mathbb{C}^\times)})[[{\boldsymbol{z}}]]$ by $H_*(\mathfrak{L}^{\tilde{\gamma}(\mathbb{C}^\times)}) \otimes_{\mathbb{C}} \mathcal{Z}$.

In particular, in the terminology of Section~\ref{sec:graded traces vs twisted traces}, it follows that there exists a class $a^{\mathrm{sph}}\in H_*(\mathfrak L^{\tilde{\gamma}(\mathbb{C}^\times)}) \otimes_{\mathbb{C}} \mathcal{Z}$
such that
\begin{equation}\label{eq:T sph via vertex}
T_{\mathrm{sph}}(\tau) = \int_{\widetilde{\mathcal M}_H^{\tilde{\gamma}(\mathbb{C}^\times)}} \widehat{V}^{(\tau)}_{2\hbar=q=2} \cap a^{\mathrm{sph}}.
\end{equation}
It would be interesting to describe this class $a^{\mathrm{sph}}$.

\begin{remark}
One particularly intriguing aspect of this discussion is that $T_{\mathrm{sph}}$ does not always exist: because of some \textit{analytic} issues, it is defined only for good or ugly theories.

On the other hand, the capped vertex function exists whenever the quiver variety is nonempty. 
\end{remark}

Let us finally note that $V^{(-)}$ is {\emph{much}} simpler to compute than $\widehat{V}^{(-)}$ so it would be desirable to have a description of graded traces on $\mathcal{A}_f$ similar to Corollary \ref{cor: description of graded traces} but {\emph{purely}} in terms of $V^{(-)}$ (compare with (\ref{eq:vertex gives solution of grtr application}) above). The main issue is that $V^{(-)}$ itself does not have a $q=2\hbar$-limit (see Example \ref{ex:vertex has pole at hbar q} above). It could still be true that for appropriately chosen classes $a \in H_*(\mathfrak{L}^{\tilde{\gamma}(\mathbb{C}^\times)})(q)[[{\boldsymbol{z}}]]$ 
the integrals
\begin{equation*}
\int_{\widetilde{\mathcal M}_H^{\tilde{\gamma}(\mathbb{C}^\times)}} V^{(\tau)}\cap a
\end{equation*}
have $2\hbar=q=2$ specializations and every graded trace can be obtained this way. In particular, it is possible that one should seek (\ref{eq:T sph via vertex}) with $\widehat{V}^{(\tau)}$ replaced by $V^{(\tau)}$. It would be interesting to figure this out, starting with the hypertoric case.

\begin{remark}
Note that \cite[Section 1.1]{GO} 
mentions that there should be a relation between sphere traces and the quantum Hikita conjecture. We hope that the present Section~\ref{sec:application of our results} provides a framework for this relation.
\end{remark}

\section{Quantum \texorpdfstring{$D$}{D}-module and vertex functions}\label{sec: vertex}

\subsection{Nakajima quiver varieties}\label{sec: quiver}

Retain the notation of Section~\ref{ssec: quiver gauge theories}. For the remainder of Section~\ref{sec: vertex}, we fix $\Gamma$, $\dv$, and $\dw$, take $\theta=(1,\ldots,1)$, and write
$$
\begin{aligned}
\stackqv&=[\mu^{-1}(0)/G_{\dv}],\\
X&=\qv(\dv,\dw)=\widetilde{\mathcal M}_H\\
&=T^*{\bf N}(\dv,\dw)/\!\!/\!\!/_\theta G_{\dv}\\
&=\mu^{-1}(0)^{\theta\text{-}\mathrm{st}}/G_{\dv}.
\end{aligned}
$$

By the trace pairing, we have $\Hom(V,V')^{*} \cong \Hom(V',V)$ so that we can denote a general element of 
$$
T^{*}{\bf N}(\dv,\dw)\cong {\bf N}(\dv,\dw)\oplus {\bf N}(\dv,\dw)^{*}
$$
by a quadruple 
$$
(\{Y_{e}\}_{e \in E}, \{Z_{e}\}_{e \in E}, \{A_{i}\}_{i \in I}, \{B_{i}\}_{i \in I})
$$
where $Y_{e} \in \Hom(V_{t(e)},V_{h(e)})$, $Z_{e} \in \Hom(V_{h(e)},V_{t(e)})$, $A_{i} \in \Hom(W_{i},V_{i})$ and $B_{i} \in \Hom(V_{i},W_{i})$. We abbreviate this by $(Y,Z,A,B)$.

The stability condition is described explicitly by the following proposition.

\begin{proposition}[\cite{GinzburgLectures}, Proposition 5.1.5]
A quadruple $(Y,Z,A,B)\in \mu^{-1}(0)$ is $\theta$-stable if and only if there is no collection of subspaces $T_i\subseteq V_i$, not all equal to $V_i$, preserved by $Y$ and $Z$ and satisfying $T_i\supseteq\Ima A_i$.
\end{proposition}

Fix a torus
$$
\mathsf F\subset \operatorname{Aut}_{G_{\dv}}({\bf N}),
$$
where $\operatorname{Aut}_{G_{\dv}}({\bf N})$ denotes the group of linear automorphisms of ${\bf N}$ commuting with $G_{\dv}$. We call such an $\mathsf F$ a \emph{flavor torus}. Its cotangent lift to $T^*{\bf N}$ preserves the moment map and the $\theta$-stable locus, and hence induces an action on $X$. An additional $\mathbb{C}^{\times}_{\hbar}$ acts on $X$ by the formula
 \begin{equation}\label{action_hbar_quiver}
t \cdot (Y,Z,A,B) = (tY,tZ,t^2A,B).
\end{equation}
We denote by $\hbar$ the inverse of the standard character of
$\mathbb C^\times_{\hbar}$. Thus, for the action
\eqref{action_hbar_quiver}, the coordinates $Y,Z,A,B$ have respective
weights
$$
\hbar^{-1},\qquad \hbar^{-1},\qquad \hbar^{-2},\qquad 1.
$$
In particular, the symplectic form has weight $\hbar^{-2}$. Thus,
$\bT=\mathsf F\times\mathbb C^\times_{\hbar}$ acts on $\qv(\dv,\dw)$.

For $j\in I$, let $\tb_j$ be the rank $\dv_j$ vector bundle on
$[\mu^{-1}(0)/G_{\dv}]$ defined by
\begin{equation}\label{eq: induced bundle}
\tb_j=(\mu^{-1}(0)\times V_j)/G_{\dv},
\end{equation}
where $G_{\dv}$ acts by
$$
g\cdot(x,v)=(g^{-1}x,g_j^{-1}v).
$$
By abuse of notation, we denote by the same symbol its restriction to
the open substack
$$
\qv(\dv,\dw)\subset[\mu^{-1}(0)/G_{\dv}].
$$
The bundle $\tb_j$ is naturally $\bT$-equivariant, with equivariant
structure induced by the action on $\mu^{-1}(0)$. Similarly, the vector
spaces $W_i$ determine topologically trivial equivariant bundles
$\tbw_i$.

A polarization of $X$ compatible with the above action is
$$
T^{1/2}
=
\hbar^{-1}\sum_{e\in E}
\Hom(\tb_{t(e)},\tb_{h(e)})
+
\hbar^{-2}\sum_{i\in I}
\Hom(\tbw_i,\tb_i)
-
\sum_{i\in I}\Hom(\tb_i,\tb_i).
$$
It satisfies
$$
TX=T^{1/2}+\hbar^{-2}(T^{1/2})^\vee
\qquad\text{in }K_{\bT}(X).
$$

\subsection{Quasimaps}\label{quasimaps}

We will use quasimap theory, which we now briefly review, to define our quantum $D$-module. See \cite{qm} and \cite{OkLec} for more details.

From the discussion in Section~\ref{sec: quiver}, $X$ is the stable open substack of the quotient stack $\stackqv$.

\begin{definition}
Let $C$ be a connected projective curve having at worst nodal singularities. A quasimap from $C$ to $X$ is a morphism
$$
f\colon C\longrightarrow\stackqv.
$$
It is called stable if the generic point of every irreducible component of $C$ is mapped to $X$.

More explicitly, let $\qmtbw_{i}=W_{i}\times C$ be the rank $\dw_i$ trivial vector bundle on $C$. A quasimap consists of the data $f=(s,\{\qmtb_{i}\}_{i \in I})$, where
\begin{itemize} 
    \item $\qmtb_{i}$ is a rank $\dv_{i}$ vector bundle over $C$ for $i \in I$ and
    \item $s \in H^{0}\left(C, \mathscr{M} \oplus \hbar^{-2} \mathscr{M}^{\vee} \right)$, where
    $$
\mathscr{M}=\hbar^{-1}\bigoplus_{e \in E} \Hom(\qmtb_{t(e)},\qmtb_{h(e)}) \oplus \hbar^{-2}\bigoplus_{i \in I} \Hom(\qmtbw_{i},\qmtb_{i})
    $$
\end{itemize}
such that $s(p)$ satisfies the moment map equations for all $p \in C$. Equivalently, the quasimap is stable if $s(p)$ is $\theta$-stable away from finitely many points of $C$. A point $p\in C$ is called a nonsingular point of the quasimap if $s(p)$ is $\theta$-stable.

The degree of a quasimap is defined to be 
\begin{equation*}
\deg (s,\{\qmtb_{i}\}_{i \in I})= ( \deg \qmtb_{i})_{i \in I} \in \mathbb{Z}^{I}.
\end{equation*} 
\end{definition}

Let us now explain the types of quasimap moduli spaces that we will consider in this paper, following \cite[Definition 7.2.1]{qm}.

Let $C_0=\mathbb P^1$ and let
$S_{\mathrm{ns}},S_{\mathrm{rel}}\subset C_0$ be disjoint finite
sets.

\begin{definition}
An {\emph{allowed}} source curve is a connected nodal curve $C$ of
arithmetic genus zero together with a regular map
$$
\pi\colon C\longrightarrow C_0
$$
and marked points $\hat p\in C$ for $p\in S_{\mathrm{rel}}$ satisfying
the following conditions.  There is a distinguished irreducible
component $C_{\mathrm{par}}\subset C$ such that
$$
\pi|_{C_{\mathrm{par}}}\colon C_{\mathrm{par}}
\xrightarrow{\sim} C_0
$$
is an isomorphism.  The map $\pi$ is an isomorphism over
$C_0\setminus S_{\mathrm{rel}}$, while for each
$p\in S_{\mathrm{rel}}$ it contracts a possibly empty chain of
rational curves attached to $C_{\mathrm{par}}$ at $p$.  The point
$\hat p$ is a smooth point at the free end of this chain; when the
chain is empty, $\hat p$ is the point of $C_{\mathrm{par}}$
corresponding to $p$. 
\end{definition} 

We can now define the quasimap moduli space we will be dealing with.
\begin{definition}
$\qm_{\ns S_{\mathrm{ns}},\rel S_{\mathrm{rel}}}$ is the moduli
space parameterizing an allowed source curve $C$ together with a stable
quasimap $f\colon C\dashrightarrow X$ which is nonsingular at:
\begin{itemize}
    \item the
points of $C_{\mathrm{par}}$ corresponding to $S_{\mathrm{ns}}$, 
    \item the points $\hat p$ for $p\in S_{\mathrm{rel}}$, 
    \item and every node of
$C$,
\end{itemize}
and whose automorphism group over $C_0$ is finite.
\end{definition}

Fix $S_{\mathrm{ns}}$ and $S_{\mathrm{rel}}$ and denote $\qm:=\qm_{\ns S_{\mathrm{ns}}, \rel S_{\mathrm{rel}}}$. For $d \in \mathbb{Z}^{I}$, let $\qm^{d}\subset \qm$ be the substack of degree $d$ quasimaps.

Let the one-dimensional torus $\mathbb{C}^{\times}_{q}$ act on $\mathbb{P}^{1}$ with $(\mathbb{P}^{1})^{\mathbb{C}^{\times}_{q}}=\{p_1,p_2\}$, and denote the weight of the tangent space at $p_2$ by $q^{-1}$. 
From now on we assume that $S_{\mathrm{ns}},S_{\mathrm{rel}}\subset\{p_1,p_2\}$.

\begin{lemma} The actions of $\bT$ on $X$ and of $\mathbb{C}^{\times}_{q}$ on $\mathbb{P}^{1}$ induce an action of $\bT_{q}:=\bT \times \mathbb{C}^{\times}_{q}$ on $\qm^{d}$. 
\end{lemma}
\begin{proof}
Let us show that the $\mathbb C_q^\times$-action on $C_0=\mathbb{P}^1$  lifts canonically to
every allowed source curve $C$.  On $C_{\mathrm{par}}$ it is transported
through the isomorphism
$\pi|_{C_{\mathrm{par}}}\colon C_{\mathrm{par}}\xrightarrow{\sim}C_0$,
and it acts trivially on every component contracted by $\pi$.  These
actions glue because all attaching points are fixed by
$\mathbb C_q^\times$, and they make $\pi$ equivariant.  Together with
the $\bT$-action on $X$, pullback by this action defines the desired action of $\bT_q$
on $\qm^d$.
\end{proof}

There are three types of evaluation morphisms:
\begin{itemize}
    \item $\ev_{p}: \qm \to X$ for $p \in S_{\mathrm{ns}}$, 
    \item $\widehat{\ev}_{p}: \qm \to X$ for $p \in S_{\mathrm{rel}}$,
    \item $\ev_{p}: \qm \to \stackqv$ for $p \in C_{0} \setminus (S_{\mathrm{ns}} \cup S_{\mathrm{rel}})$.
\end{itemize}

If $S_{\mathrm{ns}}=\emptyset$, then the evaluation morphisms of the second type are proper. In general, the restrictions of the first two evaluation morphisms to $(\qm^{d})^{\mathbb{C}^{\times}_{q}}$ are proper. Thus, the pushforwards in equivariant cohomology can be defined by localization.

\begin{proposition}[\cite{qm}]\label{prop:pot}
 The moduli space $\qm^{d}_{\ns S_{\mathrm{ns}}, \rel S_{\mathrm{rel}}}$ is a Deligne--Mumford stack of finite type, equipped with a perfect obstruction theory.
\end{proposition}

We review the definition of the obstruction theory. Let $\operatorname{Bun}_{G_{\dv}}$ be the stack parameterizing an allowed source curve together with a $G_{\dv}$-bundle on it. Let $\pi\colon \mathcal C \to \qm$ be the universal curve, $\mathfrak P$ the universal $G_{\dv}$-bundle over the universal curve, $u\colon \mathcal C \to \mathfrak P \times_{G_{\dv}} \mu^{-1}(0)$ the universal quasimap, $\varrho\colon \mathfrak P \times_{G_{\dv}} \mu^{-1}(0) \to \mathcal C$ the projection (with fiber isomorphic to $\mu^{-1}(0)$), and $\mathbb{T}_{\varrho}$ the relative tangent complex of $\varrho$. There is a forgetful morphism $\operatorname{forg}\colon \qm \to \operatorname{Bun}_{G_{\dv}}$. The $\operatorname{forg}$-relative obstruction theory (cf.~\cite[Section 4.5]{qm}) is defined by $R\pi_*(u^*\mathbb{T}_{\varrho})$.

Taking into account the obstruction theory on $\operatorname{Bun}_{G_{\dv}}$, we arrive at the following formula for the virtual tangent space to a quasimap $f=(s,\{\qmtb_{i}\}_{i \in I})$ when $S_{\mathrm{rel}}=\emptyset$ (cf.~\cite[Section 4.3.16]{OkLec}):

\begin{align}
\qmtan|_{f} &= H^*(C_{0},\mathscr M \oplus \hbar^{-2}\mathscr M^\vee) - (1+\hbar^{-2})\sum_{i \in I} \operatorname{Ext}^*(\qmtb_{i},\qmtb_{i}) \nonumber \\
&= H^{*}(C_{0}, \qmpol\oplus \hbar^{-2}(\qmpol)^{\vee})\label{Tvir},
\end{align}
where  
$$
\qmpol:=\hbar^{-1}\sum_{e\in E} \Hom(\qmtb_{t(e)},\qmtb_{h(e)}) +\hbar^{-2}\sum_{i \in I} \Hom(\qmtbw_{i},\qmtb_{i}) - \sum_{i \in I} \Hom(\qmtb_{i},\qmtb_{i})
$$
and $\qmtb_i$, $\qmtbw_{i}$, and $\mathscr M$ are as in the definition of a quasimap. 

When $S_{\mathrm{rel}} \neq \varnothing$, the virtual tangent space has additional terms arising from deformations of the curve $C$.

By Proposition \ref{prop:pot}, there exists a virtual fundamental class on $\qm^{d}$. To save on notation, we will denote this by $\vrs$, relying on context to determine what nonsingular/relative conditions are assumed.

Finally, we set $\eff(X):=\eff(\mathfrak X):=\mathbb{Z}_{\geqslant 0}^I \subset \mathbb Z^I$.

The following lemma shows that every ordinary or relative quasimap has
degree in $\eff(X)$.  In particular, generating functions that encode the enumerative invariants of quasimaps are supported
on effective degrees.

For $\mathbb C_q^\times$-fixed nonsingular quasimaps, the effectivity
statement also follows from the fixed-point description in
\cite[Lemma~7.2.10 and Equation~7.2.14]{OkLec}.  The elementary argument
below supersedes that fixed-point argument in the present setting: it
does not use $\mathbb C_q^\times$-fixedness and applies uniformly to
ordinary and relative quasimaps.

\begin{lemma}
\begin{enumerate}
\item[(a)] If $\qm^d\neq\varnothing$, then $d\in\eff(X)$.
\item[(b)] If $\qm_{\rel p_2}^d\neq\varnothing$, then
$d\in\eff(X)$.
\end{enumerate}
\end{lemma}

\begin{proof}
Let $f=(s,\{\qmtb_i\}_{i\in I})$ be either an ordinary or a relative
quasimap, and let $D$ be an irreducible component of its source.  Fix
$i\in I$.  At the generic point of $D$, stability implies that $V_i$ is
generated by the images of the framing maps under compositions of the
maps $Y_e$ and $Z_e$.  Consequently, finitely many paths $\gamma$ in
the doubled quiver give a generically surjective morphism
$$
\Phi_{i,D}\colon
\bigoplus_{j,\gamma}\qmtbw_j|_D
\longrightarrow
\qmtb_i|_D,
\qquad
\left.\Phi_{i,D}\right|_\gamma
=
\gamma(Y,Z)\circ A_j.
$$
The bundle on the left is trivial.  Hence a nonzero maximal minor of
$\Phi_{i,D}$ gives a nonzero global section of
$\det(\qmtb_i|_D)$, and therefore
$$
\deg_D(\qmtb_i|_D)\geqslant0.
$$
Since degree is additive over the irreducible components of a nodal
curve,
$$
d_i=\deg\qmtb_i
=
\sum_{D\subset C}\deg_D(\qmtb_i|_D)
\geqslant0.
$$
This holds for every $i\in I$, so $d\in\eff(X)$.
\end{proof}

\subsection{Vertex functions}\label{sec:multcap} 
Quasimap spaces are used to construct ``enumerative invariants'' of $X$.

For a ring $R$, let $R[[{\boldsymbol{z}}]]=\Big\{\sum_{d\in \eff(X)} r_{d} z^{d} \, \mid \, r_{d} \in R\Big\}$. Let
$$
H^*_{\bT_q}(X)_{\loc}:=H^*_{\bT_q}(X)\otimes_{H^*_{\bT_q}(\mathrm{pt})}\operatorname{Frac}\bigl(H^*_{\bT_q}(\mathrm{pt})\bigr).
$$
and
$$
\mathscr H_{\hbar,q}=H^*_{\bT_q}(\stackqv)
$$

\begin{definition}\label{def: vertex}
    The vertex function with descendant $\tau \in \mathscr H_{\hbar,q}$ is defined to be
    $$
\ver{\tau}=\sum_{d \in \eff(X)} \ev_{p_2,*}\left(\qm_{\ns p_2}^{d},\vrs \cap \ev_{p_1}^{*}(\tau)\right) z^{d}  \in H^*_{\bT_q}(X)_{\loc}[[{\boldsymbol{z}}]].
$$
\end{definition}

Although the $\mathbb C_\hbar^\times$-action \eqref{action_hbar_quiver} is not itself contracting, it commutes with the scalar action
$$
s\cdot(Y,Z,A,B)=(sY,sZ,sA,sB).
$$
The latter preserves $\mu^{-1}(0)$, commutes with $\bT_q\times G_{\dv}$, and extends to an $\mathbb A^1$-action contracting $\mu^{-1}(0)$ to the origin. It therefore gives a $\bT_q\times G_{\dv}$-equivariant contraction, and hence the space of descendants identifies with
$$
\mathscr H_{\hbar,q}
=H^*_{\bT_q}(\stackqv)
=H^*_{\bT_q\times G_{\dv}}\bigl(\mu^{-1}(0)\bigr)
\cong H^*_{\bT_q\times G_{\dv}}(\on{pt}).
$$

\begin{definition}\label{def: capped vertex}
    The capped vertex function with descendant $\tau \in \mathscr H_{\hbar,q}$ is defined to be
    $$
\cpver{\tau}=\sum_{d \in \eff(X)} \widehat{\ev}_{p_2,*}\left(\qm_{\rel p_2}^{d},\vrs \cap \ev_{p_1}^{*}(\tau)\right) z^{d}  \in H^*_{\bT_q}(X)[[{\boldsymbol{z}}]].
    $$
\end{definition}

\begin{definition}
    For $\tau \in H^*_{\bT}(\stackqv)$, the quantum tautological class associated to $\tau$ is $\widehat{\tau}=\cpver{\tau}|_{q=0}$. Equivalently, it is given by the formula of Definition~\ref{def: capped vertex} after setting the $\mathbb C_q^\times$-equivariant parameter to zero.
\end{definition}

We record the compatibility of capped vertex functions with restriction
of equivariance.  Let $\mathsf F'\subseteq\mathsf F$ be a subtorus and set
$$
 \bT':=\mathsf F'\times\mathbb C_\hbar^\times,
 \qquad
 \bT'_q:=\bT'\times\mathbb C_q^\times.
$$
Thus, the $\mathbb C_\hbar^\times$-factor
and the source-rotation torus $\mathbb C_q^\times$ are unchanged.  We
write
$$
 \operatorname{res}^{\bT}_{\bT'}
 \quad\text{and}\quad
 \operatorname{res}^{\bT_q}_{\bT'_q}
$$
for the corresponding restriction homomorphisms in equivariant
cohomology.  In the following statement, subscripts on capped vertex
functions indicate the torus equivariance being used.

\begin{proposition}\label{Vasyapoprosil}
The capped vertex commutes with restriction of equivariance.  More
precisely, for every $\tau\in H^*_{\bT}(\stackqv)$, one has
\begin{equation}
\label{eq:restriction-capped-vertex}
 \operatorname{res}^{\bT_q}_{\bT'_q}
 \left(\cpver{\tau}_{\bT}\right)
 =
 \cpver{\operatorname{res}^{\bT}_{\bT'}(\tau)}_{\bT'}
 \qquad\text{in }H^*_{\bT'_q}(X)[[{\boldsymbol{z}}]].
\end{equation}
After extending the capped vertex maps
$H^*_{\bT_q}(\mathrm{pt})[[{\boldsymbol{z}}]]$-linearly, the diagram
$$
\begin{tikzcd}
 H^*_{\bT_q}(\stackqv)[[{\boldsymbol{z}}]]
 \arrow[r,"\cpver{-}_{\bT}"]
 \arrow[d,"\operatorname{res}^{\bT_q}_{\bT'_q}"']
 &
 H^*_{\bT_q}(X)[[{\boldsymbol{z}}]]
 \arrow[d,"\operatorname{res}^{\bT_q}_{\bT'_q}"]
 \\
 H^*_{\bT'_q}(\stackqv)[[{\boldsymbol{z}}]]
 \arrow[r,"\cpver{-}_{\bT'}"']
 &
 H^*_{\bT'_q}(X)[[{\boldsymbol{z}}]]
\end{tikzcd}
$$
is commutative.

In particular, the associated quantum tautological classes satisfy
\begin{equation}
\label{eq:restriction-quantum-tautological}
 \operatorname{res}^{\bT}_{\bT'}(\widehat{\tau}_{\bT})
 =
 \widehat{\operatorname{res}^{\bT}_{\bT'}(\tau)}_{\bT'}
 \qquad\text{in }H^*_{\bT'}(X)[[{\boldsymbol{z}}]].
\end{equation}
\end{proposition}

\begin{proof}
Fix $d\in\eff(X)$ and put
$$
M_d:=\qm^d_{\rel p_2}.
$$
We write
$$
\ev_{p_1}^{\bT_q}\colon M_d\longrightarrow\stackqv,
\qquad
\widehat{\ev}_{p_2}^{\bT_q}\colon M_d\longrightarrow X
$$
for the $\bT_q$-equivariant evaluation morphisms, and
$$
\ev_{p_1}^{\bT'_q}\colon M_d\longrightarrow\stackqv,
\qquad
\widehat{\ev}_{p_2}^{\bT'_q}\colon M_d\longrightarrow X
$$
for the same underlying morphisms equipped with the restricted $\bT'_q$-equivariant structures. The underlying moduli space $M_d$, the underlying evaluation morphisms, and the perfect obstruction theory do not depend on the choice of equivariance torus. Restricting the $\bT_q$-equivariant structures to $\bT'_q$ gives precisely the corresponding $\bT'_q$-equivariant structures. In particular,
\begin{equation}
\label{eq:restriction-virtual-class}
\operatorname{res}^{\bT_q}_{\bT'_q}
\left([M_d]^{\mathrm{vir}}_{\bT_q}\right)
=
[M_d]^{\mathrm{vir}}_{\bT'_q}.
\end{equation}

For completeness, this may be seen using finite-dimensional approximations to the Borel constructions. Choose a representation $U$ on which $\bT_q$ acts freely. It is then also free for $\bT'_q$, and the horizontal arrows in the diagram
$$
\begin{tikzcd}
(M_d\times U)/\bT'_q
\arrow[r]
\arrow[d,"\widehat{\ev}_{p_2}^{\bT'_q}"']
&
(M_d\times U)/\bT_q
\arrow[d,"\widehat{\ev}_{p_2}^{\bT_q}"]
\\
(X\times U)/\bT'_q
\arrow[r]
&
(X\times U)/\bT_q
\end{tikzcd}
$$
are smooth, with fiber $\bT_q/\bT'_q\simeq\mathsf F/\mathsf F'$. The diagram is Cartesian. The obstruction theory and its intrinsic normal cone pull back through the upper horizontal arrow, which gives \eqref{eq:restriction-virtual-class}.

Restriction of equivariance also commutes with pullback and cap product. Therefore
$$
\operatorname{res}^{\bT_q}_{\bT'_q}
\left(
[M_d]^{\mathrm{vir}}_{\bT_q}
\cap
\left(\ev_{p_1}^{\bT_q}\right)^*(\tau)
\right)
=
[M_d]^{\mathrm{vir}}_{\bT'_q}
\cap
\left(\ev_{p_1}^{\bT'_q}\right)^*
\left(\operatorname{res}^{\bT}_{\bT'}(\tau)\right).
$$
Moreover, $\widehat{\ev}_{p_2}^{\bT_q}$ and $\widehat{\ev}_{p_2}^{\bT'_q}$ have the same underlying proper morphism. Proper base change in the preceding Cartesian diagram consequently gives
$$
\operatorname{res}^{\bT_q}_{\bT'_q}
\circ
\widehat{\ev}_{p_2,*}^{\bT_q}
=
\widehat{\ev}_{p_2,*}^{\bT'_q}
\circ
\operatorname{res}^{\bT_q}_{\bT'_q}.
$$
Combining these identities, we obtain
\begin{align*}
&
\operatorname{res}^{\bT_q}_{\bT'_q}
\left[
\widehat{\ev}_{p_2,*}^{\bT_q}
\left(
[M_d]^{\mathrm{vir}}_{\bT_q}
\cap
\left(\ev_{p_1}^{\bT_q}\right)^*(\tau)
\right)
\right]
\\
&\qquad=
\widehat{\ev}_{p_2,*}^{\bT'_q}
\left(
[M_d]^{\mathrm{vir}}_{\bT'_q}
\cap
\left(\ev_{p_1}^{\bT'_q}\right)^*
\left(\operatorname{res}^{\bT}_{\bT'}(\tau)\right)
\right).
\end{align*}
Multiplying by $z^d$ and summing over $d\in\eff(X)$ proves \eqref{eq:restriction-capped-vertex}.

Finally, restriction from $\bT_q$ to $\bT'_q$ commutes with the specialization $q=0$. Applying this specialization to \eqref{eq:restriction-capped-vertex} proves \eqref{eq:restriction-quantum-tautological}.
\end{proof}

We will later make use of the following straightforward proposition. 

\begin{proposition}\label{prop: vertsurj} The $H^*_{\bT_q}(\mathrm{pt})[[{\boldsymbol{z}}]]$-linear map
\begin{align*}
     H^*_{\bT_q}(\stackqv)[[{\boldsymbol{z}}]] &\longrightarrow H^*_{\bT_q}(X)[[{\boldsymbol{z}}]] \\ 
    \tau &\mapsto \cpver{\tau}
\end{align*}
is surjective.
\end{proposition}

\begin{proof} Kirwan surjectivity for quiver varieties \cite{kirv} implies that the pullback under the inclusion $X \hookrightarrow \stackqv$ is surjective. The proposition now follows by the graded Nakayama lemma applied to the ring $H^*_{\bT_q}(\mathrm{pt})[[{\boldsymbol{z}}]]$.
\end{proof}

The capped vertex is related to the uncapped (i.e., bare) vertex by means of the so-called capping operator.

\begin{definition}
    The capping operator is
    $$
\cp=\sum_{d \in \eff(X)} (\ev_{p_1} \times \widehat{\ev}_{p_2})_{*}\left(\qm_{\ns p_1, \rel p_2}^{d},\vrs \right) z^{d}  \in H^*_{\bT_q}(X\times X)_{\loc}[[{\boldsymbol{z}}]].
    $$
\end{definition}

We will view $\cp$ and other classes on $X \times X$ as operators on (localized) cohomology by the convolution action, defined by
$$
\alpha  \circ \beta =\operatorname{pr}_{2,*}(\operatorname{pr}_1^{*}(\beta) \cup \alpha), \quad \alpha \in H^*_{\bT_q}(X \times X)_{\loc}, \quad  \beta \in H^*_{\bT_q}(X),
$$
where $\operatorname{pr}_{i}:X \times X \to X$ is the projection to the $i$th factor.

\begin{proposition}[\cite{PSZ,OkLec}]\label{prop: cappingeq}
    The operator $\cp$ is invertible and satisfies the formula
    $$
  \cpver{\tau}=\cp \ver{\tau}.
    $$
\end{proposition}

\begin{proof}
The proof of the $K$-theoretic version of this statement, which relies on $\mathbb{C}^{\times}_{q}$-equivariant localization, appears in~\cite{PSZ}. We reproduce it here in the cohomological case for completeness of the exposition and to correct a minor typo.

Let $f \in \qm_{\rel p_2}^{\mathbb C_q^{\times}}$. The element $f$ defines a quasimap from $C$ to $X$, where the curve $C$ is obtained from $\mathbb P^1 = C_0$ by gluing a chain of additional $\mathbb P^1$-components to the point $p_2$.

We observe that
\begin{itemize}
    \item  when restricted to $C_0$, $f$ gives a $\mathbb{C}^{\times}_{q}$-fixed quasimap $C_0 \to X$ nonsingular at $p_2$, i.e., an element of $\qm_{\ns p_2}^{\mathbb C_q^{\times}}$ (cf.~\cite[Section 7.1.7]{OkLec});
    \item when restricted to the closure of $C \setminus C_{0}$, $f$ gives a stable quasimap from a chain of rational curves with a nonsingular point at $p_2$ and another nonsingular point $\hat{p}_2$ at a smooth point on the last component in the chain. 
\end{itemize}

Let us denote the stack parameterizing quasimaps in the second bullet by $\mathcal Q$ (cf.~\cite[Theorem 11]{PSZ}), so that
$$
\qm_{\rel p_2}^{\mathbb C_q^{\times}} \simeq \qm_{\ns p_2}^{\mathbb C_q^{\times}} \times_{X} \mathcal Q
$$
where the fiber product is taken for the morphism $\ev_{p_2}$. 

The virtual normal bundle to $\qm_{\rel p_2}^{\mathbb{C}^{\times}_{q}} \hookrightarrow \qm_{\rel p_2}$ has two factors. The first comes from deformations of the source curve and is $N_1:=\psi_2 \otimes \mathbb C_q$, where $\psi_2$ is the cotangent line at the point $p_2$, and $\mathbb C_q$ is the weight $1$ representation of $\mathbb C_q^{\times}$.

The second factor, denoted $N_{2}$, comes from deformations of the quasimap itself and is the part of the virtual tangent space~\eqref{Tvir} with nonzero $\mathbb C_q^{\times}$-weight. Equivalently, it is the normal bundle to $\qm_{\ns p_2}^{\mathbb C_q^{\times}} \hookrightarrow \qm_{\ns p_2}$.

By localization, we get
\begin{multline} \cpver{\tau}= \sum_{d_1} z^{d_1} (\ev_{p_2} \times \widehat{\ev}_{p_2})_{*}\left(\mathcal Q, \frac{[\mathcal Q]^{\text{vir}}}{e(N_1)}\right) \\ \circ \left(\sum_{d_2} z^{d_2} \ev_{p_2,*}  \left(\qm_{\ns p_2}^{\mathbb C_q^{\times}} , \frac{[\qm_{\ns p_2}^{\mathbb C_q^{\times}}]^{\text{vir}}\cap \ev_{p_1}^{*}(\tau)}{e(N_2)}\right)\right) \label{locform}
\end{multline}
where $\circ$ denotes the convolution action of the first factor on the second one over the $p_2$-index.

By definition, the second factor in \eqref{locform} is the vertex. We need to recognize the first factor in~\eqref{locform} as the capping operator.

Recall that the capping operator is defined using the space $\qm_{\ns p_1, \rel p_2}$. It is straightforward to see that $\left(\qm_{\ns p_1, \rel p_2}\right)^{\mathbb{C}^{\times}_{q}}=\mathcal{Q}$ as sets. Since a $\mathbb{C}^{\times}_{q}$-fixed quasimap in $\qm_{\ns p_1, \rel p_2}$ must be nonsingular at both $p_1$ and $p_2$, all bundles $\mathscr{V}_{i}$ constituting the quasimaps in this space are trivial when restricted to $C_0$, as in~\cite[Section 7.1.7]{OkLec}. As a result, $\mathbb C_q^{\times}$ acts nontrivially only on the part of the obstruction theory corresponding to the deformations of the curve. So the first factor in~\eqref{locform} is exactly the capping operator.
\end{proof}

\begin{remark}\label{rem:typo}
In~\cite{PSZ}, the $N_2$ factor is missing. As we explain in the last paragraph of the preceding proof, this is justified for $\qm_{\ns p_1, \rel p_2}$. However, when the nonsingularity condition at $p_1$ is omitted, the virtual normal bundle is larger.
\end{remark}

\subsection{Quantum multiplication}
Quasimap theory was used by Pushkar, Smirnov, and Zeitlin to define a quantum product in \cite{PSZ}. This operation, which we will refer to as PSZ quantum multiplication, deforms the usual multiplication of $H^*_{\bT}(X)$ in a way that resembles the ordinary definition of quantum multiplication via Gromov--Witten invariants.

Consider the moduli space $\qm_{\rel p_1,\rel p_2}$ relative to both points. The domain of such a quasimap is a curve $C$ with a distinguished component $C_{0} \cong \mathbb{P}^{1}$ and a morphism $\pi\colon C \to C_{0}$ given by contracting chains of copies of $\mathbb{P}^{1}$.

\begin{definition}
For $i \in I$, let 

\begin{multline}\label{PSZ}
M_i(z):=\sum_{d\in\eff(X)}z^d
(\widehat{\ev}_{p_1}\times\widehat{\ev}_{p_2})_*
\Bigl(
\qm^d_{\rel p_1,\rel p_2}, \\
\vrs\cap c_1\!\left(\det H^*\!\left(\qmtb_i\otimes\pi^*\mathcal O_{p_2}\right)\right)
\Bigr).
\end{multline}
\end{definition}

\begin{remark}\label{rem: cohomological gluing}
    The reader familiar with \cite{PSZ} will notice the absence of the gluing matrix in the formula~\eqref{PSZ}. The reason is that the cohomological gluing matrix is equal to the identity operator.
\end{remark}

\begin{definition}\label{def: quantum multiplication}
    Let $p \in C_{0}$ be a third point, distinct from $p_1$ and $p_2$. Let $\alpha \in H^*_{\bT}(X)$. The operator of quantum multiplication by $\alpha$ is
    $$
    \alpha \qcup :=\sum_{d \in \eff(X)} z^{d}(\widehat{\ev}_{p_1} \times \widehat{\ev}_{p_2})_{*}\left(\qm^{d}_{\rel p_1,\rel p,\rel p_2}, \vrs \cap \widehat{\ev}^{*}_{p}(\alpha)  \right).
    $$
\end{definition}

The same argument as in \cite{PSZ} shows that $\qcup$ is associative. The vector space $H^*_{\bT}(X)[[{\boldsymbol{z}}]]$, endowed with the product $\qcup$, is called the quasimap quantum cohomology of $X$.

We now discuss a few features of quasimap quantum cohomology which do not hold in the $K$-theoretic story studied in \cite{PSZ}. In \cite{PSZ}, it is shown that the $K$-theoretic version of the operator $M_{i}(z)$ specializes to a quantum multiplication operator when $q=1$. In cohomology, degree reasons force $M_{i}(z)$ to be equal to a quantum multiplication operator on the nose.

\begin{proposition}\label{prop: M and c1}
    As operators on $H^{*}_{\bT}(X)$, we have $M_{i}(z)=\widehat{c_1(\tb_{i})} \qcup$.
\end{proposition}
\begin{proof}
By definition, $M_i(z)$ is obtained by pushing forward from $\qm^d_{\rel p_1,\rel p_2}$.

Recall that there is a projective morphism $X\to\mathcal M_H$, where $\mathcal M_H$ is affine. This morphism extends to $\varpi\colon\stackqv\to\mathcal M_H$.  Since $G_{\dv}$ is reductive,
$$
\operatorname{Spec}\Gamma(\stackqv,\mathcal O_{\stackqv})
=\operatorname{Spec}\mathbb C[\mu^{-1}(0)]^{G_{\dv}}
=\mathcal M_H.
$$
Thus $\varpi$ is the affinization morphism, and its restriction to $X$ is the usual projective GIT morphism. For a family of connected proper quasimap curves $\pi\colon C\to S$, one has $\pi_*\mathcal O_C=\mathcal O_S$. Hence the composition $C\to\stackqv\xrightarrow{\varpi}\mathcal M_H$ factors through $S$. It follows that
$$
\widehat{\ev}_{p_1}\times\widehat{\ev}_{p_2}\colon
\qm^d_{\rel p_1,\rel p_2}\longrightarrow X\times X
$$
factors through the Steinberg variety
$$
Z:=X\times_{\mathcal M_H}X
$$
by a map which we denote by $\phi$.

If $\Gamma$ has no loops, every irreducible component of $Z$ has dimension at most $\dim X$ by~\cite[Theorem~7.2.4(ii)]{GinzburgLectures}, equivalently~\cite[Theorem~7.2]{NakQv}. For the Jordan-quiver Hilbert schemes considered here, the same statement follows from semismallness of the Hilbert--Chow symplectic resolution; see~\cite[Lemma~2.11]{K06}. Since $Z$ contains the diagonal, in the cases considered in this paper we have $\dim Z=\dim X$. Pushing forward by $\phi$ gives
$$
A^{(d)}:=\phi_*
\left(
\vrs\cap c_1\!\left(\det H^*\!\left(\qmtb_i\otimes\pi^*\mathcal O_{p_2}\right)\right)
\right)
\in H^{\bT_q,BM}_{2\dim X-2}(Z).
$$
Let $j\colon Z\hookrightarrow X\times X$ be the inclusion. Then
$$
M_i(z)=\sum_d z^d\,j_*A^{(d)}.
$$

The cohomological analogue of the proof of~\cite[Theorem~16]{PSZ} gives
$$
M_i(z)|_{q=0}=\widehat{c_1(\tb_i)}\qcup.
$$
Indeed, after forgetting the $\mathbb C_q^\times$-equivariance, the determinant insertion may be moved to a third nonrelative point, and the degeneration formula identifies the resulting correspondence with quantum multiplication. Since $\mathbb C_q^\times$ acts trivially on $Z$, it follows that
$$
M_i(z)-\widehat{c_1(\tb_i)}\qcup=qB
$$
for some
$$
B\in H^{\bT_q,BM}_{2\dim X}(Z)[[\boldsymbol{z}]].
$$
Since $\dim Z=\dim X$, top equivariant Borel--Moore homology is nonequivariant:
$$
H^{\bT_q,BM}_{2\dim X}(Z)\cong H^{BM}_{2\dim X}(Z).
$$
In particular, $B$ is independent of all equivariant parameters.

For $d\ne0$, the moduli space of quasimaps admits a reduced virtual class of complex virtual dimension one larger. In the convention of~\cite[Remark~5.5.6(ii)]{CK14}, the additive weight of the symplectic form is written as $-\lambda$. Since this weight is $-2\hbar$ here, $\lambda=2\hbar$, and
$$
[\qm^d]^{\mathrm{vir}}=2\hbar[\qm^d]^{\mathrm{vir,red}};
$$
therefore the positive-degree terms in both operators vanish at $\hbar=0$, while their degree-zero terms are both multiplication by $c_1(\tb_i)$. Specializing the equality
$M_i(z)-\widehat{c_1(\tb_i)}\qcup=qB$ at $\hbar=0$ gives $qB=0$. Since $B$ is nonequivariant, we conclude that $B=0$.
\end{proof}

In \cite{PSZ}, it is emphasized that the multiplicative identity in quantum $K$-theory is not the usual identity. In ordinary (i.e., stable-map) quantum cohomology, the unit does not deform. The following proposition, well known to experts, shows that this is also the case for quasimap quantum cohomology.

\begin{proposition}\label{identity_in_quant_is_identity}
    The multiplicative identity in quasimap quantum cohomology is the class $1 \in H^{*}_{\bT}(X)$.
\end{proposition}
\begin{proof}

By the cohomological analogue of~\cite[Theorem~15]{PSZ}, the multiplicative identity is the quantum tautological class
$$
\widehat 1
:=\cpver{1}|_{q=0}
=\left.
\sum_{d\in\eff(X)}z^d\,
\widehat{\ev}_{p_2,*}\left(\qm^d_{\rel p_2},\vrs\right)
\right|_{q=0}.
$$
Its degree-zero term is $1$. If $d\ne0$, then, as in the proof of Proposition~\ref{prop: M and c1},
$$
[\qm^d]^{\mathrm{vir}}=2\hbar[\qm^d]^{\mathrm{vir,red}},
$$
and the reduced class has complex virtual dimension one larger. Since $\widehat{\ev}_{p_2}$ is proper, the pushforward of the reduced class to $X$ vanishes for dimension reasons. Hence all positive-degree coefficients of $\widehat 1$ vanish, so $\widehat 1=1$ and therefore
$$
1\qcup-=1\cup-.
$$

For comparison, the cohomological tube operator satisfies an idempotence relation. More precisely, the degeneration formula in the form of~\cite[Section~6.5 and Theorem~7.1.4]{OkLec}, gives $\mathcal T^2=\mathcal T$ because the cohomological gluing operator $\mathcal T$ is the identity; see Remark~\ref{rem: cohomological gluing}. Since $\mathcal T=\operatorname{Id}+O({\boldsymbol z})$, it is invertible and hence $\mathcal  T=\operatorname{Id}$. Thus the relevant relation is idempotence, rather than $\mathcal T^2=1$. 
\end{proof}

\subsection{The quantum \texorpdfstring{$D$}{D}-module}\label{quantDmodrev}

Here, we give the geometric realization of the PSZ quantum $D$-module introduced in Section~\ref{ssec_def_of_PSZ}. Let $D=D^{\mathrm{PSZ}}$ be the algebra of logarithmic differential operators defined there. Since $H^{2}(\stackqv,\mathbb{C})$ has a basis given by $c_1(\tb_{i})$, the algebra $D$ is topologically generated over $H^*_{\bT_q}(\mathrm{pt})$ by $z_i$ and $\partial_i$, for $i\in I$, with relations
$$
[\partial_i,z_j]=q\delta_{ij}z_j,\qquad [z_i,z_j]=[\partial_i,\partial_j]=0.
$$

\begin{definition}\label{quantDmodrev:def}
The PSZ quantum $D$-module is the vector space $Q^{\mathrm{PSZ}}:=H^{*}_{\bT_q}(X)[[{\boldsymbol{z}}]]$ endowed with the structure of a module over $D$ by
    \begin{align*}
        z_i \cdot \alpha&=z_i \alpha, \\
        \partial_{i} \cdot \alpha&=\left(q z_i \frac{\partial}{\partial z_i}+M_{i}(z)\right)\alpha,
    \end{align*}
    for $\alpha \in Q^{\mathrm{PSZ}}$.
\end{definition}

Thus, consistently with the convention fixed in Section~\ref{sec:overview}, our quantum connection operators are
$$
\nabla_i^{\mathrm{PSZ}}=q z_i\frac{\partial}{\partial z_i}+M_i(z).
$$

Due to the complexity of the operators $M_{i}(z)$, it is difficult to study the quantum $D$-module directly. We will fit it into a commutative diagram that allows us to study it indirectly.

We define a normalized version of the bare vertex by $\nver{\tau}=\left(\prod_{i \in I} z_i^{c_1(\tb_{i})/q} \right) \ver{\tau}$. By Proposition \ref{prop: vertsurj} and the localization theorem, there are maps
$$
\begin{tikzcd}
  \mathscr H_{\hbar,q}[[{\boldsymbol{z}}]]\arrow[->>]{r}{\cpver{\cdot}} \arrow{rd}[swap]{\bigoplus\limits_{p\in X^{\bT}} \nver{\cdot}_{p}} & H^*_{\bT_q}(X)[[{\boldsymbol{z}}]] \\
     & \bigoplus\limits_{p \in X^{\bT}} z^{\eta_p}H^*_{\bT_q}(p)_{\loc}[[{\boldsymbol{z}}]]
\end{tikzcd}
$$

By Proposition \ref{prop: cappingeq}, the kernels of the capped and bare vertices are equal. Hence we get an induced commutative diagram
\begin{equation}\label{diag2}
\begin{tikzcd}
    \mathscr H_{\hbar,q}[[{\boldsymbol{z}}]] \arrow[->>]{r}{\cpver{\cdot}} \arrow{rd}[swap]{\bigoplus\limits_{p\in X^{\bT}} \nver{\cdot}_{p}} & H^*_{\bT_q}(X)[[{\boldsymbol{z}}]] \arrow[hookrightarrow]{d} \\
     & \bigoplus\limits_{p \in X^{\bT}} z^{\eta_p}H^*_{\bT_q}(p)_{\loc}[[{\boldsymbol{z}}]]
\end{tikzcd}
\end{equation}
\begin{warning}
Note that the right vertical arrow in (\ref{diag2}) is {\emph{not}} the restriction map to the fixed-point set.    
\end{warning}

We give the upper-left and lower-right terms of~\eqref{diag2} the $D$-module structures introduced in Sections~\ref{sssec: master D-module} and~\ref{sssec: normalized classical solutions}, respectively. Thus, $z_i$ acts by multiplication in both terms, while
\begin{align*}
z_i \cdot \tau &= z_i \tau, \\
\partial_{i} \cdot \tau &=\left(q z_i \frac{\partial}{\partial z_i} + c_1(\tb_{i})\right) \tau,
\end{align*}
for $\tau \in \mathscr H_{\hbar,q}[[{\boldsymbol{z}}]]$, and
\begin{align*}
z_i \cdot \alpha &= z_i \alpha, \\
\partial_{i} \cdot \alpha &= q z_i \frac{\partial}{\partial z_i} \alpha,
\end{align*}
for $\alpha \in \bigoplus_{p \in X^{\bT}} z^{\eta_p}H^*_{\bT_q}(p)_{\loc}[[{\boldsymbol{z}}]]$. In the latter formula, the derivative acts also on the formal factor $z^{\eta_p}$; in particular,
$$
qz_i\frac{\partial}{\partial z_i}z^{\eta_p}
=\eta_p(a_i)z^{\eta_p},
\qquad
a_i=c_1^{G_{\dv}}(V_i),
$$

\begin{theorem}\label{thm: Dmodhomom} Each of the three maps in \eqref{diag2} is a morphism of $D$-modules.
\end{theorem}

Theorem~\ref{thm: Dmodhomom} provides cohomological analogues of the well-known $K$-theoretic results of the Okounkov school.

\begin{proposition}\label{prop: derivative vertex}
     \textit{(Cf.~\cite[(7.2.5)]{OkLec})} 
     Let $\tau \in \mathscr H_{\hbar,q}[[{\boldsymbol{z}}]]$. For any $i \in I$, we have
\begin{equation}\label{eq:derivative-vertex}
V^{\Big(\left(q z_i\frac{\partial}{\partial z_i}+c_1(\tb_i)\right)\tau\Big)}=\left(q z_i\frac{\partial}{\partial z_i}+c_1(\tb_i)\right)\ver{\tau}.
\end{equation}
\end{proposition}

\begin{proof}

Suppose first that $\tau \in H^{*}_{\bT_{q}}(\stackqv)$, i.e., that $\tau$ does not depend on ${\boldsymbol{z}}$. By definition of $\ver{-}$, one has
\begin{equation}\label{quasidef}
\ver{c_1(\tb_{i}) \tau}=\sum_{d \in \eff(X)} z^{d} \ev_{p_2,*}\left(\qm_{\ns p_2}^{d},\vrs \cap \ev_{p_1}^{*}(c_1(\tb_i)\tau)\right).
\end{equation}

Let $\ev\colon \qm_{\ns p_2} \times \mathbb{P}^{1} \to \stackqv$ be the universal morphism, and let $\pi_1$ and $\pi_2$ be the two projections. One can show, using the projection formula, that
$$
\ev_{p_{j}}^{*} \tb_{i}=\pi_{1,*}\left(\ev^{*} \tb_{i} \otimes \pi_2^{*} \mathcal{O}_{p_j}\right).
$$

By the same reasoning as in \cite[Equation~(8.1.4)]{OkLec},
$$
\deg( \mathcal{F}) q= c_1\left( \det H^{*}( \mathcal{F} \otimes (\mathcal{O}_{p_1}-\mathcal{O}_{p_2}))\right)
$$
for any $\mathbb{C}^{\times}_{q}$-equivariant coherent sheaf $ \mathcal{F}$ on $\mathbb{P}^{1}$. 
So, on $\qm_{\ns p_2}^{d}$, we have
\begin{align*}
d_i q&=c_1\left( \det \pi_{1,*}(\ev^{*}\tb_{i} \otimes (\pi_2^{*}\mathcal{O}_{p_1}-\pi_2^{*}\mathcal{O}_{p_2}))\right) \\
&=c_1\left( \det ( \ev_{p_1}^{*} \tb_{i}-\ev_{p_2}^{*}\tb_{i})\right) \\
&=\ev_{p_1}^{*}( c_1(\tb_{i})) -\ev_{p_2}^{*}(c_1(\tb_{i})).
\end{align*}

So \eqref{quasidef} is equal to
$$
\sum_{d \in \eff(X)} z^{d} \ev_{p_2,*} \left(\qm_{\ns p_2}^{d},\left(\vrs \cap \ev_{p_1}^*(\tau)\right)  \cap \ev_{p_2}^*(c_1(\tb_i) + d_i q)\right).
$$

By applying the projection formula to the morphism $\ev_{p_2}$, this is equal to
\begin{align*}
&\sum_{d \in \eff(X)}  (c_1(\tb_i) + d_{i} q)z^{d}\ev_{p_2,*} \left(\qm_{\ns p_2}^{d},\vrs \cap \ev_{p_1}^*(\tau) \right) \\
&= \left(c_1(\tb_i) + q z_i \frac{\partial}{\partial z_{i}}\right) \sum_{d \in \eff(X)} z^{d}\ev_{p_2,*} \left(\qm_{\ns p_2}^{d},\vrs \cap \ev_{p_1}^*(\tau) \right) \\
&=\left(c_1(\tb_i) + q z_i \frac{\partial}{\partial z_{i}}\right) \ver{\tau}.
\end{align*}

For terms of the form $\tau z^{a}$, with $\tau \in H^*_{\bT_q}(\stackqv)$ and $a \in \eff(X)$, the proposition now follows immediately, and hence holds in general.

\end{proof}

The following corollary is an elementary check using the definition of the normalized vertex $\nver{-}$.

\begin{corollary}
$$
\widetilde{V}^{\Big(\Big(q z_i \frac{\partial}{\partial z_i} + c_1(\tb_{i})\Big)  \tau\Big)} = q z_i \frac{\partial }{\partial z_i} \nver{\tau}.
$$
\end{corollary}

The situation with the capped vertex $\cpver{-}$ is more complicated. For this, we will make use of the quantum differential equation satisfied by the capping operator.

\begin{proposition}\label{prop: quantum diff eq}
The capping operator satisfies the equation

$$
\left(q z_i \frac{\partial}{\partial z_i}+M_i(z)\right)\cp(z)=\cp(z)c_1(\tb_i).
$$
Thus, the capping operator intertwines the two connections written with the same plus-sign convention.
\end{proposition}

\begin{proof}

We will imitate the argument from~\cite[Section 2.8]{PSZ}.

Recall that domains $C$ of quasimaps with relative marked points come equipped with a map $\pi: C \to C_{0} \cong \mathbb{P}^{1}$.

Let $C$, $\pi$, and $\qmtb_{i}$ for $i \in I$ be the data provided by a quasimap in $\qm_{\ns p_1,\rel p_2}$. As in the proof of Proposition \ref{prop: derivative vertex}, we have
$$
c_1\!\left(\det H^{*}\!\left(\qmtb_{i} \otimes \pi^*(\mathcal O_{p_1} - \mathcal O_{p_2})\right)\right) = \deg(\qmtb_{i})q.
$$

Then

\begin{multline*}
q z_i\frac{\partial}{\partial z_i}\cp(z) = \sum\limits_{d \in \eff(X)} z^d(\ev_{p_1} \times \widehat{\ev}_{p_2})_{*} \\ 
\left(\qm^d_{\ns p_1, \rel p_2},   \vrs \cap c_1\!\left(\det H^{*}\!\left(\qmtb_{i} \otimes \pi^*(\mathcal O_{p_1} - \mathcal O_{p_2})\right)\right)\right).
\end{multline*}

We can write this as $A-B$, where

\begin{multline*}
A = \sum_{d \in \eff(X)} z^d(\ev_{p_1} \times \widehat{\ev}_{p_2})_{*} \\ 
\left(\qm^d_{\ns p_1, \rel p_2},  \vrs \cap c_1(\det H^{*}(\qmtb_{i} \otimes \pi^*(\mathcal O_{p_1}))) \right) \\
= \cp(z) c_{1}(\tb_{i})
\end{multline*}
and
\begin{multline*}
B=\sum_{d \in \eff(X)} z^d(\ev_{p_1} \times \widehat{\ev}_{p_2})_{*} \\
\left(\qm^d_{\ns p_1, \rel p_2}, \vrs \cap c_1(\det H^{*}(\qmtb_{i} \otimes \pi^*(\mathcal O_{p_2}))) \right).
\end{multline*}
The cohomological analogue of~\cite[Section~6.5]{OkLec}, applied as in~\cite[Section~2.5, equation~(24), and Section~2.8]{PSZ} gives
$$
B=M_{i}(z) \cp(z)
$$
with no additional gluing matrix, by Remark~\ref{rem: cohomological gluing}. This concludes the proof.

\end{proof}

\begin{proposition}\label{prop: capdiff}

Let $\tau \in \mathscr H_{\hbar,q}[[{\boldsymbol{z}}]]$. For any $i \in I$, we have
\begin{equation}\label{eq:capped-derivative-vertex}
\widehat{V}^{\Big(\left(q z_i\frac{\partial}{\partial z_i}+c_1(\tb_i)\right)\tau\Big)}=\left(q z_i\frac{\partial}{\partial z_i}+M_i(z)\right)\cpver{\tau}.
\end{equation}
\end{proposition}

\begin{proof}
This follows immediately from Propositions~\ref{prop: cappingeq}, \ref{prop: derivative vertex}, and~\ref{prop: quantum diff eq}.
\end{proof}

\begin{proof}[Proof of Theorem \ref{thm: Dmodhomom}]
This follows by combining Proposition~\ref{prop: derivative vertex}, Proposition~\ref{prop: capdiff}, the commutative diagram~\eqref{diag2}, and the definition of the three $D$-module structures.
\end{proof}

\subsection{Comparison with the stable map quantum \texorpdfstring{$D$}{D}-module}\label{sec: D module comparison}

There are two places in which our quasimap quantum $D$-module may differ from the ordinary quantum $D$-module, which is defined using moduli spaces of stable maps.

The first occurs in Definition~\ref{def: quantum multiplication}. It is proved in~\cite{CK14} that genus-zero quasimap theory and Gromov--Witten theory are related by wall-crossing, so the information encoded by the quasimap and usual quantum $D$-modules is equivalent.

The second discrepancy concerns the operators $M_i(z)$. By Proposition~\ref{prop: M and c1}, $M_i(z)$ is the operator of quantum multiplication by the deformed class $\widehat{c_1(\tb_i)}$, rather than by $c_1(\tb_i)$ itself.

This discrepancy is related to ``large framing vanishing.'' By~\cite[Theorem~7.5.23]{OkLec}, for fixed $\dv$ and $\tau$, one has $\cpver{\tau}=\tau$ on $\qv(\dv,\dw)$ for all sufficiently large $\dw$. See also~\cite{AD} for situations in which this is proved for cyclic quiver varieties. 
For an application of large framing vanishing in the Jordan-quiver setting, see also~\cite[Section~1.2]{AS25}, where it is used to obtain explicit capped descendant vertices for $\operatorname{Hilb}^n(\mathbb C^2)$ from the rank-two instanton moduli space.

Suppose that $\cpver{c_{1}(\tb_{i})}=c_{1}(\tb_{i})$ for all $i$. Then $M_{i}(z)=c_{1}(\tb_{i}) \qcup$. Conjecturally, $\cpver{c_{1}(\tb_{i})}=c_{1}(\tb_{i})$ holds for all of the cases treated in \cite{KMBP} for which quasimap theory is defined, namely type $A$ Springer resolutions and smooth hypertoric varieties. We believe that, in general, the stable-map quantum $D$-module is not the correct object to consider for the quantum Hikita conjecture.

\subsection{Calabi--Yau specialization of vertex functions}\label{CYvertex}
\subsubsection{Proof of Proposition~\ref{nonlocvertex}}

Recall that $\mathsf T\subset\mathsf F\times\mathbb C_\hbar^\times$ is a subtorus as in the statement and that $\mathsf T_q=\mathsf T\times\mathbb C_q^\times$. In particular, the symplectic form has character $\hbar^{-2}$.

\begin{proof}[Proof of Proposition~\ref{nonlocvertex}]
By linearity, take $\tau\in H^*_{\bT_q}(\stackqv)$ and work coefficientwise in ${\boldsymbol{z}}$.  Put
$$
 M_d:=\qm^d_{\ns p_2}
$$
and $\mathfrak C_{\bT}:=H^*_{\bT}(\operatorname{pt})$.  The degree-$d$ coefficient of the vertex is
$$
 V_{\bT,d}^{(\tau)}
 :=(\ev_{p_2})_*^{\loc}
 \bigl([M_d]^{\mathrm{vir}}\cap\ev_{p_1}^*\tau\bigr).
$$
Full $\bT_q$-localization gives
\begin{equation}
\label{eq:T-full-localization}
 V_{\bT,d}^{(\tau)}
 =\sum_{P\in\pi_0(M_d^{\bT_q})}
 (\ev_P)_*
 \left(
  \frac{[P]^{\mathrm{vir}}\cap\ev_{p_1}^*\tau}
       {e_{\bT_q}(N_P^{\mathrm{vir}})}
 \right),
 \qquad
 \ev_P:=\ev_{p_2}|_P .
\end{equation}
Every $\ev_P$ is proper, since $P\subset M_d^{\mathbb C_q^\times}$.

Form the derived evaluation fiber
$$
 M_{d,p}:=M_d\mathop{\times}\limits_X^{\mathbf R}\{p\}.
$$
Refined virtual base change, applied to the proper summands in the
$\mathbb C_q^\times$-localization defining the vertex, gives
\begin{equation}
\label{eq:point-fiber-localization}
 \iota_p^*V_{\bT,d}^{(\tau)}
 =\int_{[M_{d,p}]^{\mathrm{vir}}}^{\loc}\ev_{p_1}^*\tau.
\end{equation}
There is no factor $e_{\bT_q}(T_pX)^{-1}$ in
\eqref{eq:point-fiber-localization}, since $\iota_p^*$ is ordinary pullback.
Moreover, $(M_{d,p})^{\mathbb C_q^\times}$ is proper.  Hence
$$
 \iota_p^*V_{\bT,d}^{(\tau)}
 =\sum_{K\in\pi_0(M_{d,p}^{\bT_q})}
 \int_{[K]^{\mathrm{vir}}}
 \frac{\gamma_K}{e_{\bT_q}(N_K^{\mathrm{vir}})},
 \qquad
 \gamma_K:=\ev_{p_1}^*\tau|_K,
$$
and every $K$ is proper.

Let $\pi\colon\mathcal C\to M_d$ be the universal curve and set
$$
 \mathcal A_d:=R\pi_*\bigl(\qmpol(-p_2)\bigr).
$$
Since the tangent weight at $p_2$ is $q^{-1}$,
\begin{equation}
\label{eq:point-relative-POT}
 E_d:=T^{\mathrm{vir}}_{M_{d,p}}
 =\mathcal A_d-\rho\mathcal A_d^\vee,
 \qquad
 \rho:=q\hbar^{-2},
 \qquad
 E_d=-\rho E_d^\vee.
\end{equation}
Indeed,
$$
 \omega_{C_0}(p_2)\cong q^{-1}\mathcal O_{C_0}(-p_2),
 \qquad
 R\pi_*\bigl(\hbar^{-2}(\qmpol)^\vee(-p_2)\bigr)=-q\hbar^{-2}\mathcal A_d^\vee,
$$
and \eqref{eq:point-relative-POT} follows from the exact sequence at
$p_2$ and~\eqref{Tvir}.  Thus $E_d$ has virtual rank zero.

Put
$$
 \bT_{\mathrm{CY}}:=\ker(\rho\colon\bT_q\to\mathbb C^\times)
 =\{(t,\hbar(t)^2):t\in\bT\}\simeq\bT,
 \qquad
 \delta:=c_1(\rho)=q-2\hbar.
$$
Since $\rho$ has $q$-weight one, every character of $\bT_q$ is uniquely $\lambda+c\rho$, where
$\lambda\in X^*(\bT)$ and $c\in\mathbb Z$.  Write
$$
 \mathcal A_d|_K
 =\sum_{\lambda,c}A_{\lambda,c}\otimes
   \mathbb C_{\lambda+c\rho}.
$$
All ranks below are virtual.  From~\eqref{eq:point-relative-POT},
$$
 T_K^{\mathrm{vir}}=A_{0,0}-A_{0,1}^\vee,
 \qquad
 v:=\operatorname{vdim}K
 =\operatorname{rk}A_{0,0}-\operatorname{rk}A_{0,1},
 \qquad
 [K]^{\mathrm{vir}}\in A_v(K).
$$
If $v<0$, this class is zero; assume $v\geq0$.  Formula
\eqref{eq:point-relative-POT} gives
\begin{equation}
\label{eq:point-inverse-Euler}
 \frac1{e_{\bT_q}(N_K^{\mathrm{vir}})}
 =
 \frac{\displaystyle
  \prod_{(\lambda,c)\ne(0,1)}
  e\bigl(A_{\lambda,c}^\vee\otimes
          \mathbb C_{-\lambda+(1-c)\rho}\bigr)}
 {\displaystyle
  \prod_{(\lambda,c)\ne(0,0)}
  e\bigl(A_{\lambda,c}\otimes
          \mathbb C_{\lambda+c\rho}\bigr)}.
\end{equation}

Write $A^\bullet(K)$ for the ordinary Chow ring.  The resonant part of
\eqref{eq:point-inverse-Euler} is
\begin{equation}
\label{eq:point-resonant}
 R_K(\delta)
 =\frac{\displaystyle\prod_{c\ne1}
  e\bigl(A_{0,c}^\vee\otimes\mathbb C_{(1-c)\rho}\bigr)}
 {\displaystyle\prod_{c\ne0}
  e\bigl(A_{0,c}\otimes\mathbb C_{c\rho}\bigr)}
 =\delta^v\sum_{k\geq0}\delta^{-k}r_k,
 \qquad r_k\in A^k(K).
\end{equation}
Indeed, the leading exponent is
$$
 \sum_{c\ne1}\operatorname{rk}A_{0,c}
 -\sum_{c\ne0}\operatorname{rk}A_{0,c}=v,
$$
and the coefficient of $\delta^{v-k}$ has ordinary codimension $k$.
The last equality in~\eqref{eq:point-resonant} is read in the completion
by ordinary codimension; only $k\leq v$ can contribute to an integral
over $[K]^{\mathrm{vir}}$.

Let $P_K(\delta)$ be the nonresonant part of
\eqref{eq:point-inverse-Euler}.  It is Taylor-regular in $\delta$ after
localizing the nonzero $\bT$-characters.  Since $\gamma_K$ is
nonlocalized, write
$$
 P_K(\delta)\gamma_K
 =\sum_{j,\ell\geq0}\delta^j b_{j,\ell},
 \qquad
 b_{j,\ell}\in A^\ell(K)\otimes
 \operatorname{Frac}(\mathfrak C_{\bT}).
$$
Since $\bT_q$ acts trivially on $K$, a term in its contribution is
$$
 \delta^{v-k+j}
 \int_{[K]^{\mathrm{vir}}}r_kb_{j,\ell}.
$$
It can be nonzero only if $k+\ell=v$; its $\delta$-exponent is then
$\ell+j\geq0$.  Thus every summand is regular at $q=2\hbar$.

At $\delta=0$, the nonresonant factors cancel:
$$
 P_K(0)
 =\prod_{\lambda\ne0,\,c}
 \frac{e(A_{\lambda,c}^\vee\otimes\mathbb C_{-\lambda})}
      {e(A_{\lambda,c}\otimes\mathbb C_\lambda)}
 =(-1)^{\varepsilon_K},
 \qquad
 \varepsilon_K:=\sum_{\lambda\ne0,\,c}\operatorname{rk}A_{\lambda,c}.
$$
Here
$e(A^\vee\otimes\mathbb C_{-\lambda})
=(-1)^{\operatorname{rk}A}e(A\otimes\mathbb C_\lambda)$, also for a
virtual perfect complex $A$.
A term survives at $\delta=0$ only if $v-k+j=0$.  Together with
$k+\ell=v$, this gives $j=\ell=0$ and $k=v$.  Hence the specialized
contribution of $K$ is
\begin{equation}
\label{eq:point-specialized-contribution}
 (-1)^{\varepsilon_K}
 \int_{[K]^{\mathrm{vir}}}r_v\gamma_{K,0}
 \in\mathfrak C_{\bT},
 \qquad
 \gamma_{K,0}:=
 \left(\gamma_K|_{\delta=0}\right)_{\mathrm{ord}=0}.
\end{equation}
Summing over $K$ and $d$ proves regularity and nonlocalization. Since $\bT_q$ acts trivially on the connected component $Z$, restriction of a class on $Z$ to a point factors through $H^0(Z)$ and is therefore independent of $p\in Z$.
\end{proof}

Assume now that $\bT=\mathsf F\times\mathbb C_\hbar^\times$, where $\mathsf F$ is the full flavor torus, and that $X^{\mathsf F}$
is finite. In this case, Proposition~\ref{nonlocvertex} has the following more geometric proof.

\begin{proof}[Proof of Proposition~\ref{nonlocvertex}]
As above, work coefficientwise in ${\boldsymbol{z}}$.  Put
$$
 M_d:=\qm^d_{\ns p_2},
 \qquad
 \rho:=q\hbar^{-2},
 \qquad
 \bT_{\mathrm{CY}}:=\ker(\rho\colon\bT_q\to\mathbb C^\times).
$$

\textit{Step 1.}  The fixed locus $M_d^{\bT_{\mathrm{CY}}}$ is proper.
Indeed, $\mathsf F\subset\bT_{\mathrm{CY}}$ acts trivially on the
source, so a fixed quasimap is constant on the nonsingular locus, with
value in $X^{\mathsf F}$.  Since $\mathbb C_\hbar^\times$ commutes with
$\mathsf F$, it fixes this finite set pointwise.  The projection
$\bT_{\mathrm{CY}}\to\mathbb C_q^\times$ is surjective (on the
$\mathbb C_\hbar^\times$-factor it is a degree-two isogeny), so the only
possible singularity is $p_1$.

Fix $p\in X^{\mathsf F}$ and choose a stable representative.  Its gauge
stabilizer is trivial, so the $\mathbb C_\hbar^\times$-action has a
unique lift to its tautological spaces.  After pulling the
$\mathbb C_q^\times$-linearization back along this degree-two isogeny and twisting it by this lift as in
\cite[Sections~4.3.1--4.3.15]{OkLec}, the fixed-quasimap
description of \cite[Sections~7.2.6--7.2.16]{OkLec} identifies the fiber of
$$
 \ev_{p_2}\colon M_d^{\bT_{\mathrm{CY}}}
 \longrightarrow X^{\mathsf F}
$$
over $p$ with closed incidence loci in products of $(\mathsf F\times\mu_2)$-graded
framed flag varieties.  There are only finitely many flag types in fixed
degree, and these loci are proper by \cite[Corollary~7.2.15]{OkLec}.
Thus every fiber is proper, and so is
$M_d^{\bT_{\mathrm{CY}}}$.

\textit{Step 2.}  Fix $p\in X^\bT=X^{\mathsf F}$ and put
$$
 M_{d,p}:=M_d\mathop{\times}\limits_X^{\mathbf R}\{p\}.
$$
By Step~1, $M_{d,p}^{\bT_{\mathrm{CY}}}$ is proper.  Localization groups the $\bT_q$-fixed summands in
\eqref{eq:point-fiber-localization} by the components
$K\subset M_{d,p}^{\bT_{\mathrm{CY}}}$.  Put
$B_K:=(\mathcal A_d|_K)^{\mathrm{mov}}$.  By
\eqref{eq:point-relative-POT},
$$
 N_K^{\mathrm{vir}}=B_K-\rho B_K^\vee,
 \qquad
 \iota_p^*V_{\bT,d}^{(\tau)}
 =\sum_K\int_{[K]^{\mathrm{vir}}}
 \ev_{p_1}^*\tau\,
 \frac{e_{\bT_q}(\rho B_K^\vee)}{e_{\bT_q}(B_K)}.
$$
All weights of $B_K$ are nonzero on $\bT_{\mathrm{CY}}$, and the last
quotient is regular at $\rho=1$, where it equals
$(-1)^{\operatorname{rk}B_K}$.  Since $K$ is proper,
$$
 \left.\iota_p^*V_{\bT,d}^{(\tau)}\right|_{q=2\hbar}
 =\sum_{K\in\pi_0(M_{d,p}^{\bT_{\mathrm{CY}}})}
 (-1)^{\operatorname{rk}B_K}
 \int_{[K]^{\mathrm{vir}}}
 \left.\ev_{p_1}^*\tau\right|_{q=2\hbar}
 \in H^*_{\bT_{\mathrm{CY}}}(\operatorname{pt})
 =H^*_{\bT}(\operatorname{pt}).
$$
Summing over $d$ proves the proposition in this case.
\end{proof}

\subsubsection{Comparing specialized vertex functions for $X$ and $X^{\mathsf{F}}$}

The following proposition will be used to reduce the computation of $\ver{\tau}_{p}|_{q=2\hbar}$ ($p \in X^{\mathsf{F}}$ is isolated) to the case when $X$ is a single point.

Let $\mathsf F$ be an arbitrary flavor torus. Then $X^{\mathsf F}$ is the disjoint union of Nakajima quiver varieties for possibly different quivers (see \cite[Proposition 2.3.1]{MO19}). Let $Z$ be a connected component of $X^{\mathsf F}$.

Let $\tau$ be a descendant for $X$. It induces a descendant $\tau_Z$ for $Z$. On the one hand, we can consider $\ver{\tau}_{X,Z}$, the vertex function of $X$ restricted to $Z$. On the other hand, we have $\ver{\tau_Z}_{Z}$, the vertex function of $Z$. It is shown in \cite[Section 7.3]{OkLec} that the latter is a limit of the former. We need here a cohomological version of this statement after the further specialization $q=2\hbar$.

Writing $\ver{\tau}=\sum_d V_d^{(\tau)}z^d$, we define a signed version of the vertex by
$$
\signver{\tau}=\sum_d(-1)^{(T^{1/2}X,d)}V_d^{(\tau)}z^d,
$$
where $T^{1/2}X$ is a polarization of $X$. Since any two polarizations differ by a class of the form $\alpha-\hbar^{-2} \alpha^{\vee}$, the sign does not depend on the choice.
Explicitly, the shift is
$$
\signver{\tau}=\ver{\tau}|_{z_i=(-1)^{a_i} z_i},
$$
where
$$
a_{i}=\dw_{i}+\sum_{\substack{e \\ h(e)=i}} \dv_{t(e)}-\sum_{\substack{e \\ t(e)=i}} \dv_{h(e)}.
$$

\begin{proposition}\label{prop: limver}
At the Calabi--Yau specialization, the two signed vertex functions are equal:
$$
\signver{\tau}_{X,Z}\big|_{q=2\hbar}= \signver{\tau_{Z}}_{Z}\big|_{q=2\hbar}.
$$
\end{proposition}
\begin{proof}
   We compute the left-hand side by localization with respect to $\mathsf F$. Define the incidence condition at $p_2$ by the derived fiber product
   $$
   \qm^d_{p_2\mapsto Z}(X)
   :=\qm^d_{\ns p_2}(X)\mathop{\times}\limits_X^{\mathbf R}Z,
   $$
   taken with respect to $\ev_{p_2}$ and the inclusion $Z\hookrightarrow X$. Let $\qm^{d}_{\ns p_2}(Z)$ be the moduli space of degree $d$ quasimaps to $Z$ nonsingular at $p_2$. Then
   $$
   \bigl(\qm^{d}_{p_2 \mapsto Z}(X)\bigr)^{\mathsf F}
   =\qm^{d}_{\ns p_2}(Z).
   $$

   So 
   \begin{align*}
\signver{\tau}_{X,Z}&=\sum_{d \in \eff(X)}  (-1)^{(T^{1/2}X,d)}z^{d} \ev_{p_2,*}\left(\qm(X)^{d}_{p_2\mapsto Z},\vrs \cap \ev_{p_1}^{*}(\tau)\right)  \\
&=\sum_{d \in \eff(Z)} (-1)^{(T^{1/2}X,d)} z^{d} \ev_{p_2,*}\left(\qm(Z)^{d}_{\ns p_2},\frac{ \vrs \cap \ev_{p_1}^{*}(\tau_{Z}) }{e(N_{\mathrm{vir}})}\right),
\end{align*}
where $N_{\mathrm{vir}}$ is the virtual normal bundle to the $\mathsf F$-fixed locus.

Set
$$
N^{1/2}_{Z/X}:=\bigl(T^{1/2}X|_Z\bigr)^{\mathsf F\text{-}\mathrm{mov}},
\qquad
\mathcal B_d:=R\pi_*\bigl(f^*N^{1/2}_{Z/X}(-p_2)\bigr),
$$
where $f$ denotes the universal quasimap. The same Serre-duality calculation as in~\eqref{eq:point-relative-POT} gives
$$
N_{\mathrm{vir}}=\mathcal B_d-\rho\mathcal B_d^\vee,
\qquad \rho=q\hbar^{-2}.
$$
At $q=2\hbar$, equivalently $\rho=1$, it follows that
$$
\left.\frac{1}{e(N_{\mathrm{vir}})}\right|_{q=2\hbar}
=(-1)^{\operatorname{rk}\mathcal B_d}.
$$
Riemann--Roch gives
$$
\operatorname{rk}\mathcal B_d=(N^{1/2}_{Z/X},d),
$$
while the restricted polarization decomposes as
$$
T^{1/2}X|_Z=T^{1/2}Z+N^{1/2}_{Z/X}.
$$
Consequently,
$$
(T^{1/2}X,d)+\operatorname{rk}\mathcal B_d
\equiv (T^{1/2}Z,d)\pmod 2.
$$
The localization sign therefore converts the signed $X$-vertex into the signed $Z$-vertex, and we obtain
\begin{align*}
&\signver{\tau}_{X,Z}\big|_{q=2\hbar} \\
&=\sum_{d \in \eff(Z)} (-1)^{(T^{1/2} Z,d)} z^{d} \ev_{p_2,*}\left(\qm(Z)^{d}_{\ns p_2}, \vrs \cap \ev_{p_1}^{*}(\tau_{Z})\right)\big|_{q=2\hbar} \\
&=\signver{\tau_{Z}}_{Z}|_{q=2\hbar}.
\end{align*}
\end{proof}

\begin{example}
\label{fixpoints}
One case in which this gives a rather drastic reduction is when $\Gamma$ is of type ADE and $\dw$ is a sum of minuscule coweights. Indeed, in this case the fixed-point components are products of point-like quiver varieties.

More precisely, recall that $I=\{1,\ldots,r\}$. For $i \in I$, let $\delta_i$ be the vector with $1$ in the $i$th component and $0$ elsewhere. Let $\alpha_i$ and $\fundwt_i$ denote the simple roots and fundamental weights.

Fixed points $X^{\mathsf F_{\mathsf w}}$ are in bijection with all possible decompositions
\begin{equation*}
\mathsf v=\sum_{i\in I}\sum_{s=1}^{\mathsf w_i}\mathsf v_i^{(s)},
\end{equation*}
such that the quiver varieties $\qv({\mathsf{v}}_i^{(s)},\delta_i)$ are nonempty. Each of these quiver varieties is a point.

The variety $\qv(\dv,\delta_i)$ is nonempty if and only if $\fundwt_i-\sum_j\dv_j\alpha_j\in W\fundwt_i$.

Let us denote a collection $(\mathsf{v}_i^{(s)})$ as above by $\ul{\mathsf{v}}$. Explicitly, the fixed point $p_{\ul{\mathsf{v}}}$ corresponding to $\ul{\mathsf{v}}$ is equal to the image of
\begin{equation*}
\prod_{\substack{i\in I\\1\leqslant s\leqslant\mathsf w_i}}\qv(\mathsf{v}_i^{(s)},\delta_i) \hookrightarrow X,
\end{equation*}
under the map taking the direct sum of all representatives.
\end{example}

 \begin{example}\label{fixpointsGieseker}
Let $X=\qv(n,r)$ be the Gieseker variety associated with the Jordan
quiver.  Assume that $\mathsf F$ contains the loop-scaling torus and the
full framing torus.  Write
$$
W=\bigoplus_{s=1}^{r}W_s,
\qquad \dim W_s=1.
$$
If $p\in X^{\mathsf F}$, the $\mathsf F$-action lifts to the tautological
space of a stable representative, since its gauge stabilizer is trivial.
This gives a decomposition
$$
V=\bigoplus_{s=1}^{r}\bigoplus_{k\in\mathbb Z}V_k^{(s)}
$$
such that
$$
Y(V_k^{(s)})\subset V_{k-1}^{(s)},\qquad
Z(V_k^{(s)})\subset V_{k+1}^{(s)},
$$
while the framing maps have degree zero.  The moment map and stability
conditions therefore identify the corresponding fixed component with
$$
\prod_{s=1}^{r}
\qv_{A_\infty}\bigl(\mathsf v^{(s)},\delta_0\bigr),
\qquad
\mathsf v_k^{(s)}=\dim V_k^{(s)}.
$$
Since the vectors $\mathsf v^{(s)}$ have finite support, these are
finite type $A$ quiver varieties with one minuscule framing.

More explicitly, the fixed points are indexed by $r$-tuples of
partitions
$$
\ul{\mathsf Y}
=(\mathsf Y^{(1)},\ldots,\mathsf Y^{(r)}),
\qquad
\sum_s|\mathsf Y^{(s)}|=n,
$$
and
$$
\mathsf v_k(\mathsf Y^{(s)})
=
\#\{\square\in\mathsf Y^{(s)}\mid c(\square)=k\}.
$$
Here $c(i,j)=i-j$ is the content of a box.
Thus
$$
X^{\mathsf F}
=
\bigsqcup_{\ul{\mathsf Y}}
\prod_{s=1}^{r}
\qv_{A_\infty}
\bigl(\mathsf v(\mathsf Y^{(s)}),\delta_0\bigr).
$$
Every factor is nonempty and has minuscule framing, so
Example~\ref{fixpoints} implies that it is a point.  Hence the displayed
components are precisely the points $p_{\ul{\mathsf Y}}$.  For $r=1$
this recovers $X=\operatorname{Hilb}^{n}(\mathbb A^2)$.
\end{example}

\subsection{Localization in the case of isolated fixed points}

Assume that $(\qm^d_{\ns p_2})^{\bT_q}$ consists of isolated points for
every $d\in\eff(X)$. For a fixed quasimap $f$, set
$p_f:=\ev_{p_2}(f)\in X^{\bT}$ and let
$\iota_{p_f}\colon\{p_f\}\hookrightarrow X$ be the inclusion. The
localization formula immediately gives the following.

For $p\in X^{\bT}$ define the derived
evaluation fiber
$$
\qm^d_{p_2\mapsto p}
:=\qm^d_{\ns p_2}\mathop{\times}\limits_X^{\mathbf R}\{p\}.
$$

\begin{proposition}\label{prop: isolated qm}
Let $p\in X^{\bT}$ and assume that
$(\qm^d_{p_2\mapsto p})^{\bT_q}$ consists of isolated points for every
$d\in\eff(X)$. Then
$$
\iota_p^*\signver{\tau}_{X}
=e_{\bT}(T_pX)
 \sum_{d\in\eff(X)}\ \sum_{\substack{f\in(\qm^d_{\ns p_2})^{\bT_q}\\ p_f=p}}
 \frac{\tau(f)}{e_{\bT_q}(\qmtan|_{f})}
 (-1)^{(T^{1/2}X,d)}\boldsymbol{z}^d.
$$
\end{proposition}

\begin{proof}
By refined virtual base change, as in
\eqref{eq:point-fiber-localization}, $\iota_p^*\signver{\tau}_{X}$ is
computed by localization on the derived evaluation fiber
$\qm^d_{p_2\mapsto p}$. At a fixed quasimap $f$ in this fiber, its virtual
tangent space is $\qmtan|_f-T_pX$. The formula follows from virtual
localization.
\end{proof}

Concretely, $\tau(f)$ can be written as follows. The bundles $\qmtb_i$
provided by a $\bT_q$-fixed quasimap $f$ decompose as
$$
\qmtb_i\cong\bigoplus_{j=1}^{\dv_i}
\mathbb C_{w_{i,j}}\otimes\mathcal O(d_{i,j}),
$$
where $w_{i,j}$ are the additive $\bT$-weights of $\tb_i$. Write
$\tau=\tau(x)$ as a function of the Chern roots $x_{i,j}$ of the bundles
$\tb_i$. Then
$$
\tau(f)=\tau(x)|_{x_{i,j}=w_{i,j}+d_{i,j}q}.
$$
Thus, computing $\signver{\tau}_{X}\big|_{q=2\hbar}$ reduces to
understanding $\bT_q$-fixed quasimaps.

The following is an immediate corollary of Proposition \ref{prop: isolated qm}.

For a
$\mathsf T_q$-fixed quasimap $f$ based at $p\in X^{\mathsf T}$, recall
$$
  \mathcal A_d|_f=R\Gamma\!\left(\mathbb P^1,
        f^*T^{1/2}X(- p_2)\right).
$$

\begin{proposition}\label{prop:CY-isolated-localization}
Assume that
$\bigl(\qm^d_{\mathsf p_2\mapsto p}\bigr)^{\mathsf T_q}$ consists of
isolated points for every $d$.  Suppose, in addition, that no weight of $\mathcal A_d|_f$
restricts trivially to $\mathsf T_{\operatorname{CY}}$ for any such fixed point $f$.  Then
$$
  \left.\iota_p^*V_X^{(\tau)}\right|_{q=2\hbar}
  =\sum_d\sum_{f\in
       (\qm^d_{\mathsf p_2\mapsto p})^{\mathsf T_q}}
       \left.\tau(f)\right|_{q=2\hbar}\boldsymbol{z}^d.
$$
\end{proposition}

\section{Explicit formulas for certain \texorpdfstring{$q=2\hbar$}{q=2h} specialized vertex functions}\label{sec:vertpoint}

In this section, we compute the $q=2\hbar$ specialization of vertex functions at fixed points arising from {\emph{skew subheaps}} of {\emph{dominant-minuscule heaps}}. This class includes the point quiver varieties associated with ADE quivers having one-dimensional framing at a minuscule vertex. Consequently, the formulas below give explicit expressions for the geometric side of Conjecture \ref{eq:extended_conj_tr_Verma}. The extension from ordinary heaps to skew heaps was suggested by Alexis Leroux-Lapierre and Joel Kamnitzer, whom we thank for helpful discussions. See Section \ref{sssec_more_general} for possible extensions beyond the dominant-minuscule setting.

\subsection{Dominant minuscule skew-heap diagrams}\label{ssec: dominant minuscule heaps}
From now on we assume that $\Gamma$ has no loops.
Let $\mathfrak g_{\Gamma}$ be the corresponding symmetric 
Kac--Moody Lie algebra. Let $\alpha_i$ be simple roots, $\alpha_i^\vee$ be simple coroots, and $\omega_i$ be fundamental weights.
Let $W$ be the Weyl group. 

\begin{definition}
Let $\lambda$ be a weight. An element $w\in W$ is called \emph{$\lambda$-minuscule} if it has a reduced
expression $w=s_{i_1}\cdots s_{i_\ell}$ such that
\begin{equation}\label{eq:minuscule-condition}
  \left\langle
    \alpha_{i_k}^\vee,
    s_{i_{k+1}}\ldots s_{i_\ell}(\lambda)
  \right\rangle=1
  \qquad (1\leqslant k\leqslant \ell).
\end{equation}
Equivalently,
$$
  s_{i_k}\ldots s_{i_\ell}(\lambda)
  =\lambda-\alpha_{i_\ell}-\alpha_{i_{\ell-1}}-\ldots-\alpha_{i_k}.
$$
\end{definition}

\begin{definition}\label{def: heap}
For a reduced word $w=s_{i_1}\cdots s_{i_\ell}$, the {\emph{heap}} $H(w)$ is the partially ordered set $(\{1,2,\ldots,\ell\},\prec)$, where $\prec$ is the
transitive closure of
\begin{equation}\label{eq:heap-order}
  a\prec b
  \qquad\text{if}\qquad
  a>b
  \quad\text{and}\quad
  s_{i_a}s_{i_b}\neq s_{i_b}s_{i_a}.
\end{equation}
\end{definition}

Any $\lambda$-minuscule element $w$ is fully commutative (see \cite[Proposition 2.1]{stembridge}), i.e. any two reduced words differ by permuting some sequence of commuting simple reflections. Hence, its heap $H(w)$,
defined using \eqref{eq:heap-order}, is independent of the chosen reduced
word.  Its elements are colored by
$$
  {\boldsymbol{\pi}}\colon H(w)\longrightarrow I,
  \qquad {\boldsymbol{\pi}}(a)=i_a.
$$
For every $i\in I$, the fiber 
\begin{equation*}
H(w)_i := {\boldsymbol{\pi}}^{-1}(i)
\end{equation*}
is {\emph{totally}} ordered (see \cite[Remark 2.2]{heaps}), and
$H(w)$ is ranked (\cite[Corollary 3.4]{stembridge}).  
We denote its level function by
$$
  \mathrm{lev}\colon H(w)\longrightarrow\Z_{\geqslant 0}.
$$

The basic order-ideal correspondence (see \cite[Lemma 3.1]{S96}) gives
\begin{equation}\label{eq:ideal-correspondence}
  \{u\in W\mid u\leqslant_L w\}
  \xrightarrow{\sim}
  J(H(w)),
  \qquad
  u\longmapsto H(u),
\end{equation}
where $J(H(w))$ is the set of {\emph{order ideals}} and $\leqslant_L$ is the left weak Bruhat order. Every $u\leqslant_Lw$ is again $\lambda$-minuscule.

\begin{definition}
A \emph{skew subheap} inside $H(w)$ is a subposet
$$
  U=H(w)\setminus F,
$$
where $F=H(u)\subset H(w)$ for some $u \leqslant_L w$.  Thus, $U$ is an {\emph{order filter}}.  
\end{definition}

For such a skew subheap $U$, set
\begin{equation}\label{eq:dimension-vector}
  \dv_i(U)=|\pi^{-1}(i) \cap U|,
  \qquad
   {\mathsf{v}} = \mathsf{v}(U)=(\dv_i(U))_{i\in I}.
\end{equation}
\begin{equation}\label{eq:framing-vector}
  \dw_i(U)=\#\{x \in \pi^{-1}(i) \cap U\,|\, x\text{ is minimal in }U\},
  \qquad
{\mathsf{w}} =   \mathsf{w}(U)=(\dw_i(U))_{i\in I}.
\end{equation}
Because every $\pi^{-1}(i)$ is totally ordered, $\dw_i(U) \in \{0,1\}$.

\subsection{Points on quiver varieties from skew-heaps} 
Without loss of generality we may replace $\Gamma$ by the full subquiver on the {\emph{support}} $\boldsymbol{\pi}(U)$ of $U$, since both ${\mathsf{v}}(U)$ and $\mathsf{w}(U)$ vanish outside of this support. So, by \cite[Corollary 2.6]{stembridge}, we can assume that every connected component of $\Gamma$ is a {\emph{tree}}.

Let 
\begin{equation*}
\operatorname{Min}(U) = \{m_1,\ldots,m_s\}, \quad h_a :=\operatorname{lev}(m_a).
\end{equation*}
Consider the cocharacter 
\begin{equation}\label{eq:def_gamma_for_skew}
\gamma\colon \mathbb{C}^\times \rightarrow \mathsf{F}_\mathsf{w}, \quad t \mapsto (t^{-h_1},\ldots,t^{-h_s})
\end{equation}
of the framing torus $\mathsf{F}_{\mathsf{w}}=\prod_{i} (\mathbb{C}^\times)^{\dw_i}$.

To every skew-heap poset $U$ one can associate a $\mathsf{T}$-fixed point
\begin{equation*}
p_U \in \widetilde{\mathcal{M}}_H({\mathsf{v}},{\mathsf{w}})
\end{equation*}
for $\mathsf{T}=\tilde{\gamma}(\mathbb{C}^\times)$, where   
\begin{equation*}
\tilde{\gamma}\colon \mathbb{C}^\times \rightarrow \mathsf{F}_{\mathsf{w}} \times \mathbb{C}^\times_\hbar, \quad t \mapsto (\gamma(t),t).
\end{equation*}
It is defined as follows (compare with \cite[Section 3.3]{heaps}).

Fix a bipartite coloring of $I$ and the associated coloring of the covering relations of $H(w)$ by ${\mathsf{R}},\mathsf{B},\mathsf{G},\mathsf{Y}$, as in \cite[Proposition 3.5]{DEKMF}.
The existence of such colorings follows from the same argument as in the proof of \cite[Proposition 3.5]{DEKMF} since the authors only use \cite[Proposition 2.3]{stembridge}.

Let $\mathbb{C}U$ be the signed heap module defined as in \cite[Section 3.3]{heaps}, and use its dual $(\mathbb{C}U)^\vee$. Thus $V_i$ has a basis $\{e_x\,|\, x \in U_i\}$, and the arrows of the  quiver act, up to the signs prescribed by the coloring of the covering relations, by
\begin{equation*}
e_x \mapsto \pm e_y
\end{equation*}
whenever $y$ covers $x$.

Let $W_i$ have a basis indexed by the minimal elements of $U$ of color $i$. Define $A_i\colon W_i \rightarrow V_i$ by sending each such basis vector to the corresponding heap basis vector, and set $B_i=0$.

\begin{proposition}
The quadruple $(Y,Z,A,B)$ determines a $\mathsf{T}$-fixed point of the quiver variety $\widetilde{\mathcal{M}}_H({\mathsf{v}},{\mathsf{w}})$.  This point will be denoted $p_U$.  
\end{proposition}
\begin{proof}
The moment   map condition holds by \cite[Proposition 3.8]{heaps}.   

Every element of $U$ lies above some minimal element. Since the maps are nonzero along every covering relation and point upward, the smallest $Y,Z$-invariant collection of subspaces containing $\operatorname{im}(A)$ contains every basis vector $e_x$. Hence the framed representation is stable.

Define a gauge cocharacter by
\begin{equation*}
\kappa(t)e_x=t^{2+\mathrm{lev}(x)}e_x.
\end{equation*}
If $y$ covers $x$, then 
\begin{equation*}
\mathrm{lev}(y) = \mathrm{lev}(x)+1,
\end{equation*}
so 
\begin{equation*}
 \kappa(t)Y\kappa(t)^{-1}=tY, \quad \kappa(t)Z\kappa(t)^{-1}=tZ.   
\end{equation*}
We have
\begin{equation*}
\kappa(t)A=t^2A\gamma(t)^{-1}.
\end{equation*}
This is exactly the fixed-point condition for the action defined in (\ref{action_hbar_quiver}).
\end{proof}

\subsection{Vertex function \texorpdfstring{$V^{(\tau)}_{q=2\hbar,p_U}$}{VCYpU}}\label{sec:vertex_to_p_U}
Let $U$ be a skew-heap as above. Recall that $\mathsf{T}_q = \mathsf{T} \times \mathbb{C}^\times_q$.

\begin{definition}
    A reverse plane partition over a poset $U$ is a function 
    \begin{equation*}
    \phi\colon U \to \mathbb{Z}_{\geqslant 0}
    \end{equation*}
    such that $\phi(x)\leqslant \phi(y)$ whenever $x\ \prec y$. Let $\rpp(U)$ be the set of reverse plane partitions over $U$.
\end{definition}

We can fix ${\mathsf{d}} \in \mathbb{Z}_{\geqslant 0}^{I}$ and define
$$
\rpp^{\mathsf{d}}(U)=\left\{\phi \in \rpp(U)\, \mid \, \sum_{\substack{x \in U \\ {\boldsymbol{\pi}}(x)=i}} \phi(x)=d_i\text{ for every } i \in I\right\}
$$

We now describe the $\bT_{q}$-fixed quasimaps \begin{equation*}
f\colon \mathbb{P}^1 \rightarrow \widetilde{\mathcal{M}}_H(\mathsf{v}(U),{\mathsf{w}}(U))
\end{equation*}
based at $p_U$, i.e., sending $p_2 \in \mathbb{P}^{1}$ to $p_U$.

\begin{theorem}\label{thm: fixedqm}
  There is a natural bijection $\rpp^{\mathsf{d}}(U) \simeq \Big(\qm^{\mathsf{d}}_{p_2 \mapsto p_U}\Big)^{\bT_q}$. 
  
\end{theorem}
\begin{proof}
    In \cite[Section 7.2]{OkLec}, it is proven that $\Big(\qm^{\mathsf{d}}_{p_2 \mapsto p_U}\Big)^{\mathbb{C}^{\times}_{q}}$ is isomorphic to the space of framed flags of quiver representations ending in the stable point $p_U \in \mu^{-1}(0)$. 
    If the quasimap is additionally fixed by $\mathsf{T}$, then all vector spaces $V_i$ and the maps in the flag must also be graded by $\mathsf{T}$. A $\bT_{q}$-fixed quasimap provides a collection of non-negative integers $d_{x}$ for each $x \in U$ recording at what step in the flag the corresponding weight space appears. 

    Being a flag of quiver subrepresentations implies that $d_{x}\leqslant d_{y}$ whenever $x \preceq y$ in $U$. In other words, the map $x\mapsto d_{x}$ is a {\emph{reverse plane partition}} over $U$.

    Since all the $\mathsf{T}$-weights of each $V_i$ are distinct, the flag of quiver subrepresentations can be uniquely recovered, up to isomorphism, by the reverse plane partition.    
\end{proof}

Under the assumptions of this section, Proposition \ref{prop: isolated qm} gives the following. Order each fiber
\begin{equation*}
U_i = \{x_{i,1}\prec \ldots \prec x_{i,v_i}\}.
\end{equation*}
Let $\chi_{i,j}$ denote the additive $\mathsf{T}$-weight of the corresponding tautological line at $p_U$. Then define
\begin{equation*}
\tau(\phi) := \tau(x)|_{x_{i,j}=\chi_{i,j}+2\hbar\phi(x_{i,j})}.
\end{equation*}
This is the specialization described in Proposition \ref{prop: isolated qm}.

\begin{theorem}\label{thm: pointvertex}
\begin{equation}\label{eq:skew RPP formula}
\signver{\tau}_{p_U}\big|_{q=2\hbar}= \sum_{\phi \in \rpp(U)} \tau(\phi) \prod_{x \in U} z^{\phi(x)}_{{\boldsymbol{\pi}(x)}}
\end{equation}
\end{theorem}
\begin{proof}
This follows from Proposition \ref{prop:CY-isolated-localization} combined with Theorem \ref{thm: fixedqm}.
\end{proof}

The following is a corollary of Theorem~\ref{maintheorem: ADE} (compare
with the discussion around Equation
(\ref{eq:vertex gives solution of grtr application})).
\begin{corollary}
Assume that $U$ is as above supported on a subdiagram of an ADE Dynkin diagram.
Suppose also that $\mathsf{w}(U)$ is supported at minuscule nodes of this subdiagram. Then the RHS of (\ref{eq:skew RPP formula}) defines a graded trace on  $\mathcal{A}_f$ for $f$ coming from $\gamma$ in (\ref{eq:def_gamma_for_skew}).  
\end{corollary}

Let us finish this section with the following observation. Let $P=H(w)$ and $F=P \setminus U$. Dominant-minuscule heaps are $d$-complete, and $F$ is an order ideal. Therefore, by the skew hook formula of Naruse--Okada \cite[Theorem 1.2]{NaruseOkada}, the series
\begin{equation*}
L_U(z) := \sum_{\phi \in \mathrm{rpp}(U)}\prod_{x \in U} z^{\phi(x)}_{\pi(x)}
\end{equation*}
admits an explicit expression as a finite sum of rational functions indexed by the excited diagrams of $F$ in $P$. In particular, $L_U(z)$ is rational. When $U=P$, the set of excited diagrams consists only of the empty diagram, and the formula reduces to the Peterson-Proctor hook-product formula (\cite[Theorem 1.1]{NaruseOkada}, \cite{Proctorhook}); in this case the numerator is $1$. Combined with Theorem \ref{thm: pointvertex} this gives a rational expression for the specialized vertex function for $\tau=1$. Conjecture \ref{eq:extended_conj_tr_Verma} predicts that the same rational function is the normalized character of the corresponding Verma-type module. It would be interesting to interpret the numerator and denominator in the skew case representation-theoretically. In the case when $U=H(w)$ with $w$ being dominant minuscule, one can show that the factors in the Peterson--Proctor denominator  correspond to the negative tangent weights at the dual Coulomb-branch fixed point; compare with \cite[Sections 5, 6]{quiver_slant}.

\begin{remark}
In the case when $U=H(w)$ and $\lambda$ is minuscule, closely related reverse-plane-partition combinatorics appears in \cite{rppsGPT}, where reverse plane partitions are obtained from the generic Jordan types of nilpotent endomorphisms of quiver representations.
\end{remark}

\subsection{Explicit examples}

\subsubsection{Kleinian singularity}\label{sssec_kleinian_example}
Take $\Gamma=A_3$, ${\mathsf v}=(1,1,1)$ and ${\mathsf w}=(1,0,1)$; then
$X=\qv({\mathsf v},{\mathsf w})\cong
\widetilde{\mathbb A^2/(\mathbb Z/4\mathbb Z)}$.  Its middle
$\mathbb C^\times_\hbar$-fixed component is $C\cong\mathbb P^1$, whose
$\mathsf F_{\mathsf w}$-fixed endpoints $p_L,p_R$ are the product points of
Remark~\ref{fixpoints}.  The skew heap $U=\{x_1,x_2,x_3\}$ with
$x_1,x_3\prec x_2$ and ${\boldsymbol{\pi}}(x_i)=i$ gives a point $p_U\in C$.
Let $\varphi=f_1-f_3$ be the effective flavor parameter.

For $D\in\mathbb Z$, set
$$
[a]_D=\begin{cases}\prod_{m=1}^{D}(a+mq),&D\geqslant0,\\
\left(\prod_{m=0}^{-D-1}(a-mq)\right)^{-1},&D<0,
\end{cases}
\qquad
R_{\hbar,q}(a,D)=[a]_D[-2\hbar-a]_{-D}.
$$ 
Set also $R:=R_{\hbar,q}$. Writing $d=(d_1,d_2,d_3)$, localization gives the complete vertices
\begin{align*}
V_L&=\sum_{\substack{d_1,d_3\geqslant0\\d_2\geqslant d_1}}
\frac{z^d}{R(0,d_1)R(0,d_3)R(0,d_2-d_1)
R(-2\hbar-\varphi,d_3-d_2)},\\
V_R&=\sum_{\substack{d_1,d_3\geqslant0\\d_2\geqslant d_3}}
\frac{z^d}{R(0,d_1)R(0,d_3)R(-\varphi,d_2-d_1)
R(-2\hbar,d_3-d_2)},
\end{align*}
They are compatible with forgetting the flavor equivariance: at
$\varphi=0$ both become the complete skew-heap vertex
$$
V_U:=\sum_{\substack{d_1,d_3\geqslant0\\d_2\geqslant d_1,d_3}}
\frac{z^d}{R(0,d_1)R(0,d_3)R(0,d_2-d_1)R(-2\hbar,d_3-d_2)}.
$$
Nevertheless, the two specializations do not commute:
\begin{align*}
\left.V_U\right|_{q=2\hbar}
&=\frac{1-z_1z_2^2z_3}
{(1-z_2)(1-z_1z_2)(1-z_2z_3)(1-z_1z_2z_3)},\\
\left.V_L\right|_{q=2\hbar}
&=\frac{1}{(1-z_3)(1-z_2)(1-z_1z_2)},\\
\left.V_R\right|_{q=2\hbar}
&=\frac{1}{(1-z_1)(1-z_2)(1-z_2z_3)}.
\end{align*}
The last two expressions are the product vertices.  Already
$[z_3]V_L=\frac{2\hbar\varphi}{q(2\hbar+\varphi-q)}$ becomes $0$ if one first sets
$\varphi=0$, but $1$ if one first sets $q=2\hbar$.  Thus the complete vertices
respect equivariant-parameter specialization, whereas taking the
Calabi--Yau specialization first need not.  The polarization sign is trivial
in this example.

\subsubsection{More general $\mu \in W\lambda$ examples}\label{sssec_more_general}

We remark here that Theorem \ref{thm: pointvertex} holds more generally. Suppose $X$ is a quiver variety, and let $\mathsf{T} = \mathsf{F} \times \mathbb{C}^{\times}_{\hbar}$, where $\mathsf{F}$ is a flavor torus. A fixed point $p \in X^{\mathsf{T}}$ defines a homomorphism $\mathsf{T} \to G_{\mathsf{v}}$, unique up to conjugation, by which $\mathsf{T}$ acts on the vector spaces $V_{i}$. If $\mathsf{T}$ acts on each $V_i$ with simple spectrum, then a quiver representation for $p$ is determined by the poset whose elements are given by a choice of weight basis for $V_i$ for all $i$, with a covering relation whenever one basis vector is sent to another by the maps in the quiver representation. In this case, the proofs of Theorem \ref{thm: fixedqm} and Theorem \ref{thm: pointvertex} hold and give an explicit formula for $\signver{\tau}_{p}\big|_{q=2\hbar}$ in terms of reverse plane partitions over the poset.

We provide some examples here of this situation that go beyond the framework of a skew heap inside $H(w)$ for $\lambda$-minuscule $w$. These examples illustrate that none of the conditions on $w$ defined previously (i.e., dominant minuscule, minuscule, or fully commutative) coincides with the ``simple spectrum" condition of the previous paragraph. We consider it an interesting open problem to find a combinatorial characterization of such $w$.

All the examples below are for type $D$ quiver varieties which are points: so $\mu = w \lambda$. We order the vertices of the $D_n$ graph as in Section \ref{sec: type D}.

\begin{example}
    Consider the type $D_4$ quiver variety such that $\dv=(2,2,1,1)$ and $\dw=(0,1,0,0)$. Then $\lambda=\omega_2$, $\mu=-\alpha_1$, and $\mu=w \lambda$ where $w=s_1 s_2 s_3 s_4 s_2$. Then $w$ is minuscule but not dominant minuscule. The following colored poset encodes the unique point in the corresponding quiver variety. The poset has six elements, and we have depicted each by its color in $\{1,2,3,4\}$.
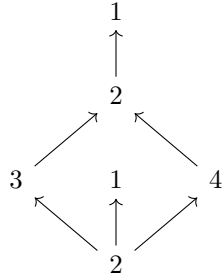
    \begin{center}
    \begin{tikzcd}
       & 1 & \\
       & 2 \arrow[u] & \\
      3 \arrow[ur] & 1 & 4 \arrow[ul]\\
       & 2 \arrow[u] \arrow[ul] \arrow[ur] & \\
    \end{tikzcd}
    \end{center}
\end{example}

\begin{example}
Consider the type $D_5$ quiver variety such that $\dv=(1,3,2,1,1)$ and $\dw=(0,1,0,0,0)$. Then $\lambda=\omega_2$, $\mu=-\alpha_2$, and $\mu=w \lambda$ where $w=s_2 s_1 s_3 s_4 s_5 s_3 s_2$. Then $w$ is fully commutative but not minuscule. The relevant poset is depicted below.
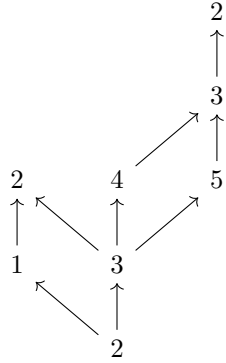
\begin{center}
    \begin{tikzcd}
       & & 2\\
        & & 3 \arrow[u]\\
        2 & 4 \arrow[ur] & 5 \arrow[u]\\
         1 \arrow[u] & 3 \arrow[ul] \arrow[u] \arrow[ur]& \\
           & 2 \arrow[ul] \arrow[u]& \\
    \end{tikzcd}
\end{center}
\end{example}

\begin{example}
Consider the type $D_5$ quiver variety such that $\dv=(2,3,2,1,1)$ and $\dw=(0,1,0,0,0)$. Then $\lambda=\omega_2$, $\mu=-\alpha_1-\alpha_2$, and $\mu=w \lambda$ where $w=s_2 s_1 s_2 s_3 s_4 s_5 s_3 s_2$. The element $w$ is not fully commutative. Nevertheless, the action of $\mathsf{T}$ on $V_i$ has simple spectrum, and torus fixed quasimaps are given by reverse plane partitions over the following poset.
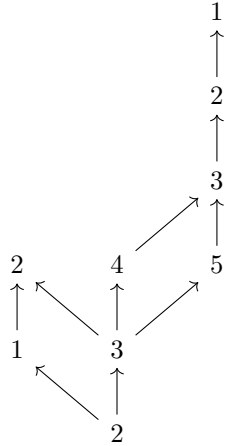
\begin{center}
    \begin{tikzcd}
      & & 1 \\
      & & 2 \arrow[u]\\
        & & 3 \arrow[u]\\
        2 & 4 \arrow[ur] & 5 \arrow[u]\\
         1 \arrow[u] & 3 \arrow[ul] \arrow[u] \arrow[ur]& \\
           & 2 \arrow[ul] \arrow[u]& \\
    \end{tikzcd}
\end{center}

\end{example}

These three examples illustrate that none of the dominant-minuscule, minuscule, and fully commutative conditions is the ``correct" combinatorial criterion on $w$ for obtaining a reverse plane partition formula.

\subsubsection{Fundamental with $\mu \notin W\lambda$ with isolated quasimap fixed points}
Let $\Gamma$ be of type $D_4$, with vertex $2$ trivalent, and take
$\lambda=\omega_2$, $\mu=0$, so that
$\dv=(1,2,1,1)$ and $\dw=(0,1,0,0)$.  Then
$X=\qv(\dv,\dw)$ is the minimal resolution of the $D_4$ Kleinian
singularity, and $X^{\mathbb C^\times_\hbar}$ contains the central
exceptional curve $C\simeq \mathbb P^1$.  For a generic $p\in C$, the
$\mathbb C^\times_\hbar$-graded representation is encoded by the colored
poset
$$
x_-\prec x_i\prec x_+\quad(i=1,3,4),\qquad
{\boldsymbol{\pi}}(x_\pm)=2,\quad {\boldsymbol{\pi}}(x_i)=i.
$$
The full $\mathbb C^\times_\hbar\times\mathbb C^\times_q$-fixed based
quasimaps are indexed by reverse plane partitions on this poset.  The
Calabi--Yau fixed locus is not discrete: in degree
$\theta=\alpha_1+2\alpha_2+\alpha_3+\alpha_4$, the two fixed points
$(0,1,1,1,2)$ and $(1,1,1,1,1)$, written in the order
$(x_-,x_1,x_3,x_4,x_+)$, lie on a $\mathbb P^1$-component.  Localizing
before setting $q=2\hbar$  nevertheless
gives a regular specialization.  Since the polarization sign is trivial in
this example, it is
\begin{align*}
\left.\ver{1}_{X,p}\right|_{q=2\hbar}
&=\sum_{0\leqslant a\leqslant b_i\leqslant c}
 z_2^{a+c}z_1^{b_1}z_3^{b_3}z_4^{b_4} \\
&=\frac{1}{1-z^\theta}\sum_{n\geqslant0}z_2^n
 \prod_{i\in\{1,3,4\}}(1+z_i+\cdots+z_i^n) \\
&=\frac{1}{(1-z^\theta)\prod_{i\in\{1,3,4\}}(1-z_i)}
 \sum_{S\subset\{1,3,4\}}
 \frac{(-1)^{|S|}z_S}{1-z_2z_S},
\end{align*}
where $z^\theta=z_1z_2^2z_3z_4$ and
$z_S=\prod_{i\in S}z_i$.  In particular, the coefficient of $z^\theta$
is $2$, in agreement with the above $\mathbb P^1$.  Conjecture
\ref{eq:extended_conj_tr_Verma} predicts that this rational function is the
normalized character of the corresponding Verma module on the Coulomb side.

\subsubsection{Fundamental with $\mu\notin W\lambda$ with nonisolated quasimap fixed points}
\label{sssec:fundamental-nonisolated}

The purpose of this subsection is to give an example, with $\lambda$ fundamental and $\mu\notin W\lambda$, in which the full $\bT_q$-fixed locus of based quasimaps is not discrete, but the Calabi--Yau specialized vertex can nevertheless be computed explicitly. In the example below, some fixed loci are positive-dimensional flag varieties; after setting $q=2\hbar$, their contributions to the signed vertex reduce to their ordinary Euler characteristics. An explicit stratification of these flag varieties then gives a closed rational formula for the complete vertex.

Let $\Gamma$ be of type $D_5$, with edges
$$
 1-2-3,\qquad 3-4,\qquad 3-5,
$$
so that vertex $3$ is trivalent.  Take
$$
 \lambda=\omega_3=e_1+e_2+e_3,\qquad
 \mu=e_4,\qquad
 \dv=(1,2,3,1,1),\qquad
 \dw=(0,0,1,0,0).
$$
The weight $e_4$ occurs in
$V(\omega_3)=\bigwedge^3(\mathbb C^{10})$, but it does not belong to
$W\omega_3$, since the elements of $W\omega_3$ are signed sums of
three distinct $e_i$.  Moreover,
$$
 \dim X=(\lambda,\lambda)-(\mu,\mu)=2,
 \qquad X=\qv(\dv,\dw).
$$

We describe a generic $\bT$-fixed point $p\in X$.  Let $H$ be a
two-dimensional vector space, and choose three distinct lines
$\mathbb C h_i\subset H$, for $i\in\{2,4,5\}$, such that
$$
 h_2+h_4+h_5=0.
$$
Choose also a nonzero form $\ell\in H^\vee$ whose kernel
$L_0:=\ker(\ell)$ is distinct from these three lines.  The graded
vector spaces of the corresponding representation are
$$
\begin{array}{c|cccc}
 &\text{degree }2&\text{degree }3&\text{degree }4&\text{degree }5\\ \hline
 V_1&0&0&\mathbb C&0\\
 V_2&0&\mathbb C&0&\mathbb C\\
 V_3&\mathbb C&0&H&0\\
 V_4&0&\mathbb C&0&0\\
 V_5&0&\mathbb C&0&0
\end{array}
$$
More explicitly, choose vectors
$$
 u\in V_3,\qquad x_i\in V_i\ (i=2,4,5),\qquad
 x_1\in V_1,\qquad y\in V_2.
$$
Orient the three edges away from vertex $3$ and the remaining edge
from $2$ to $1$.  The nonzero maps are
$$
 u\longmapsto x_i\longmapsto h_i\quad(i=2,4,5),
 \qquad
 x_2\longmapsto x_1,\qquad
 H\xrightarrow{\ell}\mathbb C y,
$$
together with a nonzero map $x_1\mapsto y$, whose scalar is chosen so
that the two length-two maps from $\mathbb Cx_2$ to $\mathbb Cy$
cancel in the moment-map relation at vertex $2$.  We also set
$A(1)=u$ and $B=0$.  The relation at vertex $3$ follows from
$h_2+h_4+h_5=0$, and all other moment-map relations are immediate.
The representation is stable, since it is generated by $u$.

In particular, $H\subset V_3$ is a
two-dimensional $\bT$-weight space; the one-dimensional framing torus
acts on it by a scalar and does not split it.  It follows from the
fixed-quasimap description of \cite[Section~7.2]{OkLec} that the
$\bT_q$-fixed based quasimaps are graded flags of subrepresentations of
this representation, and these flags need not be isolated.  For
example, in degree
$$
 \delta=\alpha_2+\alpha_3
$$
assign quasimap degree $0$ to $u,x_2,x_4,x_5,x_1$, degrees $0$ and $1$
to the two graded lines in $H$, and degree $1$ to $y$.  The
degree $1$ line in $H$ is arbitrary.  These are all the fixed flags of
degree $\delta$, and hence
\begin{equation}
 \label{eq:D5-nonisolated-fixed-locus}
 \left(\qm^\delta_{p_2\mapsto p}\right)^{\bT_q}
 \cong\mathbb P(H)\cong\mathbb P^1.
\end{equation}

We next compute the complete specialized vertex.  A fixed flag is
described by nonnegative integers
$$
 a,\quad b_2,b_4,b_5,\quad c,\quad r\leqslant s,\quad e,
$$
which are the quasimap degrees of
$u,x_2,x_4,x_5,x_1$, the two graded lines in $H$, and $y$,
respectively.  Its degree monomial is
$$
 m(a,\boldsymbol b,c,r,s,e)
 =z_1^c z_2^{b_2+e}z_3^{a+r+s}z_4^{b_4}z_5^{b_5}.
$$
If $r=s$, there is no flag to choose in $H$, and the inequalities are
\begin{equation}
 \label{eq:D5-equal-stratum}
 a\leqslant b_i\leqslant r\quad(i=2,4,5),\qquad
 b_2\leqslant c\leqslant e,\qquad r\leqslant e.
\end{equation}
Suppose that $r<s$, and let $L\subset H$ be the degree-$s$ line.  If
$$
 L\notin\{\mathbb Ch_2,\mathbb Ch_4,\mathbb Ch_5,L_0\},
$$
then
\begin{equation}
 \label{eq:D5-generic-stratum}
 a\leqslant b_i\leqslant r\quad(i=2,4,5),\qquad
 b_2\leqslant c\leqslant e, \qquad s\leqslant e.
\end{equation}
This stratum is a projective line with four points removed and has
Euler characteristic $-2$.  If $L=\mathbb Ch_j$, for
$j\in\{2,4,5\}$, the inequalities are
\begin{equation}
 \label{eq:D5-hj-stratum}
 a\leqslant b_j\leqslant s,\qquad
 a\leqslant b_i\leqslant r\ (i\neq j),\qquad
 b_2\leqslant c\leqslant e,\qquad s\leqslant e.
\end{equation}
Finally, if $L=L_0$, they are
\begin{equation}
 \label{eq:D5-kernel-stratum}
 a\leqslant b_i\leqslant r\quad(i=2,4,5),\qquad
 b_2\leqslant c\leqslant e,\qquad r\leqslant e.
\end{equation}

At $q=2\hbar$, the self-dual localization calculation of
Section~\ref{CYvertex}, together with the polarization sign, identifies
the contribution of a fixed flag variety with its topological Euler
characteristic.  Applying this to the four strata above and summing the
resulting geometric series gives the following expression for $\signver{1}_{X,p}|_{q=2\hbar}$:
\begin{align}
 {1\over
 (1-z_3)(1-z_2z_3)(1-z_1z_2z_3)}
 \sum_{0\leqslant A\leqslant B_2,B_4,B_5\leqslant C}
 z_2^{A+C}z_1^{B_2}(z_3z_4)^{B_4}(z_3z_5)^{B_5}.
 \label{eq:D5-nonisolated-vertex}
\end{align}
In particular, this is a rational function.  To write it explicitly,
put
$$
 \vartheta=\alpha_1+2\alpha_2+2\alpha_3+\alpha_4+\alpha_5,
 \qquad
 \mathcal U=\{z_1,z_3z_4,z_3z_5\},
$$
and, for $S\subset\mathcal U$, set $u_S=\prod_{u\in S}u$.  Then
\begin{equation}
 \label{eq:D5-nonisolated-rational-vertex}
 \left.\signver{1}_{X,p}\right|_{q=2\hbar}
 =
 {\displaystyle
  \sum_{S\subset\mathcal U}
  {(-1)^{|S|}u_S\over1-z_2u_S}
  \over\displaystyle
  (1-z_3)(1-z_2z_3)(1-z_1z_2z_3)
  (1-z^\vartheta)\prod_{u\in\mathcal U}(1-u)}.
\end{equation}
The coefficient predicted by
\eqref{eq:D5-nonisolated-fixed-locus} is visible directly:
$$
 [z_2z_3]\left.\signver{1}_{X,p}\right|_{q=2\hbar}
 =\chi(\mathbb P^1)=2.
$$

\subsubsection{Case of $T^*(G/P)$}\label{eq: ssec case of cootangent}
This is a continuation of Remark \ref{rem: vertex for  cotangent} above. Recall that 
\begin{equation*}
X=T^*(\operatorname{GL}_n/P), \quad Y=X^{\mathbb{C}^\times_\hbar}=\operatorname{GL}_n/P.
\end{equation*}
Set
$$
 \rho=q\hbar^{-2},
 \qquad
 \mathsf H:=\ker(\rho).
$$
We claim that the classical truncation of
$$
 \left(
 \qm^d_{\ns p_2}\mathop{\times}^{\mathbf R}_X Y
 \right)^{\mathsf H}
$$
is proper.  Let
$$
 \nu\colon T^*(\operatorname{GL}_n/P)
 \longrightarrow\mathfrak{gl}_n^*
$$
be the Springer moment map.  In the standard type~$A$ quiver
presentation, $\nu$ is induced already on
$[\mu^{-1}(0)/G_{\dv}]$ by the gauge-invariant framing moment map
$BA\in\operatorname{End}(W)$.  Consequently, every quasimap $f$
determines a regular morphism
$$
 \nu(f)\colon\mathbb P^1\longrightarrow\mathfrak{gl}_n^*.
$$
Indeed, $A$ and $B$ are morphisms between the corresponding vector
bundles, while the framing bundle $W\otimes\mathcal O_{\mathbb P^1}$
is trivial, so that $BA$ is a global section of
$\operatorname{End}(W)\otimes\mathcal O_{\mathbb P^1}$.  It is therefore
constant.  If $\ev_{p_2}(f)\in Y$, then $\nu(f)(p_2)=0$, and hence
$\nu(f)=0$.  Since
$$
 \nu^{-1}(0)=Y\subset T^*(\operatorname{GL}_n/P),
$$
the quasimap lands in the zero section on its generically stable locus.
Thus all its cotangent quiver maps vanish on a dense open subset of
$\mathbb P^1$, and consequently vanish identically as morphisms of
vector bundles.

It remains to consider quasimaps to the flag variety $Y$.  On the
zero-section quiver data, the $\mathbb C_\hbar^\times$-action is
gauge-trivial.  Moreover, the projection
$$
 \mathsf H\longrightarrow\mathbb C_q^\times
$$
is a degree-two isogeny.  After twisting the equivariant structures by
the corresponding central gauge cocharacter, $\mathsf H$-fixed
quasimaps are therefore the same as $\mathbb C_q^\times$-fixed
quasimaps to $Y$, with the grading pulled back along this isogeny.

By the fixed-quasimap description of
\cite[Sections~7.2.6--7.2.16]{OkLec}, a fixed quasimap ending at
$y\in Y$ is described by finite flags
$$
 \mathbb V_i[k]\subseteq\mathbb V_i[k-1]
 \subseteq\mathcal V_i|_y,
$$
compatible with the universal quotient maps, such that
$$
 \mathbb V_i[k]=\mathcal V_i|_y\quad(k\leq0),
 \qquad
 \mathbb V_i[k]=0\quad(k\gg0),
 \qquad
 d_i=\sum_{k>0}\dim\mathbb V_i[k].
$$
Thus only finitely many rank patterns can occur in a fixed degree.
For each rank pattern, compatibility with the quiver maps is a closed
incidence condition in a product of relative flag varieties of the
tautological bundles on $Y$.  These relative flag varieties are
projective over the projective variety $Y$, and hence every such
incidence locus is proper.  The classical truncation of the fixed locus
is therefore proper.  Since properness is unaffected by nilpotent or
derived thickening, the required derived fixed locus is proper as well.

On every connected component $K$,
the moving part of the obstruction theory relative to the evaluation
morphism has the form $B_K-\rho B_K^\vee$. Its localization factor is consequently $\frac{e(\rho B_K^\vee)}{e(B_K)},$ which is regular at $\rho=1$ and there equals
$(-1)^{\operatorname{rk}B_K}$. It follows that the restriction of the
vertex to $Y$ is regular at $q=2\hbar$.

Finally, since $X$ is an equivariant vector bundle over $Y$, restriction
to the zero section induces an isomorphism
$$
 H^*_{\mathbb C_\hbar^\times}(X)
 \xrightarrow{\ \sim\ }
 H^*_{\mathbb C_\hbar^\times}(Y).
$$
Thus the specialized vertex may equivalently be regarded as a
nonlocalized class on $X$.

For the descendant $1$, the vertex function has degree $0$. This cohomological degree can be
determined as follows.  By \eqref{Tvir} and Riemann--Roch,
$$
 \operatorname{vdim}\qm_{\ns p_2}^d
 =
 \chi(C_0,\qmpol)+\chi(C_0,(\qmpol)^\vee)
 =
 2\operatorname{rk}(\qmpol)
 =
 \dim X.
$$
Consequently,
$$
 \operatorname{vdim}\left(
 \qm_{\ns p_2}^d\mathop{\times}^{\mathbf R}_X Y
 \right)=\dim Y,
$$
so every coefficient of
$\left.V^{(1)}|_Y\right|_{q=2\hbar}$ has equivariant
cohomological degree zero.  Since the preceding argument shows that
these coefficients are nonlocalized, they belong to
$$
 H^0_{\mathbb C_\hbar^\times}(Y)=H^0(Y)=H_{\mathrm{top}}(Y).
$$
Explicitly, we have 
$$
V^{(1)}|_{q=2\hbar}=\Big(\frac{1}{\prod_{i=1}^{n-1} \prod_{j=i}^{n-1}(1-z_i \ldots z_j)^{\dv_i-\dv_{i-1}}}\Big)[Y]$$
where $\dv$ is the corresponding dimension vector.

\section{Vertex functions for AD points}\label{sec:ADE-point-vertices}

Combining Theorem \ref{thm: fixedqm} with Proposition \ref{prop: isolated qm}, it follows that the vertex function for an ADE quiver variety with one minuscule framing is a sum over reverse plane partitions over a heap. Although it is not strictly needed for the purposes of the quantum Hikita conjecture, we will demonstrate how vertex functions can be written in terms of symmetric polynomials. For that, we will do a case-by-case analysis to obtain a more explicit description of the heaps $H(w)$ from Definition \ref{def: heap}. 

As a consequence, we will prove rationality of the $q=2\hbar$ specialized {\emph{descendant}} vertex function. By the identification with graded traces, this gives a new proof of a result of Etingof--Stryker (\cite[Theorem 4.9]{etingof-stryker}). We also strengthen their result by providing refined information about the possible poles.

\subsection{Symmetric functions}

We collect here some basic facts about symmetric functions that we will need, all of which can be found in \cite{mac}. For partitions $\lambda,\mu$, let $s_{\lambda/\mu}(x)$ be the skew-Schur function in the variables $x=(x_1,x_2,\ldots)$.

The branching rule is
$$
\sum_{\mu} s_{\lambda/\mu}(x) s_{\mu/\nu}(y)= s_{\lambda/\nu}(x,y)
$$

We say that a partition $\lambda$ interlaces a partition $\mu$ from above, written $\lambda \succ \mu$, if $\lambda_1\geq \mu_1 \geq \lambda_2 \geq \mu_2 \geq \ldots$.

  The single-variable specialization of a skew-Schur function is
    \begin{equation}\label{eq: single variable schur}
s_{\lambda/\mu}(x)=\begin{cases}
    x^{|\lambda|-|\mu|} & \text{if $\lambda \succ \mu$} \\
    0 & \text{otherwise}
\end{cases}
    \end{equation}

    The skew Cauchy identity states that
\begin{equation}\label{eq: skew Cauchy}
\sum_{\lambda} s_{\lambda/\mu}(x) s_{\lambda/\nu}(y)= \left(\prod_{i,j}\frac{1}{1-x_i y_j}\right) \sum_{\gamma} s_{\nu/\gamma}(x) s_{\mu/\gamma}(y)
\end{equation}

For a finite set of variables $x=(x_1,x_2,\ldots,x_n)$, it is straightforward to see from the determinant formula for Schur functions that the operator 
\begin{equation}\label{eq: schur diff op}
D^{r}(x)=2^r \Delta(x)^{-1} e_{r}\left(x_1 \frac{\partial}{\partial x_1}, \ldots, x_n \frac{\partial}{\partial x_n} \right) \Delta(x)
\end{equation}
acts on Schur polynomials for partitions $\lambda$ of length at most $n$ by 
\begin{equation}\label{eq: diff op evals}
D^r(x) s_{\lambda}(x)=e_{r}(2(n+\lambda_{1}-1),2(n+\lambda_{2}-2),\ldots,2 \lambda_{n}) s_{\lambda}(x)
\end{equation}
where $\Delta=\prod_{1 \leq i < j \leq n}(x_i-x_j)$ is the Vandermonde determinant and $e_r$ is the $r$th elementary symmetric polynomial. These operators can be deformed to the Sekiguchi operators; see \cite[320]{mac}.

\subsection{Type \texorpdfstring{$A_{n}$}{An}, weight \texorpdfstring{$\fundwt_{k}$}{wk}, \texorpdfstring{$1 \leq k \leq n$}{1<k<n}}

We begin with the case of the type $A_{n}$ quiver $\Gamma$ which has vertices $I=\{1,2,\ldots,n\}$ and arrows as below:
$$
\begin{tikzpicture}[circ/.style={shape=circle,draw,inner sep=1.5pt}]
    \node[circ] (1) at (0,0) {1};
    \node[circ, right = of 1](2){2};
    \node[right = of 2](3){$\ldots$};
    \node[circ, right = of 3](4){$n$};

    \draw[->] (1) -- (2);
    \draw[->] (2) -- (3);
    \draw[->] (3) -- (4);
    
\end{tikzpicture}
$$

Fix $k \in I$ and let $\lambda=\fundwt_{k}$. Let $\mathcal{P}_{n,k}$ be the set of partitions 
\begin{equation*}
\nu=(\nu_{1}\geq \nu_{2}\geq \ldots)
\end{equation*} 
such that $\nu_1 \leq k$ and $l(\nu)+\nu_{1}-1\leq n$, where $l(\nu)$ is the length of $\nu$.

The set of points
$$
\{(i,j) \, \mid \, 0 \leq i \leq l(\nu)-1 , 0 \leq j \leq \nu_{i}-1 \} \subset \mathbb{Z}^{2}
$$
is called the Young diagram of $\nu$. We identify a partition with its Young diagram. The content of a box $\square =(i,j) \in \nu$ is defined to be $c(\square)=i-j$. Young diagrams can be naturally interpreted as heaps in full agreement with Section \ref{ssec: dominant minuscule heaps}, which leads to the following lemma.

\begin{lemma}
    There are bijections
    $$
\{\dv \in \mathbb{N}^{n} \, \mid \, \qv(\dv,\delta_{k}) \neq 0\} = W \lambda =\mathcal{P}_{n,k}
    $$
\end{lemma}
\begin{proof}
    The first bijection follows from the fact that $\lambda$ is a minuscule weight, and that quiver varieties correspond to weight spaces in irreducible representations of Lie algebras.

   From $\nu \in \mathcal{P}_{n,k}$, we construct a poset $H$ as follows. The elements of $H$ are the boxes in the Young diagram of $\nu$, and there are covering relations $(i+1,j) \geq (i,j)$ and $(i,j+1) \geq (i,j)$ whenever all of these points are in the Young diagram. By rotating the Young diagram to the ``Russian convention" and reading the colors of the boxes from bottom to top, we obtain an element $w$ such that $H=H(w)$. Mapping $\nu$ to $w \lambda$ provides the second bijection.

\end{proof}

An example of the correspondence between $\mathcal{P}_{n,k}$ and $\qv(\dv,\delta_{k})$ is shown in Figure \ref{Anrep}.

\begin{figure}[ht]
    \centering
\begin{tikzpicture}[circ/.style={shape=circle,draw,inner sep=1.5pt}]
    \node[circ, on grid] (41) at (0,0){};
    \node[circ, on grid, above  right= of 41](51){};
    \node[circ, on grid, above right = of 51](61){};
    \node[circ, on grid, above right = of 41](51){};
    \node[circ, on grid, above left = of 41](31){};
    \node[circ, on grid, above left = of 31](21){};
    \node[circ, on grid, above left = of 21](11){};
    \node[circ, on grid, above left = of 51](42){};
    \node[circ, on grid, above left = of 42](32){};

\draw[->] (41) edge (51) (51) edge (61) (41) edge (31) (31) edge (21) (21) edge (11) (51) edge (42) (31) edge (42) (21) edge (32) (42) edge (32);
    
\node[circ, on grid, below= of 41](4){};
\node[circ, on grid, left= of 4](3){};
\node[circ, on grid, left= of 3](2){};
\node[circ, on grid, left= of 2](1){};

\node[circ, on grid, right= of 4](5){};
\node[circ, on grid, right= of 5](6){};
\node[circ, on grid, right= of 6](7){};

\draw (1) -- (2);
\draw (2) -- (3);
\draw (3) -- (4);
\draw (4) -- (5);
\draw (5) -- (6);
\draw (6) -- (7);
\end{tikzpicture}
 \caption{An example of the heap type $A$ point quiver variety $\qv(\dv_{\nu},\delta_{k})$. Here, $n=7$, $k=4$, and $\nu=(4,3,1)$.}
    \label{Anrep}
\end{figure}
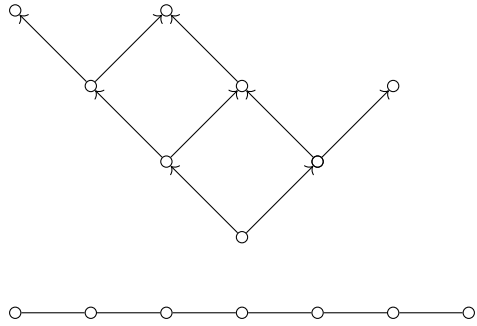

Now fix $k$, $n$, and $\nu \in \mathcal{P}_{n,k}$ for the remainder of this section. Let $\dv$ be the vector corresponding to $\nu$; it is the vector recording the number of boxes of content $i$, written $c^{\nu}_{i}$, for all $i$. Let $\qv:=\qv(\dv,\delta_{k})$. Let $\signver{\tau}$ be the (signed) descendant vertex function of $\qv$.

By Theorem \ref{thm: fixedqm}, torus fixed quasimaps are in bijection with reverse plane partitions over $\nu$, the set of which will be denoted by $\rpp(\nu)$.

A reverse plane partition $\pi \in \rpp(\nu)$ can be visualized as a 3d collection of boxes, where $\pi_{i,j}$ boxes are stacked over the point $(i,j)$. Let
$$
c^{\pi}_{i}=\sum_{\substack{l,m \\l-m =i }} \pi_{l,m}
$$
which gives the total number of 3d boxes in $\pi$ with content $i$.

In this case, Theorem \ref{thm: fixedqm} and Proposition \ref{prop: isolated qm} give the following.

\begin{proposition}
    The $q=2\hbar$ specialization of the vertex function of $\qv(\dv,\delta_{k})$ is
    $$
\signver{\tau}\big|_{q=2\hbar}=\sum_{\pi \in \rpp(\nu)} \tau(\pi) \prod_{i} z_{i}^{c^{\pi}_{i-k}}
    $$
\end{proposition}

We can recast the sum over reverse plane partitions as a sum over sequences of interlacing partitions. We define a sign map $\epsilon:\{1,2,\ldots,n+1\} \to \{-1,1\}$ associated to $\nu$ by
$$
\epsilon(i)=\begin{cases}
    (-1)^{c^\nu_{i}-c^\nu_{i-1}+1} & \text{if $i \leq k$} \\
     (-1)^{c^\nu_{i}-c^\nu_{i-1}} & \text{if $i>k$}
\end{cases}
$$
with the convention that $c^{\nu}_{0}=0$. These signs record the slope of the top boundary of $\nu$, drawn in the Russian convention.

\begin{definition}
    We say that a tuple of partitions 
    $$
\ggamma=(\gamma^{(0)}=\emptyset,\gamma^{(1)},\gamma^{(2)},\ldots,\gamma^{(n)},\gamma^{(n+1)}=\emptyset)
$$
interlaces according to the shape of $\nu$ if
    \begin{itemize}
        \item $\epsilon(i)=1 \implies \gamma^{(i)} \succ \gamma^{(i-1)}$
        \item $\epsilon(i)=-1 \implies \gamma^{(i)} \prec \gamma^{(i-1)}$.
    \end{itemize}
    for $1 \leq i \leq n+1$. Let $S_{\nu}$ be the set of tuples of partitions that interlace according to the shape of $\nu$.
\end{definition}

It is well-known and easy to see that $S_{\nu}$ and $\rpp(\nu)$ are in canonical bijection. Using the single-variable specialization of skew Schur functions \eqref{eq: single variable schur}, we obtain the following.

\begin{proposition}\label{prop: A schur vertex}
    For $i \in \{1,\ldots,n+1\}$, let $x_i$ denote a single variable. Then
    $$
 \signver{\tau}\big|_{q=2\hbar}=\sum_{\ggamma \in S_{\nu}} \tau(\ggamma) \left(\prod_{\substack{i=1 \\ \epsilon(i)=1}}^{n+1} s_{\gamma^{(i)}/\gamma^{(i-1)}}(x_i) \right) \left(\prod_{\substack{i=1 \\ \epsilon(i)=-1}}^{n+1} s_{\gamma^{(i-1)}/\gamma^{(i)}}(x_i)\right) 
    $$
    where
    $$
z_i=x_i^{\epsilon(i)} x_{i+1}^{\epsilon(i+1)} 
    $$
\end{proposition}

Proposition \ref{prop: A schur vertex} also shows that $\signver{\tau}|_{q=2\hbar}$ is the partition function of a Schur process from \cite{ORSchur}.

Using the skew Cauchy identity \eqref{eq: skew Cauchy}, one can sum the right-hand side of Proposition \ref{prop: A schur vertex}. Denote by $h^{\nu}_{\square}$ the hook in $\nu$ containing $\square$, which by definition consists of all boxes in $\nu$ (weakly) above and right of $\square$.

\begin{proposition}\label{prop: typeAvertex}
   There is an equality of rational functions of $z$:
    $$
    \signver{1}\big|_{q=2\hbar}=\prod_{\square \in \nu} \frac{1}{1-\prod_{\square' \in h^{\nu}_{\square}} z_{c(\square')+k}}
    $$
\end{proposition}

 The $z_{i}=z$ specialization of the previous result was originally proven by Stanley in Proposition 18.3 of \cite{StanleyPP}. Proposition \ref{prop: typeAvertex} is a particular case of the Peterson--Proctor hook-product formula discussed at the end of Section \ref{sec:vertex_to_p_U}.

We can insert descendants into vertex functions by applying some differential operators.
\begin{definition}
    Fix $i \in I$. Let $$
\Delta_{i}=\begin{cases}
   \prod\limits_{\substack{1 \leq l <m \leq i \\ \epsilon(l)=1 \\ \epsilon(m)=-1}} \frac{1}{1-x_l x_m}& \text{if $i\leq k$} \\
   \prod\limits_{\substack{i+1 \leq l <m \leq n \\ \epsilon(l)=1 \\ \epsilon(m)=-1}} \frac{1}{1-x_l x_m}& \text{if $i> k$}
\end{cases}
    $$
    and 
    $$
X^{(i)}=\begin{cases}
   (x_{j})_{1 \leq j \leq i, \epsilon(j)=1} & \text{if $i \leq k$} \\
   (x_{j})_{i+1 \leq j \leq n, \epsilon(j)=-1}  & \text{if $i >k$}
\end{cases}
    $$
    Let $\diffop^{r}_i$ be the differential operator
    $$
\diffop^{r}_{i} = \Delta_{i} \left(\frac{q}{2}\right)^r D^{r}(X^{(i)}) \Delta_{i}^{-1} 
    $$
    where $D^{r}$ is from \eqref{eq: schur diff op}.
\end{definition}

For convenience, we define shifted Chern classes as follows. Let $t_{i,1},\ldots,t_{i,\dv_{i}}$ be the Chern roots of $\tb_{i}$, i.e. the weights of $\tb_{i}$. Define $\tilde{c}_{r}(\tb_{i})$ by
$$
\sum_{r=0}^{\dv_{i}} \tilde{c}_{r}(\tb_{i}) X^{r}= \prod_{j=1}^{\dv_{i}}\left(1+\bigl(t_{i,j}-(2+|i-k|)\hbar\bigr) X\right)
$$
\begin{theorem}\label{thm: diffop vertex}
Fix $i \in I$. Suppose $\tau$ is a polynomial in the Chern classes of the tautological bundles $\tb_{j}$ for $j$ satisfying $|j-k|\leq |i-k|$.

Then
    $$
\signver{\tilde{c}_{r}(\tb_{i})\tau}\big|_{q=2\hbar}=\diffop_{i}^{r} \signver{\tau}\big|_{q=2\hbar}
    $$
\end{theorem}
\begin{proof}
    For definiteness, suppose $i\leq k$. By applying the branching rule and the skew Cauchy identities, we can write Proposition \ref{prop: A schur vertex} in the form
    $$
    \signver{\tau}\big|_{q=2\hbar}=\Delta_i \sum_{\substack{\gamma^{(j)} \\ j \geq i}} \tau(\gamma^{(i)},\gamma^{(i+1)},\ldots, \gamma^{(n)}) s_{\gamma^{(i)}}(X^{(i)}) \ldots
    $$
    since $\tau$ depends only on partitions $\gamma^{(j)}$ for $j \geq i$. Here the dots stand for sums of products of skew Schur polynomials that depend only on $\gamma^{(j)}$ for $j>i$ and do not depend on the variables $X^{(i)}$. At the fixed quasimap indexed by $\gamma^{(i)}$, one has
    $$
    t_{i,j}=(2+|i-k|+2(\dv_i-j))\hbar+q\gamma_j^{(i)},
    $$
    and hence, at $q=2\hbar$,
    $$
    t_{i,j}-(2+|i-k|)\hbar=q(\dv_i+\gamma_j^{(i)}-j).
    $$
    Applying \eqref{eq: diff op evals} gives
    \begin{align*}
    \diffop_{i}^{r} \signver{\tau}\big|_{q=2\hbar}&= \Delta_i \sum_{\substack{\gamma^{(j)} \\ j \geq i}}  q^r e_{r}\bigl((\dv_{i}+\gamma^{(i)}_{j}-j)_{j=1}^{\dv_i}\bigr) \tau(\gamma^{(i)},\gamma^{(i+1)},\ldots, \gamma^{(n)}) s_{\gamma^{(i)}}(X^{(i)}) \ldots \\
    &= \signver{\tilde{c}_r(\tb_{i}) \tau} \big|_{q=2\hbar}
     \end{align*}
\end{proof}

\begin{theorem}\label{thm: A rationality}
    For any $\tau$, $\signver{\tau}|_{q=2\hbar}$ is a rational function of $z$ with poles only at $z^{\alpha}=1$ for some root $\alpha$.
\end{theorem}
\begin{proof}
    By linearity, it is sufficient to prove the result for $\tau=\prod_{i=1}^{n} \prod_{r=1}^{\dv_{i}} \tilde{c}_{r}(\tb_{i})^{a_{i,r}}$ for $a_{i,r} \in \mathbb{N}$.

We proceed by induction on $\sum_{i,j} a_{i,r}$. The base case $\tau=1$ follows from Proposition \ref{prop: typeAvertex}. 

Suppose that not all $a_{i,r}$ are zero. Let $i_0$ be the minimal index such that $a_{i_0,r}\neq 0$ for some $r_0$. For definiteness, we suppose $i_0\leq k$. Define $\tau'$ by $\tau=\tilde{c}_{r_0}(\tb_{i_0}) \tau'$. By induction, we can assume that $\signver{\tau'}$ is rational. Furthermore, Theorem \ref{thm: diffop vertex} states that $\signver{\tau}|_{q=2\hbar}= \diffop^{r_0}_{i_0} \signver{\tau'}|_{q=2\hbar}$. Since $\diffop^{r_0}_{i_0}$ is rational in $z$, $\signver{\tau}|_{q=2\hbar}$ is also rational.
\end{proof}

\subsection{Type \texorpdfstring{$D_{n}$}{Dn}, weight \texorpdfstring{$\fundwt_{n}$}{wn}}\label{sec: type D}

Let $\Gamma$ be the type $D_{n}$ quiver labeled and oriented as below:

$$
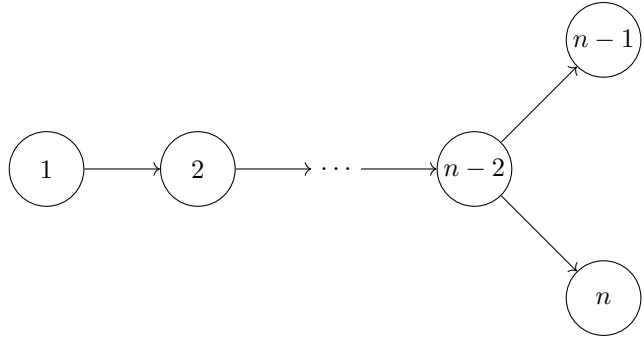
\begin{tikzpicture}[circ/.style={shape=circle,draw,inner sep=1.5pt}]
    \node[circ,label=center:1] (1) at (0,0) {\phantom{$n-1$}};
    \node[circ, right = of 1,label=center:2](2){\phantom{$n-2$}};
    \node[right = of 2](3){$\ldots$};
    \node[circ,right=of 3,label=center:$n-2$](4){\phantom{$n-2$}};
    \node[circ, above right = of 4,label=center:$n-1$](5){\phantom{$n-2$}};
    \node[circ, below right = of 4, label=center:$n$](6){\phantom{$n-2$}};

    \draw[->] (1) -- (2);
    \draw[->] (2) -- (3);
    \draw[->] (3) -- (4);
    \draw[->] (4) -- (5);
    \draw[->] (4) -- (6);
    
\end{tikzpicture}
$$

The two half spin representations are minuscule fundamental representations. The corresponding quiver varieties and vertex functions were studied in \cite{dinkjang}. We recall the important points here.

\begin{definition}
   A type $D_{n}$ lattice path is a finite sequence of edges $\nu=(\nu_{1},\nu_{2},\ldots,\nu_{n})$  where $\nu_{i}=(p_{i},p_{i+1})$ and $p_{j} \in \mathbb{Z}^{2}_{\geq 0}$ such that $p_{i+1}-p_{i}\in \{(1,0),(0,-1)\}$, $p_{1}=(0,n-1)$, $p_{n} \in \text{diag} \subset \mathbb{Z}^{2}_{\geq 0}$, and $\nu_{n}$ is the segment from $p_{n}$ to $(0,0)$. The set of all type $D_{n}$ lattice paths will be denoted by $\mathcal{P}_{n}$.
\end{definition}

We associate to a type $D_{n}$ lattice path $\nu$ a vector $\dv_{\nu} \in \mathbb{Z}^{n}_{\geq 0}$ in the following way:

$$
\dv_{\gamma,j}:=
\begin{cases}
    x(p_{j+1}) & 1 \leq j \leq n-2 \\
  \lfloor \frac{x(p_{n})}{2} \rfloor & j=n-1 \\
   \lceil \frac{x(p_{n})}{2} \rceil  & j=n
\end{cases}
$$
where $x(p_{k})$ means the $x$-coordinate of $p_{k}$.

\begin{lemma}
There are bijections
$$
\{ \dv \in \mathbb{N}^{n} \, \mid \, \qv(\dv,\delta_{n}) \neq \emptyset \} = W \lambda = \mathcal{P}_{n}
$$
\end{lemma}
\begin{proof}
    The first bijection follows for the same reasons as in type $A$.

    Fix $\nu \in \mathcal{P}_{n}$. We describe a colored poset $H$. Elements of $H$ are given by lattice points in the bounded region enclosed by $\nu$, the $y$-axis, and the line $y=x$. We allow points on the $y$-axis and the line $y=x$ but exclude those on $\nu$. There are covering relations $(i+1,j)\geq (i,j)$ and $(i,j+1) \geq (i,j)$ whenever these points are in $H$. The color of a vertex $(i,j)$ with $i \neq j$ is defined to be $i-j+n$. The color of a vertex $(i,i)$ is $n-\overline{i}$, where $\overline{i}$ denotes the residue of $i$ modulo $2$. By rotating the plane 45 degrees left and reading off the colors of elements from top to bottom, we obtain an element $w \in W$ such that $H=H(w)$.

    Then mapping $\nu \mapsto w \lambda$ provides the second bijection.
\end{proof}

By symmetry, an analogous result holds for $\qv(\dv,\delta_{n-1})$. Type $D$ lattice paths were also studied in \cite{BBCMacdonald,BRPfaffian}, where they arise in the study of Pfaffian Schur processes and half-space Macdonald processes.

Fix $\nu$ and denote by $\dv$ the corresponding dimension vector. Let $E^{\rightarrow}$ (resp. $E^{\downarrow}$) denote the set of edges of $\nu$ which proceed right (resp. down).

In \cite{BBCMacdonald}, a factorization formula was proven for the partition function of Macdonald processes. This partition function was shown in \cite{dinkjang} to coincide with the vertex function for $\qv(\dv,\delta_{n})$. Combining these two results, we deduce the following, which is again a special case of the Peterson-Proctor hook-product formula.

\begin{proposition}[\cite{BBCMacdonald} Proposition 2.2, \cite{dinkjang} Formula (24)]\label{prop: typeDvertex}
  There is an equality of rational functions of $z$:
    $$
    \signver{1}|_{q=2\hbar}=\left(\prod_{\substack{e<e' \\ e' \in E^{\downarrow}(\nu) \\ e \in E^{\rightarrow}(\nu) }} \frac{1}{1-\frac{x_{e}}{ x_{e'}}} \right) \left(\prod_{\substack{e<e' \\ e, e' \in E^{\rightarrow}(\nu) \cup \{\nu_{n}\}}} \frac{1}{1-x_{e}x_{e'}} \right)
    $$
    where the variables are related to $z_{i}$ by $z_{i}=x_{\nu_{i}}/x_{\nu_{i+1}}$ for $i\leq n-1$ and $z_{n}=x_{\nu_{n-1}} x_{\nu_{n}}$ (i.e. $x_{\nu_{i}}=e^{-\epsilon_{i}}$ in the usual description of the $D_{n}$ root system).
   \end{proposition}

   The factorization of Proposition \ref{prop: typeDvertex} is proven in a way similar to that of Proposition \ref{prop: typeAvertex}. Namely, there exists an expression analogous to that of Proposition \ref{prop: A schur vertex} in terms of Schur polynomials; see \cite[Equation (24)]{dinkjang}. Then one sums the resulting expression using Cauchy identities. Since our proof of Theorem \ref{thm: A rationality} proceeds by applying the appropriately conjugated rescaled differential operators $(q/2)^rD^r$ to insert (modified) Chern class descendants after subtracting the lowest heap weight at each color, we can run the exact same argument to prove the following.

   \begin{theorem}\label{thm: D rationality}
       For any $\tau$, $\signver{\tau}|_{q=2\hbar}$ is a rational function of $z$ with poles only at $z^{\alpha}=1$ for some root $\alpha$.
   \end{theorem}

The proof for $\fundwt_{n-1}$ is analogous.

\subsection{Type \texorpdfstring{$D_{n}$}{Dn}, minuscule weight \texorpdfstring{$\fundwt_{1}$}{w1}}

\begin{figure}[t]
    \centering
\begin{tikzpicture}[circ/.style={shape=circle,draw,inner sep=1.5pt}]
    \node[circ, on grid] (11) at (0,0){};
    \node[circ, on grid, above  right= of 11](21){};
    \node[circ, on grid, above right = of 21](31){};
    \node[circ, on grid, above right = of 31](41){};
    \node[circ, on grid, blue, above right = of 41](51){};
    \node[circ, on grid, red, above left = of 41](32){};
    \node[circ, on grid, above left = of 51](42){};
    \node[circ, on grid, above left = of 42](33){};

    \node[circ, on grid, above left = of 33](22){};
    \node[circ, on grid, above left = of 22](12){};

\draw[->] (11) edge (21) (21) edge (31) (41) edge (51) (41) edge (32) (51) edge (42) (32) edge (42) (42) edge (33) (22) edge (12);
\draw (31) edge[dotted] (41) (33) edge[dotted] (22);

The Dynkin diagram underneath
\node[circ, on grid, below= of 11](1){2};
\node[circ, on grid, right= of 1](2){2};
\node[circ, on grid, right= of 2](3){2};
\node[circ, on grid, right= of 3](4){2};
\node[circ, on grid, blue,  below right= of 4](5){1};
\node[circ, on grid, red, above right= of 4](6){1};

 \draw (1) -- (2);
 \draw (2) -- (3);
 \draw (3) edge[dotted] (4);
\draw (4) -- (5);
\draw (4) -- (6);
\end{tikzpicture}
 \caption{The largest possible quiver representation of type $D_{n}$ with framing at the leftmost vertex. The red and blue vertices in the graph match the red and blue vertices in the Dynkin diagram below.}
    \label{Dnrep}
\end{figure}
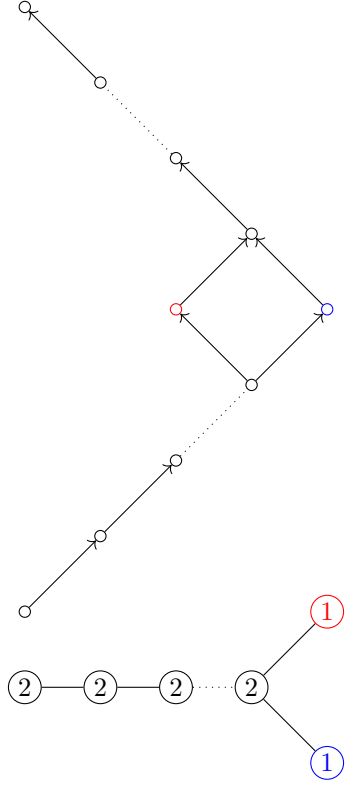

Consider next type $D_{n}$ and minuscule weight $\fundwt_{1}$. There are $2n$ possible $\dv$ such that $\qv(\dv,\delta_{1})$ is nonempty. For a fixed $n$, the largest possible $\dv$ is $\dv_{i}=2$ for $i \leq n-2$ and $\dv_{n-1}=\dv_{n}=1$. The corresponding heap is drawn in Figure \ref{Dnrep}. The complete set of $2n$ quiver representations is obtained by taking order filters of this heap.

For these vertex functions, an expression in terms of symmetric functions analogous to Proposition \ref{prop: A schur vertex} was given in \cite[Lemma 8.2 and Lemma 8.5]{dinkjang}. As above, $\signver{1}|_{q=2\hbar}$ can be explicitly calculated to be a rational function using Cauchy identities. Then arbitrary descendants can be inserted using the appropriately conjugated rescaled differential operators $(q/2)^rD^r$ from \eqref{eq: schur diff op}, after subtracting the lowest heap weight at each color. Thus we again deduce the same rationality result.

\begin{theorem}
    For any $\tau$, $\signver{\tau}|_{q=2\hbar}$ is a rational function of $z$ with poles only at $z^{\alpha}=1$ for some root $\alpha$.
\end{theorem}

\subsection{Deformations}\label{ssec:deformation operators}
To be clear, all of the results discussed here admit deformations outside the $q=2\hbar$ specialization, see \cite{dinkjang, dinksmir, dinksmir3}. Namely, vertex functions can be written as sums of products of Jack polynomials, and descendants can be inserted by applying Sekiguchi operators, which deform the differential operators of \eqref{eq: schur diff op}. This also applies to the $K$-theoretic setting, but with Macdonald polynomials and Macdonald operators replacing Jack polynomials and Sekiguchi operators. In both of these deformed settings, we obtain a similar constraint on the possible locations of poles of the descendant vertex functions.

\subsection{Type \texorpdfstring{$E_6$}{E6} and \texorpdfstring{$E_7$}{E7}}

In types $E_6$ and $E_7$, the heaps can again be described using explicit combinatorics. The largest heap for $E_6$, $\fundwt_{1}$ (resp. $E_7$, $\fundwt_{1}$) is shown in Figure \ref{E6rep} (resp. Figure \ref{E7rep}).

As before, Theorem \ref{thm: pointvertex} can be used to compute the vertex function. We do not know how to express the vertex function in terms of Schur polynomials, though we expect such a formula to exist.

\begin{remark}
    Given the interpretation of the partition functions for Schur processes and Pfaffian Schur processes provided by vertex functions of point quiver varieties \cite{BRPfaffian,BBCMacdonald, ORSchur}, it seems natural to define more general probabilistic processes corresponding to a simple Lie algebra and a minuscule fundamental weight using vertex functions. Even more generally, one could contemplate defining such processes for the data $(\Gamma,\lambda,w)$ where $\Gamma$ is a quiver without loops and $w$ is $\lambda$-minuscule. 
\end{remark}

\begin{figure}[h]
    \centering
\begin{tikzpicture}[circ/.style={shape=circle,draw,inner sep=1.5pt}]
    \node[circ, on grid] (11) at (0,0){};
    \node[circ, on grid, above  right= of 11](21){};
    \node[circ, on grid, above right = of 21](31){};
    \node[circ, on grid, above right = of 31](41){};
    \node[circ, on grid, above right = of 41](51){};
    \node[circ, on grid, red, above left = of 31](61){};
    \node[circ, on grid, above left = of 41](32){};
    \node[circ, on grid, above left = of 51](42){};
    \node[circ, on grid, above left = of 32](22){};
    \node[circ, on grid, above left = of 42](33){};
    \node[circ, on grid, above left = of 22](12){};
    \node[circ, on grid, above left = of 33](23){};
    \node[circ, on grid, above right = of 23](34){};
    \node[circ, on grid, red, above right = of 33](62){};
    \node[circ, on grid, above right = of 34](43){};
    \node[circ, on grid, above right = of 43](52){};

\draw[->] (11) edge (21) (21) edge (31) (31) edge (41) (41) edge (51) (31) edge (61) (61) edge (32) (41) edge (32) (32) edge (42) (32) edge (22) (22) edge (12) (22) edge (33) (12) edge (23) (33) edge (23) (51) edge (42) (42) edge (33) (33) edge (62) (62) edge (34) (23) edge (34) (34) edge (43) (43) edge (52);
    
\node[circ, on grid, below= of 11](1){2};
\node[circ, on grid, right= of 1](2){3};
\node[circ, on grid, right= of 2](3){4};
\node[circ, on grid, right= of 3](4){3};
\node[circ, on grid, right= of 4](5){2};
\node[circ, on grid, red, above= of 3](6){2};

\draw (1) -- (2);
\draw (2) -- (3);
\draw (3) -- (4);
\draw (4) -- (5);
\draw (3) -- (6);
\end{tikzpicture}
 \caption{The quiver representation corresponding to the lowest weight vector of the minuscule representation of highest weight $\fundwt_{1}$ of type $E_{6}$. The red nodes signify basis vectors belonging to $V_{6}$.}
    \label{E6rep}
\end{figure}

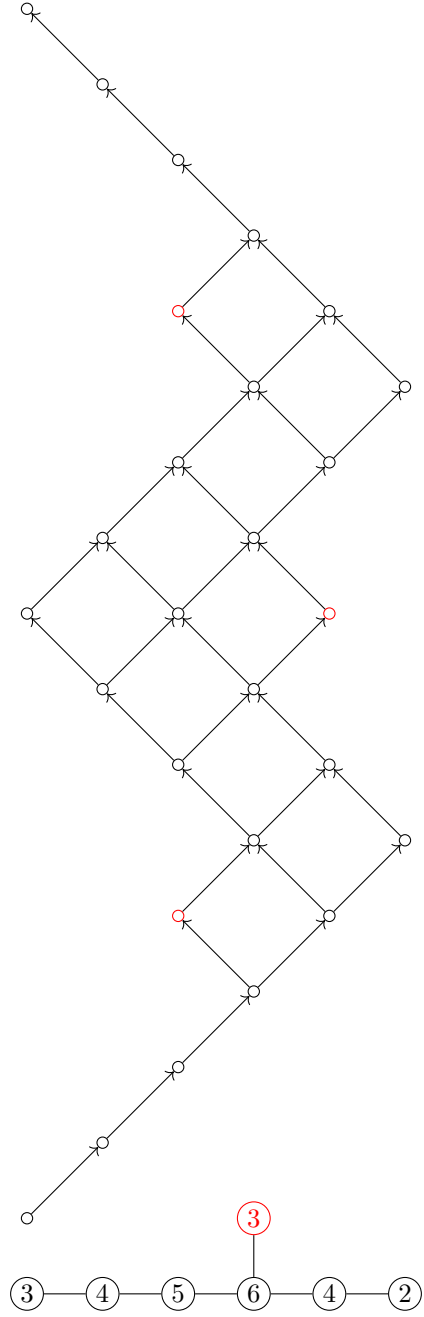
\begin{figure}[h]
    \centering
\begin{tikzpicture}[circ/.style={shape=circle,draw,inner sep=1.5pt}]
    \node[circ, on grid] (11) at (0,0){};
    \node[circ, on grid, above  right= of 11](21){};
    \node[circ, on grid, above  right = of 21](31){};
    \node[circ, on grid, above  right = of 31](41){};
    \node[circ, on grid, above  right = of 41](51){};
    \node[circ, on grid, above  right = of 51](61){};
    
    \node[circ, on grid, above  left = of 51](42){};
    \node[circ, on grid, above  right = of 42](52){};
    \node[circ, on grid, red, above  left = of 41](71){};
    \node[circ, on grid, above  left = of 42](32){};
    \node[circ, on grid, above  left = of 52](43){};
    \node[circ, on grid,red, above  right = of 43](72){};
    \node[circ, on grid, above  left = of 32](22){};
    \node[circ, on grid, above  left = of 22](12){};
    \node[circ, on grid, above  right = of 22](33){};
    \node[circ, on grid, above  right = of 12](23){};
    \node[circ, on grid, above  right = of 33](44){};
    \node[circ, on grid, above  right = of 23](34){};
    \node[circ, on grid, above  right = of 44](53){};
    \node[circ, on grid, above  right = of 34](45){};
    \node[circ, on grid, above  right = of 45](54){};
    \node[circ, on grid, above  right = of 53](62){};
    \node[circ, on grid, above  left = of 54](46){};
     \node[circ, on grid, red,above  left = of 45](73){};
      \node[circ, on grid, above  left = of 46](35){};
       \node[circ, on grid, above  left = of 35](24){};
        \node[circ, on grid, above  left = of 24](13){};

        \draw[->] (11) edge (21) (21) edge (31) (31) edge (41) (41) edge (51) (51) edge (61) (41) edge (71) (71) edge (42) (51) edge (42) (42) edge (52) (61) edge (52) (42) edge (32) (32) edge (22) (22) edge (12) (32) edge (43) (52) edge (43) (43) edge (33) (22) edge (33) (43) edge (72) (72) edge (44) (33) edge (44) (12) edge (23) (33) edge (23) (23) edge (34) (44) edge (34) (44) edge (53) (34) edge (45) (53) edge (45) (53) edge (62) (62) edge (54) (45) edge (54) (45) edge (73) (73) edge (46) (54) edge (46) (46) edge (35) (35) edge (24) (24) edge (13);

\node[circ, on grid, below= of 11](1){3};
\node[circ, on grid, right= of 1](2){4};
\node[circ, on grid, right= of 2](3){5};
\node[circ, on grid, right= of 3](4){6};
\node[circ, on grid, right= of 4](5){4};
\node[circ, on grid, right= of 5](6){2};
\node[circ, on grid, red,above= of 4](7){3};

\draw (1) -- (2);
\draw (2) -- (3);
\draw (3) -- (4);
\draw (4) -- (5);
\draw (5) -- (6);
\draw (4) -- (7);
\end{tikzpicture}
 \caption{The quiver representation corresponding to the lowest weight vector of the minuscule representation of type $E_{7}$. The red nodes signify basis vectors belonging to $V_{7}$.}
    \label{E7rep}
\end{figure}

\clearpage

\printbibliography

\end{document}